\documentclass[a4paper,11pt]{article}

\usepackage{ucs}
\usepackage[utf8]{inputenc}

\usepackage{graphicx}
\usepackage{amsfonts}
\usepackage{dsfont}
\usepackage{amssymb}
\usepackage{amsmath}
\usepackage{amsthm}
\usepackage{enumerate}
\usepackage{stmaryrd}
\usepackage{fullpage}
\usepackage{ifthen}
\usepackage{subfigure}
\usepackage{epic}
\usepackage{authblk}
\usepackage{textcomp}
\usepackage{mathrsfs}
\usepackage{bm}
\usepackage[small]{caption}
\usepackage{tikz,tkz-tab}
\usetikzlibrary{matrix}
\usepackage{pgfplots}
\pgfplotsset{compat=newest} 
\pgfplotsset{plot coordinates/math parser=false}

\usepackage[hypertexnames=false,colorlinks=true,linkcolor=blue,citecolor=blue]{hyperref}
\usepackage[numbers,comma,square,sort&compress]{natbib}
\usepackage[letterpaper,text={6.5in,9in},centering]{geometry}

\usepackage{color}
\usepackage{titlesec}

\graphicspath{{eps/}{pdf/}{png/}}

\def\varep{\varepsilon}

\newcommand{\bqq}{\begin{equation}}
\newcommand{\eqq}{\end{equation}}
\newcommand{\bqs}{\begin{equation*}}
\newcommand{\eqs}{\end{equation*}}

\newcommand{\R}{\mathbb{R}} 
\newcommand{\N}{\mathbb{N}} 
\newcommand{\I}{\mathcal{I}}

\renewcommand{\H}{\mathcal{H}}
\newcommand{\F}{\mathcal{F}}
\renewcommand{\S}{\mathcal{S}}
\renewcommand{\SS}{\mathbb{S}}

\newcommand{\K}{\mathcal{K}}
\newcommand{\X}{\mathcal{X}}

\newcommand{\T}{\mathcal{T}}

\newcommand{\md}{\mathrm{d}}

\newtheorem{lem}{Lemma}[section]
\newtheorem{thm}{Theorem}[section]
\newtheorem{prop}{Proposition}[section]

\newtheorem{rmk}[lem]{Remark}
\newtheorem{defi}[lem]{Definition}

 {\begin{trivlist}\item[]\textbf{Proof#1 }}%
 {\hspace*{\fill}$\rule{0.3\baselineskip}{0.35\baselineskip}$\end{trivlist}}

\newenvironment{Hypothesis}[1]%
  {\begin{trivlist}\item[]{\bf Hypothesis #1 }\em}{\end{trivlist}}

\numberwithin{equation}{section}

\title{Spreading properties in Kermack - McKendrick models with nonlocal spatial interactions\\ -- A new look --}

\author[]{Gr\'egory Faye, Jean-Michel Roquejoffre\footnote{Corresponding author: \texttt{jean-michel.roquejoffre@math.univ-toulouse.fr}}~ \& Mingmin Zhang}
\affil[]{\small CNRS, UMR 5219, Institut de Math\'ematiques de Toulouse, 31062 Toulouse Cedex, France}

\date{{\it To the memory of H. Brezis, with our deepest respect.}}
\begin{document}
\maketitle

\begin{abstract}
In this paper, we revisit the famous Kermack - McKendrick model with nonlocal spatial interactions by shedding new light on associated spreading properties and we also prove the existence and uniqueness of traveling fronts. Unlike previous studies that have focused on integrated versions of the variable representing the susceptible population, we analyze the long time dynamics of the underlying age-structured model for the cumulative density of infected individuals and derive precise asymptotic estimates for the infected population. Our approach consists in studying the long time dynamics of an associated  transport equation with nonlocal spatial interactions whose spreading properties are close to those of classical Fisher-KPP reaction-diffusion equations. Our study is self-contained and relies on comparison  arguments.  
\end{abstract}

\bigskip

\section{Introduction}

We consider a classical Kermack - McKendrick model \cite{KMcK,Diekmann1,AR1976} of the form
\bqq
\left\{
\begin{split}
\partial_t \S(t,x) &= - \left( \int_0^\infty \tau(i) \K*I(t,i,x)\md i \right) \S(t,x), \quad t>0, \quad x\in\R^d,\\
\partial_t I(t,i,x) + \partial_iI(t,i,x) &= - \gamma(i) I(t,i,x) , \quad t>0,  \quad i>0,  \quad x\in\R^d, \\
I(t,0,x)&= \left( \int_0^\infty \tau(i) \K*I(t,i,x)\md i \right) \S(t,x), \quad t>0,  \quad x\in\R^d,
\end{split}
\right.
\label{EDP}
\eqq 
where $\S(t,x)$ represents the susceptible population at time $t>0$ structured by a spatial variable $x \in \R^d$ for some $d\geq1$, and $I(t,i,x)$ represents the population of infected individuals at time $t>0$ with age of infection $i>0$ also structured in space. For future reference, we let $\I(t,x)$ be the total population of infected individuals such that we have
\bqs
\I(t,x)=\int_0^\infty I(t,i,x)\md i.
\eqs
Here we have set $\K*I(t,i,x):=\int_{\R^d} \K(x-y)I(t,i,y)\md y$. The second equation of the model can be understood as follows. The number of infected individuals at time $t+h$, spatial position $x$, with an age of infection $i$, is given by the difference of the incoming flux of infected individuals at time $t$, spatial position $x$, with an age of infection $i-h$, and the outgoing flux of infected individuals at time $t$, spatial position $x$, with an age of infection $i$, which are removed at rate $h \gamma(i)$. Namely, we have assumed the balance equation
\bqs
I(t+h,i,x)=I(t,i-h,x)-h\gamma(i) I(t,i,x)+o(h),
\eqs
from which we deduce the second equation of \eqref{EDP} by substracting $I(t,i,x)$ on both side and letting $h\rightarrow0$. The boundary condition at $i=0$ is also very natural and corresponds to the hypothesis that newly infected individuals are proportional  to the fraction of the susceptible individuals that have been in contact with infectious individuals.  It is important to note that in the above model represented by \eqref{EDP}, we have made several simplifying assumptions that we now list and comment. First, we have ignored all demographic effects (natural birth/death of the population) which can be neglected at first approximation. We have set the model on the full infinite space $\R^d$ since we are interested in long range spatial spreading properties of the epidemic, and we have mainly the cases $d=1$ and $d=2$ in mind although our results will hold for all $d\geq1$. We have also supposed that the rate of infection between the susceptible and infected populations can be composed as $\tau(i)\K(x-y)$ where $\tau(i)$ represents the infection rate within a homogeneous population with an age of infection $i\ge 0$ while $\K(x-y)$ describes the probability density of individuals exerting a force of infection from position $y$ to position $x$. And throughout, we will always assume that $\int_{\R^d}\K(x)\md x=1$ and $\K>0$ in $\R^d$. We assume that the habitat is isotropic (invariant under rotation), therefore the connectivity kernel $\K$ is radially symmetric and thus $\K(x-y)$ only depends on the distance $\|x-y\|$, where $\|\cdot\|$ stands for the Euclidean norm on $\R^d$. Let us finally note that a more general rate of infection could be written. For example, instead of $\tau(i)\K(x-y)$, one could introduce rates of infection of the form $\K(i,x-y)$ where the decoupling between the spatial variable and the age of infection is no longer possible. Our analysis would carry naturally over to such general cases, but here we rather prefer to stick with the decoupled case to gain in readability.

We supplement the model represented by \eqref{EDP} by a set of initial conditions which takes the form
\bqs
\left\{
\begin{split}
\S(t=0,x) &= \S_0, \quad x\in\R^d,\\
I(t=0,i,x) & = I_0(i,x),  \quad x\in\R^d, \quad i>0, 
\end{split}
\right.
\eqs
with $I_0$ bounded, nonnegative and compactly supported. Let us remark that assuming a homogeneous distribution $\S_0$ across the susceptible population is debatable from a biological point of view as in practical situations this distribution is most likely to be heterogeneous. Here, following previous works \cite{BRR21,BF21}, we adopt this formalism since it will allow us to carry out a fairly complete mathematical analysis with somehow closed-form formulas which are relatively simple to interpret. 

Next, we introduce the cumulative number of infected individuals with elapsed time since infection $i>0$ and at position $x\in\R^d$ by setting
\bqs
\rho(t,i,x):=\int_0^t I(s,i,x)\md s.
\eqs
We note that the first equation in \eqref{EDP} can be integrated and expressed more simply as
\bqs
\frac{\S(t,x)}{\S_0}=\exp\left(-\int_0^\infty \tau(i) \K*\rho(t,i,x)\md i \right),
\eqs
such that one can reduce system \eqref{EDP} to a single equation on the evolution of $\rho$ given by
\bqq \label{edpI}\left\{
\begin{split}
\partial_t \rho(t,i,x) + \partial_i\rho(t,i,x) &=I_0(i,x) - \gamma(i) \rho(t,i,x) ,\quad t>0, \quad i>0,  \quad x\in\R^d,  \\
\rho(t,0,x) &=\S_0 \left(1-\exp\left(-\int_0^\infty \tau(i) \K*\rho(t,i,x)\md i \right)\right), \quad  t> 0,  \quad x\in\R^d,\\
\end{split}\right.
\eqq
eventually complemented with the initial condition $\rho(0,i,x)=0$ for all $i\ge 0$ and $x\in\R^d$.

At this stage of the presentation, we claim that \eqref{edpI} shares lots of  features with traditional Fisher-KPP equations encountered in the reaction-diffusion community \cite{KPP37,fisher,AW2}. To make our point clearer, let us introduce the homogenous counterpart of \eqref{edpI}, that is, we set $I_0(i,x)=0$ and consider 
\begin{equation}
\label{kppI}\left\{
		\begin{split}
\partial_t \rho(t,i,x) + \partial_i\rho(t,i,x)& = - \gamma(i) \rho(t,i,x),\quad t>0,\quad i>0,  \quad x\in\R^d, \\
\rho(t,0,x)&=\S_0 \left(1-\exp\left(-\int_0^\infty \tau(i) \K*\rho(t,i,x)\md i \right)\right), \quad t> 0,  \quad x\in\R^d,\\
\end{split}\right.
\end{equation}
where we allow for general initial condition $\rho(0,i,x)=\rho_0(i,x)$ compactly supported in $[0,+\infty)\times\R^d.$

The link to the usual Fisher-KPP equations becomes evident when one assumes that both $\tau(i)=\tau>0$ and $\gamma(i)=\gamma$ are independent of the infection age $i$. In that case, the total cumulative density of infected individuals defined as $\mathcal{C}(t,x)=\int_0^\infty \rho(t,i,x)\md i=\int_0^t \I(s,x)\md s$
satisfies the following nonlocal reaction-diffusion equation
\bqs
\partial_t \mathcal{C}(t,x)=\S_0 \left(1-\exp\left(- \tau \K*\mathcal{C}(t,x) \right)\right)-\gamma \mathcal{C}(t,x), \quad t>0, \quad x\in\R^d.
\eqs
Note that the above equation can be slightly reformulated as
\begin{align*}
\partial_t \mathcal{C}(t,x)=\underbrace{\S_0\tau(-\mathcal{C}(t,x)+\K*\mathcal{C}(t,x))}_{\text{nonlocal diffusion}}&+\underbrace{(\S_0\tau-\gamma)\mathcal{C}(t,x)}_{\text{linear reaction}}\\
& -\underbrace{\S_0 \left(\exp\left(- \tau \K*\mathcal{C}(t,x) \right)-1+ \tau \K*\mathcal{C}(t,x)\right)}_{\text{nonlinear nonlocal reaction}}.
\end{align*}
Such models are very close to the spatially extended SIR models studied recently in the literature, see for example \cite{kallen,DG14,CGH17,BRR21,BF21,ducasse,BNR,peng} and references therein.

Let us finally remark that our transport model with nonlocal interactions is rather different from the setting considered for example in \cite{DM09,DM11,DMR}. Indeed, in the models presented and studied in \cite{DM09,DM11,DMR} and subsequent works, the interactions between individuals are purely local and only driven by diffusion in space which could be interpreted as spatial migration from a population dynamics perspective. Here we focus on the case where the interactions are nonlocal in space induced by some spatial connectivity kernel $\K$ which encodes the spatial range of interactions of infected individuals.  

\subsection{Assumptions and main results}

We first present the main assumptions on the parameters that shall stand throughout the paper, for which we shall be guided by the
biological interpretation and by the desire to let the analysis proceed along standard
lines, covering as much  generality (possibilities) as possible. Our first set of assumptions are on the recovery rate function $\gamma$ of infected individuals and the transmission rate function $\tau$.
\begin{Hypothesis}{(H1) - Recovery and transmission rates.}
We assume the following:
\begin{itemize}
\item[(i)] $\gamma:[0,i_\dagger)\mapsto\R_+$,  with some $i_\dagger\in(0,+\infty]$, is nonnegative such that $\gamma(i)\in L^1_{loc}([0,i_\dagger))$.  In particular, when $i_\dagger\in(0,+\infty)$, we assume that 
\bqq\label{pi}
\pi(i):=e^{-\int_0^i\gamma(s)\md s}\longrightarrow 0, \text{ as }  i\to i_\dagger.
\eqq
\item[(ii)] $\tau:[0,i_\dagger)\mapsto\R_+$ is nonnegative, bounded and absolutely continuous on $[0,i_\dagger)$, with $0\leq\tau(i)\leq \tau_\infty$ for $i\in[0,i_\dagger)$ and some $\tau_\infty>0$. 
\item[(iii)] $\omega(i):=\tau(i)\pi(i) \in L^1([0,i_\dagger))$.
\item[(iv)] When $i_\dagger<+\infty$, the functions $\tau$ and $\omega$ are both extended by $0$ outside the interval $[0,i_\dagger)$.
\end{itemize}
\end{Hypothesis}
The quantity $\pi(i)$ is the probability for an individual to stay in the class of infected individuals after time $i\ge 0$, as a consequence, our assumption \eqref{pi} simply reflects the fact after the maximal age of infection $i_\dagger$ there should be no more infected individuals. Moreover, \eqref{pi} implies that necessarily  $\gamma(i)\to +\infty$ as $i\to i_\dagger$ when $i_\dagger<+\infty$.  We remark that our assumptions on $\gamma$ and $\tau$ above are rather generic which encompass many biologically relevant situations. We refer to \cite{Thieme3,MMW2010,Richardetal,Fouteletal,IM17} for concrete examples. Now, regarding the interaction kernel $\K$, we make the following assumptions.

\begin{Hypothesis}{(H2) - Interaction kernel.}
We assume that $\K\in W^{1,1}(\R^d)$ is positive everywhere, radially symmetric and normalized such that $\int_{\R^d}\K(x)\md x=1$.
\end{Hypothesis}
The above assumptions on $\K$ are very natural. The fact that we impose $\K$ to be radially symmetric indicates the fact that we consider an environment which is spatially isotropic. We refer to \cite{Weinberger,VdBMD} for results in the anisotropic case. The regularity assumption that $\K\in W^{1,1}(\R^d)$, which implies that each $\partial_{x_p} \K \in L^1(\R^d)$ ($p=1,\cdots,d$), is a technical assumption that allows us to gain some regularity in space for the solution in the case where $I_0\equiv0$. Finally, the assumption that $\K(x)>0$ for all $x\in\R^d$ may be seen as a strong assumption since it implies an all to all coupling among the infected population of individuals. Nevertheless, since $\K\in L^1(\R^d)$ we necessarily have $\K(x)\rightarrow0$ as $\|x\|\rightarrow+\infty$, and thus the probability of interactions between separated individuals decreases to $0$ as a function of their relative distance. Actually, we will require stronger localization assumptions on the kernel when dealing with the asymptotic properties of the solutions to \eqref{edp} (see Hypothesis {\bf(H2$\mu$)} below) which will precisely quantify the decay rate of the interactions as $\|x\|\rightarrow+\infty$. Relaxing the positivity condition of $\K$ in order for example to take into account the case of compactly supported kernels would necessarily be at the expense of having stronger assumptions on the initial density of infected individuals $I_0$ but also on the recovery and transmission rates. However, one expects to observe similar spreading properties as the ones presented here when compactly supported kernels are considered. Finally, using the radial symmetry of $\K$, we note that there exists $\K_0\in W^{1,1}(\R)$ with $\K_0(z)=\K_0(-z)>0$ for $z\in\R$, such that the following equality holds
 \bqq\label{1.4'}
 \K_0(z):=\int_{\R^{d-1}}\K(z,x_2,\cdots,x_d)\md x_2\cdots \md x_d, \quad z\in\R.
 \eqq

The main goal of this paper is to investigate the long time behavior of \eqref{edpI} starting from certain nonnegative initial condition $\rho_0$, namely,
\bqq \label{edp}\left\{
\begin{split}
\partial_t \rho(t,i,x) + \partial_i\rho(t,i,x) &\overset{(i)}{=}I_0(i,x) - \gamma(i) \rho(t,i,x) ,\quad t>0, \quad i\in(0,i_\dagger),  \quad x\in\R^d,  \\
\rho(t,0,x) &\overset{(ii)}{=}\S_0 \left(1-\exp\left(-\int_0^\infty \tau(i) \K*\rho(t,i,x)\md i \right)\right), \quad  t> 0,  \quad x\in\R^d,\\
\rho(0,i,x)&\overset{(iii)}{=}\rho_0(i,x),   \quad     i\in [0,i_\dagger), \quad x\in\R^d.
\end{split}\right.
\eqq
It is worth  mentioning that  we allow ourselves to consider more general initial conditions than the zero initial condition which naturally arises from our change of unknown $\rho(t,i,x)=\int_0^t I(s,i,x)\md s$ leading to $\rho(0,i,x)=0$ at initial time.  Furthermore, in the case when the maximal age of infection is finite (i.e. when $i_\dagger<\infty$), the integral appearing in \eqref{edp}(ii) has to be understood as
\bqs
\int_0^\infty \tau(i) \K*\rho(t,i,x)\md i = \int_0^{i_\dagger} \tau(i) \K*\rho(t,i,x)\md i.
\eqs
Throughout this paper, we shall work with the following notion of solutions for the initial boundary value problem~\eqref{edp}.
\begin{defi}\label{def:classical} 	We say that a function $\rho: \R_+\times[0,i_\dagger)\times\R^d\mapsto \R$  is a global
	 classical solution of  \eqref{edp} in $\R_+\times[0,i_\dagger)\times\R^d$ with initial condition $\rho_0$ and source term $I_0$ defined on $[0,i_\dagger)\times\R^d$, if $\rho$ is continuous on $\left\{(t,i,x)\in\R_+\times[0,i_\dagger)\times\R^d~|~t\neq i\right\}$, if $\rho(0,i,x)=\rho_0(i,x)$ for all  $(i,x)\in[0,i_\dagger)\times\R^d$, if $\partial_t \rho$ and $\partial_i\rho$ exist a.e. in  $(0,+\infty)\times(0,i_\dagger)\times\R^d$ and $\rho$ satisfies \eqref{edp}(i) a.e. in $(0,+\infty)\times(0,i_\dagger)\times\R^d$, and if $\rho$ satisfies \eqref{edp}(ii) for each $t>0$ and $x\in\R^d$.
\end{defi}

With assumptions {\bf (H1)-(H2)}, we can prove the following well-posedness result for \eqref{edp}.

\begin{prop}[Well-posedness]\label{prop-edpI-well-posedness}
Assume Hypotheses {\bf (H1)-(H2)}, and that $I_0$ is nonnegative, bounded, continuous and compactly supported on $[0,i_\dagger)\times\R^d$, and that $\rho_0$ is nonnegative such that $\rho_0/\pi$ is bounded and absolutely continuous on $[0,i_\dagger)\times \R^d$, then problem \eqref{edp} admits a unique nonnegative classical solution $\rho$  in $\R_+\times[0,i_\dagger)\times\R^d$ with initial condition $\rho_0$, in the sense of Definition~\ref{def:classical}, which satisfies $\rho/\pi\in L^\infty(\R_+\times[0,i_\dagger)\times \R^d)$. Furthermore, if $I_0\equiv0$, then $\partial_{x_p}\rho$ ($p=1,\cdots,d$) exist a.e. on $(0,+\infty)\times(0,i_\dagger)\times\R^d$.
\end{prop}

As it will be shown in Section~\ref{prem-results},   the proof of Proposition \ref{prop-edpI-well-posedness} shows that if one further assumes that the initial condition  $\rho_0>0$ on $[0,i_\dagger)\times \R^d$, then the  solution $\rho$ to problem \eqref{edp} is positive in $\R_+\times[0,i_\dagger)\times\R^d$. However, the situation of $\rho_0\geq0$ on $[0,i_\dagger)\times \R^d$, typically when $\rho_0$ is compactly supported or when $\rho_0\equiv0$, is less clear.  In the sequel, we show that it is possible to obtain the eventual positivity of the solution $\rho$ to problem \eqref{edp} when either $I_0$ or $\rho_0$ is nontrivial, by imposing further conditions on their supports. For convenience, when $I_0\not\equiv0$ in $[0,i_\dagger)\times \R^d$, let us denote the support of $I_0$ in $i$ variable by
\begin{equation}
	\label{I_0 support}
	\mathcal{D}_{I_0}:=\overline{\left\{ i\geq 0 ~|~ I_0(i,x)\neq0 \text{ for some }x\in\R^d\right\}}\subset[0,i_\dagger).
\end{equation}

\begin{prop}[Positivity]\label{prop-positivity}
Under the assumptions of Proposition~\ref{prop-edpI-well-posedness}, let $\rho$ be the solution  of  problem \eqref{edp} given by Proposition~\ref{prop-edpI-well-posedness}. We have:
\begin{itemize}
\item[(i)] If $I_0\not\equiv0$, we assume that $\mathrm{Int}(\mathrm{supp}(\tau))\cap \mathrm{Int}(\mathcal{D}_{I_0})\neq \emptyset$, and define
\bqq\label{i_star}
i_\star:=\min\left(\overline{\mathrm{Int}(\mathrm{supp}(\tau))\cap \mathrm{Int}(\mathcal{D}_{I_0})}\right)\in[0,i_\dagger),
\eqq
then $\rho(t,i,x)>0$ for $(t,i,x)\in(0,+\infty)\times[0,i_\dagger)\times \R^d$ with $t>i+i_\star$. 
 
\item [(ii)] If $\rho_0\not\equiv0$, we assume that there exist $0<\varpi\in\mathrm{Int}\big(\mathrm{supp}(\tau)\big)$ and $x_0\in\R^d$  such that
\begin{equation}
\label{H3}
[0,\varpi]\times\left\{x_0\right\} \subset \mathrm{supp}(\rho_0),
\end{equation} 
then $\rho(t,i,x)>0$ for $(t,i,x)\in(0,+\infty)\times[0,i_\dagger)\times \R^d$ with $t>i$.
\end{itemize}
\end{prop}

\begin{rmk}
	Under compactly supported kernels and nontrivial $\rho_0$, the condition \eqref{H3} ensuring positivity of the solution for $t>i$ will be replaced by the existence of $\tau_0>0$ such that
	\bqq\label{eqrmkpos}
	[0,\tau_0]\subset \mathrm{supp}(\tau)\cap \left\{ i\in[0,i_\dagger)~|~ \rho_0(i,x)>0 \text{ for some } x\in\R^d\right\}.
	\eqq
	In the sequel, we focus on kernels satisfying  {\bf (H2)} and present the associated results, which also apply to the case of compactly supported kernels.  To avoid repetition, we do not provide a separate treatment of the latter case, and instead indicate the necessary modifications in the proofs.
\end{rmk}

The above proposition indicates that we always obtain the positivity of the solutions of \eqref{edp} for $t>i+i_\star$ when $I_0\not\equiv0$, upon assuming that the supports of $I_0$ and of $\tau$ have a common intersection. As it will be seen in the proof, the condition that $\mathrm{Int}(\mathrm{supp}(\tau))\cap \mathrm{Int}(\mathcal{D}_{I_0})\neq \emptyset$ is optimal in the case when $\rho_0\equiv0$. Indeed, if $\rho_0\equiv0$ and $\mathrm{Int}(\mathrm{supp}(\tau))\cap \mathrm{Int}(\mathcal{D}_{I_0})= \emptyset$, then the dynamics is trivial and the solution is simply given by
\bqq \label{1.6'}
\rho(t,i,x)=
\left\{
\begin{split}
\left(\int_{i-t}^i \frac{I_0(\xi,x)}{\pi(\xi)} \md \xi\right)\pi(i),& \quad i\ge t,\\
\left(\int_0^i \frac{I_0(\xi,x)}{\pi(\xi)} \md \xi\right)\pi(i),& \quad i< t.
\end{split}
\right.\eqq
In the case where $I_0\equiv0$ and $\rho_0\not\equiv0$ is compactly supported in $[0,i_\dagger)\times\R^d$, it is possible to extend the region of positivity to those values of $t>i$ at the expense of imposing the above extra condition \eqref{H3} on the support of the initial condition $\rho_0$. This is somehow an optimal result, since  the solution of  problem~\eqref{edp} for $t\leq i$ is essentially the initial condition transported along the characteristics of \eqref{edp}, and thus no uniform positivity result can be obtained.

Next, we introduce the definition of super- and subsolutions and prove a comparison principle for problem \eqref{edp}.
\begin{defi}\label{def_super sub} We say that a nonnegative, bounded and continuous function $\overline \rho :\R_+\times[0,i_\dagger)\times\R^d\to \R$  is a
	supersolution of problem \eqref{edp} in $\R_+\times[0,i_\dagger)\times\R^d$, with an initial condition $\overline\rho_0$ and source term $\overline I_0$ defined on $[0,i_\dagger)\times\R^d$, if $\overline\rho(0,x,i)=\overline\rho_0(i,x)$ for all  $(i,x)\in[0,i_\dagger)\times\R^d$, if $\partial_t \overline\rho$ and $\partial_i\overline\rho$ exist a.e. in  $(0,+\infty)\times(0,i_\dagger)\times\R^d$ and $\overline\rho$ satisfies \eqref{edp}(i) a.e. in $(0,+\infty)\times(0,i_\dagger)\times\R^d$ with the ``='' replaced by ``$\ge$'', and if moreover $\overline\rho$ satisfies \eqref{edp}(ii) for each $t>0$ and $x\in\R^d$ with the ``='' replaced by ``$\ge$''. A subsolution $\underline \rho$ can be defined in a similar way with both the inequality signs above being reversed.
\end{defi}

\begin{prop}[Comparison principle]
	\label{prop-cp}
Assume Hypotheses {\bf (H1)-(H2)}. Let $\overline \rho$ and $\underline \rho$ be respectively a super- and a  subsolution of \eqref{edp} in $\R_+\times[0,i_\dagger)\times\R^d$ in the sense of Definition \ref{def_super sub} associated with nonnegative initial data $\overline \rho_0$ and $\underline \rho_0$ defined on $[0,i_\dagger)\times\R^d$ satisfying $\overline \rho_0/\pi,\underline \rho_0/\pi\in L^\infty([0,i_\dagger)\times\R^d)$ and with nonnegative, bounded, continuous and compactly supported source terms $\overline I_0$ and $\underline I_0$ defined on $[0,i_\dagger)\times\R^d$. Assume that $\overline \rho_0\ge \underline \rho_0$ and $\overline I_0\ge \underline I_0$ in $[0,i_\dagger)\times\R^d$, then $\overline \rho\ge \underline \rho$ in $\R_+\times[0,i_\dagger)\times\R^d$. 
Furthermore,
\begin{itemize}
	\item[(i)]  if $\overline I_0\neq \underline I_0$ in $[0,i_\dagger)\times\R^d$, by further assuming that $\mathrm{Int}(\mathrm{supp}(\tau))\cap \mathrm{Int}(\mathcal{D}_{\overline I_0-\underline I_0})\neq \emptyset$, we have $\overline \rho(t,i,x)> \underline \rho(t,i,x)$ for  $(t,i,x)\in(0,+\infty)\times[0,i_\dagger)\times\R^d$ with $t>i+i_\star$, where $\mathcal{D}_{\overline I_0-\underline I_0}$ and $i_\star$ are respectively given by \eqref{I_0 support} and \eqref{i_star}  with $I_0$ repalced by  $\overline I_0-\underline I_0$;
	\item[(ii)] if  $\overline \rho_0\neq\underline \rho_0$, by further assuming that  \eqref{H3} is satisfied with $\rho_0$ replaced by  $\overline \rho_0-\underline \rho_0$, we have $\overline \rho(t,i,x)> \underline \rho(t,i,x)$ for  $(t,i,x)\in(0,+\infty)\times[0,i_\dagger)\times\R^d$ with $t>i$.
\end{itemize}

\end{prop}

The above comparison principle immediately extends to generalized super- and subsolutions, given by the minimum of supersolutions and maximum of subsolutions respectively.

 Since we are concerned with the ``nontrivial'' long time behavior of the solution $\rho$ to \eqref{edp} associated with nonnegative initial data $\rho_0$ (including the case that $\rho_0\equiv0$), we impose the following technical condition on $I_0$, for which Proposition \ref{prop-positivity}(i) shows that $\rho$ is eventually positive for large times.

\begin{Hypothesis}{(H3) - On the initial distribution $I_0$.}
We assume that $I_0\not\equiv0$ is nonnegative, bounded, continuous and compactly supported in $[0,i_\dagger)\times\R^d$ such that $\mathrm{Int}(\mathrm{supp}(\tau))\cap \mathrm{Int}(\mathcal{D}_{I_0})\neq \emptyset$ and 
\begin{equation}
\label{H2}
\mathcal{D}_{I_0}\subset\mathrm{Int}\big(\mathrm{supp}(\gamma)\big).
\end{equation} 
\end{Hypothesis}

To investigate the spreading property of \eqref{edp}, let us first look at its stationary problem:
\bqq\label{edpI-stationary}
\left\{
\begin{split}
\partial_i\rho(i,x) &=I_0(i,x) - \gamma(i) \rho(i,x) ,~  \quad i \in(0,i_\dagger), \quad x\in\R^d, \\
\rho(0,x)&=\S_0 \left(1-\exp\left(-\int_0^\infty \tau(i) \K*\rho(i,x)\md i \right)\right), \quad  x\in\R^d.
\end{split}\right.
\eqq
Due to the presence of $I_0$, problem \eqref{edpI-stationary} has no constant  solutions anymore. However, we can still manage to prove that problem \eqref{edpI-stationary} has a unique positive bounded solution. In order to characterize more precisely the asymptotic behavior (as $\|x\|\rightarrow+\infty)$ of the solution to  \eqref{edpI-stationary}, we introduce the following quantity 
\bqq
\label{basicR0}
\mathscr{R}_0:=\S_0 \int_0^\infty \omega(i)\md i \in(0,+\infty).
\eqq
Here, $\mathscr{R}_0$  stands for the basic reproduction number associated with problem \eqref{edp} \cite{Diekmann1,Diekmann2,Diekmann3,Thieme1}. Our second main result reads as follows.
\begin{thm}
	\label{thm_edpI-Liouville}
Assume {\bf(H1)-(H2)} and  that $I_0\not\equiv0$ satisfies {\bf(H3)}. The stationary problem \eqref{edpI-stationary} admits a unique positive bounded solution $U$ in $[0,i_\dagger)\times \R^d$. Moreover, $U$ satisfies
	\begin{align}
		\label{eqn-limit property}
		\lim\limits_{\|x\|\to+\infty}U(i,x)=\begin{cases}
			0,&\text{if}~\mathscr{R}_0\le 1,\\
			\rho^s(i), &\text{if}~\mathscr{R}_0>1,
		\end{cases}
		~~	\text{locally uniformly in}~i\in[0,i_\dagger),
	\end{align}
where $\rho^s(i):=\S_0 \rho^*\pi(i)$ for $i \in[0,i_\dagger)$ and $\rho^*\in(0,1)$ is the unique positive constant solution of equation $v=1-e^{-\mathscr{R}_0 v}$ when $\mathscr{R}_0>1$.
\end{thm}

As previously emphasized, model \eqref{edp} can be interpreted as a kind of nonlocal reaction-diffusion equation of Fisher-
KPP type with a heterogeneous forcing term $I_0(i,x)$. In this spirit, it is very closely related to the so-called field-road reaction-diffusion models studied in the past few years \cite{BRR16,BRR21}. In our model structured with the age since infection, and similarly as in the continuous and discrete cases \cite{BRR21,BF21}, the unique stationary solution $U$ given in the previous theorem is actually  a global attractor for the dynamics of \eqref{edp} starting from a nonnegative bounded compactly supported initial condition. To state our next main result, let us define the class of initial conditions $\rho_0$ that we shall be working with. 

\begin{Hypothesis}{(H4) - On the initial condition $\rho_0$.}
We assume that $\rho_0$ is nonnegative, absolutely continuous and compactly supported  in $[0,i_\dagger)\times\R^d$, such that $\rho_0/\pi\in L^\infty([0,i_\dagger)\times\R^d)$.  In particular, $\rho_0$ can be identically 0 in $[0,i_\dagger)\times\R^d$.
\end{Hypothesis}

\begin{thm}
	\label{thm-edpI_long time behavior}
Assume {\bf(H1)-(H2)} and suppose that $I_0\not\equiv0$ satisfies {\bf(H3)}. Let $\rho$ be the solution of \eqref{edp} associated with an initial condition $\rho_0$ satisfying {\bf(H4)}. Then,
	\bqs
	\rho(t,i,x) \rightarrow U(i,x)~~~~ \text{ as } t\rightarrow +\infty,
	\eqs
	locally uniformly in $(i,x)\in[0,i_\dagger)\times\R^d$, where $U$ is the unique positive stationary solution to \eqref{edp} given in Theorem \ref{thm_edpI-Liouville}.
\end{thm}

 Let us point out that even in the case of the zero initial condition $\rho_0\equiv0$ on $[0,i_\dagger)\times\R^d$, the long time 
dynamics of \eqref{edp} is nontrivial due to the presence of the source term $I_0(i,x)$. When $\mathscr{R}_0>1$, we can precisely characterize at which speed the epidemic spreads into the spatial domain, by imposing a  further assumption on the interaction kernel $\K$. That is,

\begin{Hypothesis}{(H2$\mu$) - Exponential localization.}
	Assume that $\K$ satisfies Hypothesis {\bf (H2)}.  Furthermore, there exists $\mu_0>0$, such that for any direction $\mathbf{e}\in\SS^{d-1}$,  one has $\int_{\R^d} \K(x)e^{\mu x\cdot \mathbf{e}} \md x<+\infty$ for $|\mu|<\mu_0$.   
	
\end{Hypothesis}
The condition that $\int_{\R^d} \K(x)e^{\mu x\cdot \mathbf{e}} \md x<+\infty$ for $|\mu|<\mu_0$ simply says that the kernel $\K$ decays at least at exponential rate along any direction $\mathbf{e}\in\SS^{d-1}$ as $\|x\|\rightarrow+\infty$. 
For future reference, we denote  
\bqq
\widetilde{\K}(\mu):=\int_{\R^d} \K(x)e^{\mu x\cdot \mathbf{e}} \md x=\int_\R\K_0(z)e^{\mu z}\md z
\label{eqtildeKmu}
\eqq
thanks to Hypothesis {\bf (H2)} and \eqref{1.4'}, which does not depend on $\mathbf{e}\in\SS^{d-1}$. 

 The spreading property of problem \eqref{edp} is the following.

\begin{thm}
	\label{thm-edpI-spreading property}
Assume {\bf(H1)-(H2$\mu$)} and $\mathscr{R}_0>1$, and suppose that $I_0\not\equiv0$ satisfies {\bf(H3)}. Then, there  exists some $c_*>0$, which is called the asymptotic spreading speed, such that the solution $\rho$ of \eqref{edp} starting from an initial condition $\rho_0$ satisfying {\bf(H4)}  satisfies:
	\begin{itemize}
		\item[(i)] for any $0<c<c_*$ and all $j\in(0,i_\dagger)$,
		\bqs
		\lim_{t\to+\infty}\sup_{\|x\|\le ct,~0\le i\le j}\big|\rho(t,i,x)-U(i,x)\big|= 0;
		\eqs
		
		\item[(ii)] for any $c>c_*$ and all $j\in(0,i_\dagger)$,
		\bqs
		\underset{t\rightarrow+\infty}{\lim } \underset{\|x\|\geq ct,~ 0 \leq i \leq j}{\sup} \rho(t,i,x) =0.
		\eqs
	\end{itemize}
\end{thm}

We will prove that the asymptotic  spreading speed $c_*$ of the epidemic wave actually  coincides with the asymptotic  spreading speed for the homogeneous  problem \eqref{edp} with $I_0\equiv0$, which is a result of independent interest. This feature is very similar to previously obtained results on spreading speeds for spatially extended epidemic models \cite{DG14,BRR21,BF21}. Although we are not making use of the following formula, it is possible to express the asymptotic speed of spreading $c_*$ through
\bqs
c_*=\underset{\alpha>0}{\min}~ \frac{1}{\alpha} \mathcal{L}[\omega]^{-1}\left(\frac{1}{\S_0\widetilde{\K}(\alpha)}\right),
\eqs
where we have defined $\widetilde{\K}(\alpha)$ in \eqref{eqtildeKmu} and $\mathcal{L}[\omega]^{-1}$ is the reciprocal function of the Laplace transform of $\omega$ defined as $\mathcal{L}[\omega](x):=\int_0^\infty\omega(i) e^{-x i}\md i$ for $x\geq0$. In our proof below, we will show that, when $\mathscr{R}_0>1$, the minimum is achieved at a unique positive real value $\alpha_*>0$.

Moreover, the asymptotic spreading speed $c_*$ of the solutions to the initial boundary value problem \eqref{edp} also turns out to be the threshold for the existence of traveling wave solutions associated with the homogeneous problem \eqref{kppI}. Let us first give the definition of traveling wave solutions. A traveling wave solution of \eqref{kppI} along any direction $\mathbf{e}\in\SS^{d-1}$ with speed $c\in\R$ is a solution of the form $\rho(t,i,x)=w(i,x\cdot \mathbf{e}-ct)$ satisfying
	\begin{equation}
	\label{TW}
	\left\{
	\begin{split}
		-c\partial_z w(i,z) &+ \partial_i  w(i,z) = - \gamma(i) w(i,z),  \quad i\in(0,i_\dagger), \quad z\in\R, \\
		w(0,z)&=\S_0 \left(1-\exp\left(-\int_0^\infty \tau(i) \K_0*w(i,z)\md i \right)\right),\quad  z\in\R,
	\end{split}
	\right.
\end{equation}
where we have set $z=x\cdot \mathbf{e}-ct$, and 
	\begin{equation}
		\label{TW-limit cdn}\left\{
		\begin{split}
		w(i,-\infty)&=\rho^s(i)~\text{and}~w(i,+\infty)=0~~\text{for each}~i\in[0,i_\dagger),\\ 0&\,<\, w(i,z) \, <\, \rho^s(i)~~~\text{for}~(i,z)\in[0,i_\dagger)\times\R,
		\end{split}\right.
	\end{equation}
 where $\rho^s(i)=\S_0\rho^*\pi(i)$ is the unique positive stationary solution of \eqref{kppI}, obtained in Theorem \ref{thm-liouville} below, when $\mathscr{R}_0>1$.  Then, substituting $w(i,x\cdot \mathbf{e}-ct)$ into \eqref{kppI}, we derive that

\begin{thm}\label{thm-TW-KPP model}
Assume {\bf(H1)-(H2$\mu$)} and $\mathscr{R}_0>1$. For any direction $\mathbf{e}\in\SS^{d-1}$, problem \eqref{kppI} admits a decreasing $($in $z$$)$ traveling front $w_c(i,z)$ satisfying \eqref{TW}--\eqref{TW-limit cdn} with speed $c$ if and only if $c\ge c_*$, with $c_*$ given in Theorem \ref{thm-edpI-spreading property}. Moreover, for $c\geq c_*$ the profile $w_c(i,z)$ is unique (modulo translation in $z$) and can be written as
\bqs
w_c(i,z)=\S_0\chi_c(z+ci)\pi(i), \quad (i,z)\in[0,i_\dagger)\times\R,
\eqs
with $0<\chi_c<\rho^*$ and $\chi_c'<0$ in $\R$ together with $\chi_c(-\infty)=\rho^*$ and $\chi_c(+\infty)=0$. Furthermore, there exist a unique $\alpha_c$ associated with $c>c_*$ and a unique $\alpha_*$ associated with $c_*$ satisfying $0<\alpha_c<\alpha_*$ such that $($up to normalization$)$
\begin{align*}
	\frac{\chi_c(\xi)}{ e^{-\alpha_c \xi}}\longrightarrow1 ~~~(\text{for}~c>c_*),~~~~\frac{\chi_{c_*}(\xi)}{\xi e^{-\alpha_* \xi}}\longrightarrow1\quad \text{ as } \xi\rightarrow+\infty.
\end{align*}
\end{thm}

The strategy of proof is to derive a nonlinear integral equation for the profiles $\chi_c$. As explained in the corresponding section below, there exists an astute change of variable which allows one to recover the traveling wave integral equation originally derived and studied by Diekmann in \cite{Diekmann1}. Our result is actually more precise in the sense that we get a full characterization of all possible nonincreasing $($in $z$$)$ traveling fronts satisfying \eqref{TW}--\eqref{TW-limit cdn}. In \cite{Diekmann1}, only the existence of super-critical fronts with wave speed $c>c_*$ was performed, and the critical case was later obtained through a limiting argument procedure in \cite{AR1976}. Here, we directly prove the existence of critical fronts with wave speed $c=c_*$ by a constructive procedure, which automatically gives the precise asymptotic behavior as $\xi\rightarrow+\infty$. Regarding the uniqueness part, the case of super-critical fronts with wave speed $c>c_*$ can be handled by using the results of \cite{Diekmann4}. Here, we also show the uniqueness of the critical fronts with speed $c=c_*$ which, to the best of our knowledge, was not present in the existing literature. Our approach is to use the strategy developped in  \cite{carrchmaj} where the uniqueness of critical fronts for nonlocal Fisher-KPP equations with compactly supported kernels was proved.

\subsection{What is new? and what is not?}

Our results provide a different, but complementary, perspective to the pioneering works of Aronson \cite{Aronson}, Diekmann \cite{Diekmann1,Diekmann2} and Thieme \cite{Thieme1,Thieme2,Thieme4} where asymptotic  speed of propagation for spatially extended epidemic models with nonlocal interactions were already proved. In that sense, our main results Theorems~\ref{thm_edpI-Liouville}--\ref{thm-TW-KPP model} are not surprising, but they offer a different perspective on the problem. More precisely, all previous studies \cite{Aronson,Diekmann1,Diekmann2,Thieme1,Thieme2,Thieme4} have worked on a fully integrated version of the model which has led to the development of new techniques to handle abstract functional nonlocal equations \cite{Diekmann1,Diekmann2,Diekmann4,S-1,S-2,Thieme1,Thieme2,Thieme4}, to name a few. Using the notation of the present paper, the initial starting point of the aforementioned works is to directly focus on the susceptible population $\S(t,x)$ by integrating \eqref{EDP} along the characteristics. Doing so, one first derives that
\bqs
\partial_t\S(t,x) = \S(t,x)\left( \int_0^t \omega(i) \frac{\K*\partial_t \S(t-i,x)}{\pi(t-i)}\md i -\int_t^{\infty} \omega(i) \frac{\K*I_0(t-i,x)}{\pi(t-i)}\md i\right), \quad t\geq 0, \quad x\in \R^d,
\eqs
which can then be integrated giving
\bqq
\mathcal{U}(t,x) = \S_0 \int_0^t \frac{\omega(i)}{\pi(t-i)}\left(1-\K*\exp(-\mathcal{U}(t-i,x))\right)\md i -\int_0^t \left(\int_s^{\infty} \omega(i) \frac{\K*I_0(t-i,x)}{\pi(t-i)}\md i\right) \md s, 
\label{eqcalU}
\eqq
for each $t\geq0$ and $x\in\R^d$, where $\mathcal{U}(t,x)$ is defined as
\bqs
\mathcal{U}(t,x):=-\ln\left(\frac{\S(t,x)}{\S_0}\right).
\eqs
The nonlinear nonlocal equation \eqref{eqcalU} is precisely the equation derived by Diekmann in \cite{Diekmann1} which has then led to the subsequent studies \cite{Diekmann2,Thieme1,Thieme2}.  Our point of view here is to directly work on the integrated version of the nonlocal transport equation \eqref{EDP}, and instead of working on the susceptible population, we rather focus on the cumulative density of infected individuals. This has the advantage to better use the intrinsic transport structure of the model, which is somehow hidden in the fully integrated nonlinear equation \eqref{eqcalU}. This alternate point of view will typically allow us to prove the strict positivity of the solutions of our nonlocal transport problem \eqref{edp} for the cumulative density of infected individuals, then yielding strong comparison principles.  We also argue that we obtain a better description of the epidemic dynamic by having a precise asymptotic behavior of  the cumulative density $\rho(t,i,x)$ of infected individuals. It also seems that our approach has the advantage to better understand the role of each parameters entering into the system. We finally refer to our subsequent work \cite{FRZ25} where we prove sharp asymptotics of the level sets of solutions, up to logarithmic corrections, for the homogeneous problem \eqref{kppI} with compactly supported initial conditions. 

From a purely mathematical perspective, we find it illuminating to compare the homogeneous model \eqref{kppI}, obtained by letting $I_0\equiv0$ into the original formulation \eqref{edp}, to reaction-diffusion equations with nonlocal spatial interactions and Fisher-KPP type nonlinearity, and  more specifically, to the so-called field-road reaction-diffusion models studied in the past few years \cite{BRR13,BRR13bis,BRR16,BRR21}. The analogy is only at the mathematical level, not in terms of modeling, since in our case the ``field" would be $(0,i_\dagger)\times\R^d$ and the ``road" would be $\left\{0\right\}\times\R^d$. In contrast with the cases studied in \cite{BRR13,BRR13bis,BRR16} where the dynamics in the field is modeled by a parabolic equation (typically a diffusion equation with possible reaction terms \cite{BRR13,BRR13bis}), our equation in the field is a transport equation at constant speed one with an inhomogeneous recovery rate $\gamma(i)$, which can be somehow comparable to the hostile parabolic field model proposed in \cite{BRR16}. Other key differences are in the equation on the road itself and the way that the model takes into account exchanges between the field and the road. More precisely, in the aforementioned works \cite{BRR13,BRR13bis,BRR16,BRR21}, the equation on the road is a pure diffusion equation and exchanges between the field and the road come from a Robin-like boundary condition. In our case, the dynamics on the road and the exchange terms are combined into a single equation where interactions in space are fully nonlocal and the contribution at the boundary of the domain, that is at $i=0$, is obtained by integrating the solution along the full field $(0,i_\dagger)\times\R^d$ (see \eqref{edp}(ii)). This is actually closer in spirit to the model presented by Pauthier \cite{pauthier} with nonlocal exchange terms in the standard  field-road reaction-diffusion models. Remarkably, despite these apparent differences, both models present the same rich asymptotic behavior with spreading. As the model with a hostile field investigated in \cite{BRR16}, we also show the existence of a sharp threshold, here characterized by the basic reproduction number $\mathscr{R}_0$ being below or above one, for the spreading dynamic to happen. 

Coming back to the epidemic modeling point of view, the transport nonlocal model has direct practical applications at inferring epidemic dynamics as it has been evident in the past few years \cite{Fouteletal,Richardetal}. It is also the building block for more advanced and relevant models which could include for example the age of infected and susceptible populations \cite{reyne} (and thus augment the model with additional transport equations) or consider several different strains (or variants) of a disease in a population \cite{ducasse2}, or even study the impact of vaccination strategies on epidemics \cite{gandon}. In all the possible extensions just mentioned above, it does not seem so obvious that one can formulate (and study), if possible, a fully integrated version of the model comparable to \eqref{eqcalU}, this is why we believe that directly tackling the nonlocal transport equation has valuable merits for future investigations.

\paragraph{Outline.}

In Section~\ref{prem-results}, we study the well-posedness of problem \eqref{edp} under some fairly general assumptions on the initial condition $\rho_0$ and source term $I_0$, and then use the transport structure of the model to derive somehow sharp positivity properties of the solutions under stronger assumptions when either $\rho_0$ or $I_0$ is nontrivial, which then lead to  strong comparison principles. In the following Section~\ref{secKPP}, we study the homogeneous problem, that is we study the long time behavior of the solutions to \eqref{edp} by letting $I_0\equiv0$. Finally, in Section~\ref{secMain} and Section~\ref{secTF}, we present the proofs of our main results. Along the way, we also provide in Section~\ref{secMain} a further asymptotic  property of the positive stationary solution $U$ to \eqref{edp}.

\section{Preliminary results: well-posedness, positivity and comparison principles}\label{prem-results}

In this section, we aim to show the existence and uniqueness of the solution to \eqref{edp}, i.e. Proposition~\ref{prop-edpI-well-posedness}, as well as the positivity of the solution stated in Proposition~\ref{prop-positivity} that arises as a consequence of the KPP structure of the boundary condition and the properties of the transport equation. Moreover, we will prove the comparison principle Proposition~\ref{prop-cp} for \eqref{edp}, which will be the main tool for the investigation of the long time dynamics of \eqref{edp}. Throughout this section, we assume that {\bf(H1)-(H2)} are satisfied. 

As a preliminary step, we first perform in \eqref{edp} the following change of  unknowns $\varrho(t,i,x):=\frac{\rho(t,i,x)}{\pi(i)}$ with $\varrho_0(i,x):= \frac{\rho_0(i,x)}{\pi(i)}$, then the initial boundary value problem satisfied by $\varrho$ is simply
\begin{equation}
\label{cauchy}
\left\{\begin{split}
\partial_t \varrho(t,i,x) &+ \partial_i\varrho(t,i,x) =\frac{I_0(i,x)}{\pi(i)} ,\quad t>0, \quad i\in(0,i_\dagger), \quad x\in\R^d \\
\varrho(t,0,x)&=\S_0 \left(1-\exp\left(-\int_0^\infty \omega(i) \K*\varrho(t,i,x)\md i \right)\right), \quad t> 0,  \quad x\in\R^d,\\
\varrho(0,i,x)&= \varrho_0(i,x),\quad  i\in[0,i_\dagger), \quad x\in\R^d.
\end{split}\right.
\end{equation}
We recall from {\bf(H1)}(iv) that in the case of $i_\dagger<\infty$, we have extended  the function $\omega$ by $0$ such that the above integral is well-defined. Integrating along the characteristics, we derive the following semi-explicit formula for the solution $\varrho$ of \eqref{cauchy} 
\bqq\label{integral eqn-edpI}
\varrho(t,i,x)=
\left\{
\begin{split}
\varrho_0(i-t,x)+\int_{i-t}^i \frac{I_0(\xi,x)}{\pi(\xi)} \md \xi,& \quad i\ge t,\\
\Phi(t-i,x) +\int_0^i \frac{I_0(\xi,x)}{\pi(\xi)} \md \xi,& \quad i< t,
\end{split}
\right.\eqq
for each $x\in\R^d$. The function $\Phi$ in \eqref{integral eqn-edpI} satisfies the following Volterra integral equation
\begin{align}
\Phi(t,x)
&= \S_0\left(1 - \exp\left(-\int_0^\infty\omega(i)\K*\varrho(t,x,i)\md i\right)\right)\label{renewal-1}\\
&=\S_0\left(1 - \exp\left(-\int_0^t\omega(i)\K*\Phi(t-i,x)\md i-\Gamma_1(\varrho_0)(t,x)-\Gamma_2(I_0)(t,x)\right)\right) \label{renewal}
\end{align}
for each $t\geq0$ and $x\in\R^d$, where we have set
\begin{align*}
\Gamma_1(\varrho_0)(t,x)&:=\int_0^\infty \omega(i+t)\K*\varrho_0(i,x)\md i,\\
\Gamma_2(I_0)(t,x)&:=\int_0^t\omega(i) \K*\left(\int_0^i \frac{I_0(\xi,x)}{\pi(\xi)} \md \xi\right)\md i +\int_t^{\infty}\omega(i)\K*\left(\int_{i-t}^i \frac{I_0(\xi,x)}{\pi(\xi)}\md\xi\right)\md i.
\end{align*}
Based on {\bf  (H1)}(iv), the above two terms have to be understood as follows in the case of $i_\dagger<\infty$
\bqs
\Gamma_1(\varrho_0)(t,x)=\left\{\begin{split}
&\int_0^{i_\dagger-t} \omega(i+t)\K*\varrho_0(i,x)\md i , &\quad t< i_\dagger,\\
&0, &\quad t\geq i_\dagger,\end{split}\right.
\eqs
\bqs
\Gamma_2(I_0)(t,x)=\left\{\begin{split}
&\int_0^t\omega(i)\K*\left(\int_0^i \frac{I_0(\xi,x)}{\pi(\xi)} \md \xi\right) \md i +\int_t^{i_\dagger}\omega(i)\K*\left(\int_{i-t}^i \frac{I_0(\xi,x)}{\pi(\xi)}\md\xi\right)\md i, &\quad t< i_\dagger,\\
&\int_0^{i_\dagger}\omega(i)\K*\left(\int_0^i \frac{I_0(\xi,x)}{\pi(\xi)} \md \xi\right)\md i , &\quad t\geq i_\dagger.\end{split}\right.
\eqs
We shall prove the well-posedness of problem \eqref{cauchy} as well as the comparison principle by studying  the Volterra integral equation \eqref{renewal}.  For later use and for the sake of convenience, let us define the right-hand side of \eqref{renewal} by the mapping $\F$ as
\begin{equation}
	\F(\Phi;\varrho_0,I_0)(t,x):= \S_0\Big(1 - \H(\Phi;\varrho_0,I_0)(t,x)\Big)~~~t> 0,~x\in\R^d,
\end{equation}
with 
\begin{equation}
\label{function h}
\H(\Phi;\varrho_0,I_0)(t,x):= \exp\left(
			-\int_0^t\omega(i)\K*\Phi(t-i,x)\md i		
			-\Gamma_1(\varrho_0)(t,x)-\Gamma_2(I_0)(t,x)\right).
\end{equation}
Then, problem \eqref{renewal} can be written abstractly as
\begin{equation}
	\label{eqn-F}
	\Phi=\F(\Phi;\varrho_0,I_0).
\end{equation}

\subsection{Some results on the Volterra equation \eqref{eqn-F}}

We first introduce the notion of super- and subsolutions for the Volterra equation \eqref{eqn-F}.

\begin{defi}\label{def_super+sub_renewal eqn}
	We say that a function $\Phi\in \mathscr{C}(\R_+,\mathscr{C}_b(\R^d))$ is a supersolution (resp. subsolution) of \eqref{eqn-F} on $\R_+\times\R^d$ associated with some nonnegative bounded functions $\varrho_0$ and $I_0$ defined on $[0,i_\dagger)\times\R^d$ such that $\Gamma_1(\varrho_0)$ and $\Gamma_2(I_0)$ are well-defined on $\R_+\times\R^d$, if 
	\begin{equation*}
		\Phi\ge \F(\Phi;\varrho_0,I_0),~~~(\text{resp.}~\Phi\le \F(\Phi;\varrho_0,I_0))~~~\text{on}~\R_+\times\R^d.
	\end{equation*}
\end{defi}

We have the following comparison principle.

\begin{lem}
	\label{lem_cp}
	Assume that $\overline\Phi\in \mathscr{C}(\R_+,\mathscr{C}_b(\R^d))$ is a supersolution to \eqref{eqn-F} associated with some nonnegative bounded functions $\overline \varrho_0$ and $\overline I_0$ defined on $[0,i_\dagger)\times\R^d$ such that $\Gamma_1(\overline\varrho_0)$ and $\Gamma_2(\overline I_0)$ are well-defined on $\R_+\times\R^d$, and that $\underline\Phi\in\mathscr{C}(\R_+,\mathscr{C}_b(\R^d))$ is a subsolution to \eqref{eqn-F} associated with some nonnegative bounded functions $\underline \varrho_0$ and $\underline I_0$ defined on $[0,i_\dagger)\times\R^d$ such that $\Gamma_1(\underline \varrho_0)$ and $\Gamma_2(\underline I_0)$ are well-defined on $\R_+\times\R^d$, in the sense of Definition \ref{def_super+sub_renewal eqn}. If $\overline \varrho_0\ge \underline{\varrho}_0$ and $\overline I_0\ge \underline I_0$ in $[0,i_\dagger)\times\R^d$, then $\overline \Phi\ge\underline \Phi$ on $\R_+\times\R^d$.
\end{lem}

\begin{proof}[Proof of Lemma \ref{lem_cp}] 
	Let $\overline \Phi, \underline \Phi\in\mathscr{C}(\R_+,\mathscr{C}_b(\R^d))$ be respectively a super- and a subsolution to \eqref{eqn-F} associated with nonnegative bounded functions $(\overline \varrho_0,\overline I_0)$ and $(\underline \varrho_0,\underline I_0)$ defined on $[0,i_\dagger)\times\R^d$ satisfying $\overline \varrho_0\ge \underline \varrho_0\ge 0$ and $\overline I_0\ge \underline I_0\geq0$ in $[0,i_\dagger)\times\R^d$. For $t\geq0$ and for $k>0$, set
	\begin{align*}
		w_k(t):=e^{-kt}\sup_{x\in\R^d}\left(\underline\Phi-\overline\Phi \right)^+(t,x),
	\end{align*}
	where we use the convention that $a^+=\max(0,a)$. Fix now any $T>0$, we then set
	\begin{equation*}
		W_k:=\sup_{t\in[0,T]}w_k(t).
	\end{equation*}
	By a straightforward computation, we have for $t\in[0,T]$ and $x\in\R^d$,
	\begin{align*}
		\left(\underline \Phi-\overline\Phi\right)(t,x)&\le \F(\underline \Phi;\underline \varrho_0, \underline I_0)(t,x) -  \F(\overline \Phi;\overline \varrho_0,\overline I_0)(t,x) \\
		&=\S_0 \H(\underline\Phi;\overline \varrho_0,\overline I_0)(t,x)\left(1-\exp\left(-\left(\Gamma_1(\underline \varrho_0)(t,x)-\Gamma_1(\overline \varrho_0)(t,x)\right)\right)\right)\\
		&~~~+\S_0 \H(\underline\Phi;\underline \varrho_0,\overline I_0)(t,x)\left(1-\exp\left(-\left(\Gamma_2(\underline I_0)(t,x)-\Gamma_2(\overline I_0)(t,x)\right)\right)\right)\\
		&~~~+\S_0\H(\overline\Phi;\overline \varrho_0,\overline I_0)(t,x)\left(1-\exp\left(-\int_0^t\omega(i)\K*\left(\underline\Phi(t-i,x)-\overline\Phi(t-i,x) \right)\md i\right)\right),
	\end{align*}
	where $\H$ is given in \eqref{function h}. Since $\overline \varrho_0\ge \underline \varrho_0\ge 0$ in $[0,i_\dagger)\times\R^d$, we get that $\Gamma_1(\overline \varrho_0)\ge \Gamma_1(\underline \varrho_0)\ge 0$. Similarly, since $\overline I_0\ge \underline I_0\ge 0$ in $[0,i_\dagger)\times\R^d$, we also have that $\Gamma_2(\overline I_0)\ge \Gamma_2(\underline I_0)\ge 0$. This implies that the first two terms  of the right-hand side in the above formula are nonpositive on $[0,T]\times\R^d$. This, combined with the fact that 
	$\H(\overline\Phi;\overline \varrho_0,\overline I_0)(t,x)\le 1$ for $(t,x)\in[0,T]\times\R^d$, leads to
	\begin{equation*}
		\left(\underline \Phi-\overline\Phi\right)(t,x)\le  \S_0 \int_0^t\omega(i)\K*\left(\underline\Phi(t-i,x)-\overline\Phi(t-i,x) \right)^+\md i,
	\end{equation*}
	which further implies that
	\begin{align*}
		w_k(t)\le \S_0  \int_0^t\omega(i)w_k(t-i) e^{-k i}\md i~~~\text{for}~t\in[0,T].
	\end{align*}
	Therefore, by taking the supremum over $[0,T]$ in the above inequality, it is further deduced that
	\begin{equation*}
		W_k	\le \S_0 W_k \int_0^{T}\omega(i) e^{-k i}\md i,
	\end{equation*}
	where the right-hand side converges to zero as $k\to+\infty$ by applying the Lebesgue's  dorminated convergence theorem, thanks to the Hypothesis {\bf(H1)}(iii)-(iv) that $\omega\in L^1([0,i_\dagger))$ and that it is extended by $0$ when $i_\dagger<\infty$. Consequently, $W_k\le 0$ for sufficiently large $k$, whence $\underline\Phi\le \overline \Phi$ in $[0,T]\times\R^d$. Since $T>0$ was chosen arbitrarily, it follows that  $\underline\Phi\le \overline \Phi$ in $\R_+\times\R^d$. This completes  the proof.
\end{proof}

Now we present a proof of the existence and uniqueness of solutions for \eqref{eqn-F} for the sake of completeness, which follows rather standard lines. It will also pave the way towards the proof of the positivity of the solutions. We refer to  \cite{IM17} for an exhaustive treatment in the spatially homogeneous case. 


\begin{lem}\label{lemma_well-posedness_renewal eqn} For any nonnegative bounded initial condition $\varrho_0$ on $[0,i_\dagger)\times\R^d$ and for any nonnegative bounded $I_0$ on $[0,i_\dagger)\times\R^d$ such that $\I_0(i,x):=\int_0^i\frac{I_0(\xi,x)}{\pi(\xi)}\md\xi$ is well-defined  and bounded for $(i,x)\in [0,i_\dagger)\times\R^d$, problem \eqref{eqn-F} admits a unique  nonnegative bounded solution $\Phi\in\mathscr{C}(\R_+,\mathscr{C}_b(\R^d))$.
\end{lem}

\begin{proof}[Proof of Lemma \ref{lemma_well-posedness_renewal eqn}]
	Let $\varrho_0$ and $I_0$ be as in the statement.  
	Fix any $T\in(0,+\infty)$, and define 
\bqs
\X_T:=\left\{ \Phi \in \mathscr{C}([0,T],\mathscr{C}_b(\R^d)) ~|~ \Phi\geq 0\right\},
\eqs
the space of nonnegative continuous vector-valued functions from $[0,T]$ to $\mathscr{C}_b(\R^d)$, where $\mathscr{C}_b(\R^d)$ denotes the space of bounded continuous functions on $\R^d$. We endow $\X_T$ with the norm of the uniform convergence $\|\cdot\|_{\X_T}$:
\bqs
\|\Phi\|_{\X_T}:=\sup_{t\in[0,T]}\left\|\Phi(t,\cdot)\right\|_{L^\infty(\R^d)}.
\eqs
 Since $\varrho_0$ and $\I_0$ are bounded in $[0,i_\dagger)\times\R^d$, together with  Hypotheses {\bf(H1)}, {\bf(H2)}  on $\omega$ and $\K$, we readily get that $\Gamma_1(\varrho_0)$ and $\Gamma_2(I_0)$ belong to $\mathscr{C}([0,T],\mathscr{C}_b(\R^d))$,  by a direct application of the Lebesgue's dominated convergence theorem. Furthermore, since $\varrho_0$ and $I_0$ are both nonnegative, we also have $\Gamma_1(\varrho_0)\geq0$ and $\Gamma_2(I_0)\geq0$ on $[0,T]\times\R^d$.  As a consequence, we deduce that for a given $\Phi\in\X_T$, we have $\F(\Phi;\varrho_0,I_0)\in\mathscr{C}([0,T],\mathscr{C}_b(\R^d))$ and by the monotone increasing property of $\Phi\in \X_T\mapsto \F(\Phi;\varrho_0,I_0)$ we also have $0\leq\F(0;\varrho_0,I_0)\leq \F(\Phi;\varrho_0,I_0)$, whence $\F(\Phi;\varrho_0,I_0)\in\X_T$. Finally, one also has the upper bound $\F(\Phi;\varrho_0,I_0)\leq \S_0$ on $[0,T]\times\R^d$.
	
We now construct iteratively a monotone sequence of functions $\Phi^k\in\X_T$ as follows. For each $k\in\N$,  let
\bqs
\left\{\begin{split} \Phi^{k+1} &= \F(\Phi^k;\varrho_0,I_0),\\
\Phi^0&=0,
\end{split}\right.\quad \text{ on } [0,T]\times\R^d.
\eqs
Based on the preceding discussion, we note that $\Phi^0=0\leq \F(0;\varrho_0,I_0)=\Phi^1$, such that inductively, we get that $\Phi^k\geq0$ for each $k\in\N$ due to the monotone increasing property of $\Phi\in \X_T\mapsto \F(\Phi;\varrho_0,I_0)$. At the same time, we also deduce that
\bqs
0 \leq \Phi^1\leq \cdots \leq\Phi^k \leq \Phi^{k+1} \leq \cdots \leq\S_0, \quad\text{ on } [0,T]\times\R^d,
\eqs
for each $k\in\N$. Furthermore, by induction, we also get that $\Phi^k\in\mathscr{C}([0,T],\mathscr{C}_b(\R^d))$ since both $\Gamma_1(\varrho_0)$ and  $\Gamma_2(I_0)$ belong to $\mathscr{C}([0,T],\mathscr{C}_b(\R^d))$. Next, we compute for each $k\in\N$, $t\in[0,T]$ and $x\in\R^d$,
\begin{align*}
\Phi^{k+1}(t,x)-\Phi^{k}(t,x)&=\F(\Phi^{k};\varrho_0,I_0)(t,x)-\F(\Phi^{k-1};\varrho_0,I_0)(t,x)\\
&=\S_0 \H(\Phi^{k-1};\varrho_0,I_0)(t,x)\left(1-\exp\left(-\int_0^t\omega(i)\K*\left(\Phi^{k}(i,x)-\Phi^{k-1}(i,x)\right)\md i\right)\right)\\
&\leq \S_0\int_0^t\omega(i)\K*\left(\Phi^{k}(i,x)-\Phi^{k-1}(i,x)\right)\md i,\\
&\leq \S_0 \tau_\infty \int_0^t\K*\left(\Phi^{k}(i,x)-\Phi^{k-1}(i,x)\right)\md i,
\end{align*}
since $ \H(\Phi^{k-1};\varrho_0,I_0)\leq 1$, $\omega\leq \tau\leq \tau_\infty$ and $\Phi^{k}-\Phi^{k-1}\geq0$. We can iterate the above procedure to obtain that
\bqs
\Phi^{k+1}(t,x)-\Phi^{k}(t,x)\leq (\S_0\tau_\infty )^k \int_0^t \int_0^{s_1}\cdots\int_0^{s_{k-1}} \K*\cdots*\K*\left(\Phi^{1}(s_k,x)-\Phi^{0}(s_k,x)\right)\md s_k \cdots \md s_1,
\eqs
which leads to
\bqs
\Phi^{k+1}(t,x)-\Phi^{k}(t,x)\leq \frac{(\S_0\tau_\infty T)^k}{k!}\left\|\Phi^1-\Phi^0\right\|_{\X_T}, \quad k\in\N, \quad t\in[0,T], \quad x\in\R^d.
\eqs
The above estimate implies that $(\Phi^k)_{k\geq0}$ is a Cauchy sequence in the Banach space $(\X_T,\|\cdot\|_{\X_T})$. As a consequence, $(\Phi^k)_{k\geq0}$ converges in $\X_T$ towards some limiting function $\Phi\in\X_T$, and passing to the limit in $ \Phi^{k+1} = \F(\Phi^k;\varrho_0,I_0)$, we have $\Phi=\F(\Phi;\varrho_0,I_0)$ with $0\leq \Phi \leq \S_0$ on $[0,T]\times\R^d$. Since $T\in(0,+\infty)$ is arbitrary, we derive that $\Phi\in\mathscr{C}(\R_+,\mathscr{C}_b(\R^d))$ is a solution of \eqref{eqn-F}.

Assume that $\Phi_1\in\X_T$ and $\Phi_2\in\X_T$ are two  solutions of \eqref{eqn-F}. Then, repeating the previous computations, we find that for each $t\geq0$,
\begin{align*}
\left|\Phi_1(t,x)-\Phi_2(t,x)\right|\leq \S_0\int_0^t\omega(i)\K*\left|\Phi_1(i,x)-\Phi_2(i,x)\right|\md i
\leq \S_0 \tau_\infty \int_0^t\left\|\Phi_1(i,\cdot)-\Phi_2(i,\cdot)\right\|_{L^\infty(\R^d)}\md i,
\end{align*}
and thus
\bqs
\left\|\Phi_1(t,\cdot)-\Phi_2(t,\cdot)\right\|_{L^\infty(\R^d)}\leq \S_0 \tau_\infty \int_0^t\left\|\Phi_1(i,\cdot)-\Phi_2(i,\cdot)\right\|_{L^\infty(\R^d)}\md i.
\eqs
The Gr\"onwall's lemma then implies that $\left\|\Phi_1(t,\cdot)-\Phi_2(t,\cdot)\right\|_{L^\infty(\R^d)}\equiv0$ for each $t\geq0$, and thus $\Phi_1(t,x)\equiv\Phi_2(t,x)$ for each $t\geq0$ and $x\in\R^d$. This completes the proof.
\end{proof}

From the above proof, we see clearly that when $\varrho_0\equiv0\equiv I_0$ on $[0,i_\dagger)\times\R^d$, then the iteration procedure yields that the solution $\Phi$ of problem \eqref{eqn-F} is nothing but the trivial solution $\Phi\equiv0$ on $\R_+\times\R^d$. In contrast, we obtain a nontrivial solution as soon as $\Gamma_1(\varrho_0)$ or $\Gamma_2(I_0)$ is nontrivial. In the case when $I_0\not\equiv0$, by requiring some further assumptions on the support of $I_0$, we can prove a strict positivity property for the solution $\Phi$ for $t$ large enough (which can be quantified).

\begin{lem}
\label{lem_strict positivity}
Under the assumption of Lemma \ref{lemma_well-posedness_renewal eqn} with $I_0\not\equiv0$, we assume that $\mathrm{Int}(\mathrm{supp}(\tau))\cap \mathrm{Int}(\mathcal{D}_{I_0})\neq \emptyset$, then upon setting 
\bqs
i_\star:=\min\left(\overline{\mathrm{Int}(\mathrm{supp}(\tau))\cap \mathrm{Int}(\mathcal{D}_{I_0})}\right)\in[0,i_\dagger),
\eqs 
we have that the unique nonnegative solution $\Phi$ of \eqref{eqn-F} satisfies $\Phi(t,x)>0$ for all $t> i_\star$ and $x\in\R^d$.
\end{lem}
\begin{proof}[Proof of Lemma~\ref{lem_strict positivity}]
The proof simply relies on the fact that for each $t> i_\star$ and $x\in\R^d$,
\bqs
\int_0^t\omega(i) \K*\left(\int_0^i \frac{I_0(\xi,x)}{\pi(\xi)} \md \xi\right)\md i>0,
\eqs
by definition of $i_\star$. This implies that $\Gamma_2(I_0)(,x)>0$ for $t> i_\star$ and $x\in\R^d$. Based on the proof of Lemma \ref{lemma_well-posedness_renewal eqn}, we then deduce that the solution  $\Phi$ to problem \eqref{eqn-F} satisfies  $\Phi(t,x)>0$ for all $t> i_\star$ and $x\in\R^d$.
\end{proof}

We close this subsection by the following observation. When $\varrho_0\equiv0$ and $\mathrm{Int}(\mathrm{supp}(\tau))\cap \mathrm{Int}(\mathcal{D}_{I_0})= \emptyset$, then we get that $\Gamma_1(\varrho_0)\equiv0\equiv\Gamma_2(I_0)$ on $\R_+\times\R^d$, and thus the solution $\Phi$ of problem \eqref{eqn-F} is identically 0  on $\R_+\times\R^d$. We then obtain from  \eqref{integral eqn-edpI} that the solution $\varrho$ of \eqref{cauchy} has the form, for each $x\in\R^d$, 
\bqs
\frac{\rho(t,i,x)}{\pi(i)}=\varrho(t,i,x)=
\left\{
\begin{split}
\int_{i-t}^i \frac{I_0(\xi,x)}{\pi(\xi)} \md \xi,& \quad i\ge t,\\
\int_0^i \frac{I_0(\xi,x)}{\pi(\xi)} \md \xi,& \quad i< t,
\end{split}
\right.\eqs
 as claimed initially in \eqref{1.6'} in the introduction. Similarly, when $I_0\equiv0$ and $\mathrm{Int}(\mathrm{supp}(\tau))\cap \mathrm{Int}(\mathcal{D}_{\varrho_0})= \emptyset$ with $\mathcal{D}_{\varrho_0}\subset[0,i_\dagger)$ where $\mathcal{D}_{\varrho_0}$ is defined as in \eqref{I_0 support} with $I_0$ replaced by $\varrho_0$ this time, then we get that $\Gamma_1(\varrho_0)\equiv0\equiv\Gamma_2(I_0)$ on $\R_+\times\R^d$, and thus the solution $\Phi$ of problem \eqref{eqn-F} is the trivial solution $\Phi\equiv0$ on $\R_+\times\R^d$ which then implies that, for each $x\in\R^d$, 
\bqs
\varrho(t,i,x)=
\left\{
\begin{split}
\varrho_0(i-t,x),& \quad i\ge t,\\
0,& \quad i< t.
\end{split}
\right.\eqs

\subsection{Well-posedness -- Proof of Proposition \ref{prop-edpI-well-posedness}}

\begin{proof}[Proof of Proposition~\ref{prop-edpI-well-posedness}.] 
Assume Hypotheses {\bf (H1)-(H2)}, and that $I_0$ is nonnegative, bounded, continuous and compactly supported on $[0,i_\dagger)\times\R^d$, and that $\rho_0$ is nonnegative such that $\rho_0/\pi$ is bounded and absolutely continuous on $[0,i_\dagger)\times \R^d$. First of all, we derive from Lemma \ref{lemma_well-posedness_renewal eqn} that problem \eqref{eqn-F} admits a unique nonnegative bounded and continuous solution $\Phi$ on $\R_+\times\R^d$ associated with $\varrho_0=\rho_0/\pi$ and $I_0$.  Then, it is immediate to conclude from the semi-explicit formula \eqref{integral eqn-edpI} that problem \eqref{cauchy} admits a unique nonnegative bounded solution $\varrho$ given by \eqref{integral eqn-edpI} on  $\R_+\times[0,i_\dagger)\times \R^d$  which is continuous in $(t,i)$ with $t\neq i$ uniformly with respect to $x\in\R^d$.  Now, since $\tau$ is absolutely continuous (and thus $\omega$) and both $\varrho_0$ and $\I_0$ are bounded in $ [0,i_\dagger)\times\R^d$, we get that $\Gamma_1(\varrho_0)$ and $\Gamma_2(I_0)$ are continuously differentiable in $t$ uniformly with respect to $x\in\R^d$. This then implies that $\partial_t \varrho(t,i,x)$ and $\partial_i \varrho(t,i,x)$ exist a.e. for  $(t,i,x)\in\R_+\times (0,i_\dagger)\times \R^d$ with $t\neq i$.  Therefore, by Definition \ref{def:classical}, the function $\rho=\varrho\pi$ in $\R_+\times[0,i_\dagger)\times\R^d$ is a unique classical solution to problem \eqref{edp} in $\R_+\times[0,i_\dagger,\R^d)$.
	
Now assume that $I_0\equiv0$, since the solution $\rho$ is identically 0 when $\rho_0\equiv0$,   we only consider the nontrivial case, i.e. $\rho_0\not \equiv0$. Since $\K\in W^{1,1}(\R^d)$ and thus $\partial_{x_p} \K \in L^1(\R^d)$ ($p=1,\cdots,d$), we also get that $\Gamma_1(\varrho_0)$ is also continuously differentiable with respect to $x\in\R^d$ for each $t>0$ with
\bqs
\partial_{x_p} \Gamma_1(\varrho_0)(t,x)=\int_0^\infty\omega(i+t) (\partial_{x_p}\K)*\varrho(i,x)\md i, \quad p=1,\cdots,d, \quad t>0, \quad x\in\R^d,
\eqs
and thus for each $ p=1,\cdots,d$
\bqs
\partial_{x_p}\Phi(t,x)=\S_0\left(\int_0^t\omega(i)(\partial_{x_p}\K)*\Phi(t-i,x)\md i  +\partial_{x_p}\Gamma_1(\varrho_0)(t,x)\right)\H(\Phi;\varrho_0,0)(t,x), \quad t>0, \quad x\in\R^d,
\eqs
showing that $\Phi$ constructed in Lemma~\ref{lemma_well-posedness_renewal eqn} is continuously differentiable with respect to $x\in\R^d$ for each $t>0$. Hence, we deduce that each $\partial_{x_p} \rho(t,i,x)$ $(p=1,\cdots,d)$ exists a.e. for $(t,i,x)\in\R_+\times (0,i_\dagger)\times \R^d$ with $t>i$. Since we assumed that $\rho_0$ is absolutely continuous on $[0,i_\dagger)\times\R^d$, then 
 $\partial_{x_p} \rho(t,i,x)$ ($p=1,\cdots,d$)  exist a.e. 
 for $(t,i,x)\in\R_+\times (0,i_\dagger)\times \R^d$ with $t\neq i$. This completes the proof.
\end{proof}

\subsection{Positivity -- Proof of Proposition \ref{prop-positivity}}

\begin{proof}[Proof of Proposition~\ref{prop-positivity}] 
	Assume Hypotheses {\bf (H1)-(H2)}, and that $I_0$ is nonnegative, bounded, continuous and compactly supported on $[0,i_\dagger)\times\R^d$, and that $\rho_0$ is nonnegative such that $\rho_0/\pi$ is bounded and absolutely continuous on $[0,i_\dagger)\times \R^d$. Let $\rho$ be the solution to problem \eqref{edp}.
	
	\noindent \textit{Proof of Statement (i).} Assume that $I_0\not\equiv 0$, and that $\mathrm{Int}(\mathrm{supp}(\tau))\cap \mathrm{Int}(\mathcal{D}_{I_0})\neq \emptyset$ with $i_\star\in[0,i_\dagger)$.
	To prove the strict positivity of the solution to \eqref{edp} in the case where $I_0\not\equiv0$, we simply observe from  \eqref{integral eqn-edpI} that
\bqs
\frac{\rho(t,i,x)}{\pi(i)}=\Phi(t-i,x)+\int_0^i \frac{I_0(\xi,x)}{\pi(\xi)}\md \xi, \quad t>i,~~ x\in\R^d.
\eqs
Thus, using Lemma~\ref{lem_strict positivity} with $i_\star$ being defined there, we get that $\Phi(t-i,x)>0$ for each $t-i> i_\star$ and all $x\in\R^d$, and the conclusion follows.

	\noindent \textit{Proof of Statement (ii).}
Suppose now that $\rho_0\not\equiv 0$ in $[0,i_\dagger)\times\R^d$ and that there are $0<\varpi\in\mathrm{Int}\big(\mathrm{supp}(\tau)\big)$ and $x_0\in\R^d$  satisfying 
\bqq\label{2.7'}
[0,\varpi]\times\left\{x_0\right\} \subset \mathrm{supp}(\rho_0).
\eqq
It is sufficient to prove, with our change of function $\varrho(t,i,x)=\frac{\rho(t,i,x)}{\pi(i)}$,  the positivity of the solutions $\varrho$ to problem \eqref{cauchy} for $(t,i,x)\in(0,+\infty)\times[0,i_\dagger)\times \R^d$ with $t>i$ with $I_0\equiv0$ and $\varrho_0\not\equiv0$.
Let us denote $\varrho$ the unique nonnegative bounded solution to problem \eqref{cauchy} provided by Proposition~\ref{prop-edpI-well-posedness} with initial datum $\frac{\rho_0}{\pi}\not\equiv 0$ in $[0,i_\dagger)\times\R^d$.

 To do so, we first claim that $\varrho(t,0,x)>0$ for $(t,x)\in(0,+\infty)\times\R^d$. Assume towards the contradiction that it were not true, then there would exist a point $(t_0,x_0)\in(0,\infty)\times \R^d$ such that $\varrho(t_0,0,x_0)=0$. We then infer from the boundary condition in \eqref{cauchy} that 
	\begin{equation}
		\label{boundary+}
		\int_0^\infty\omega(i)\int_{\R^d}\K(x_0-y)\varrho(t_0,i,y) \md y \md i=0.
	\end{equation}
	This immediately implies, since $\K>0$ in $\R^d$, that
	\begin{equation}\label{2.4}
		\varrho(t_0,i,x)=0~~~~\text{for}~i\in\text{supp}(\tau)\subset[0,i_\dagger),~x\in\R^d.
	\end{equation}
	Furthermore, we derive from \eqref{integral eqn-edpI}  that
	\begin{equation*}
		\begin{aligned}
			\varrho(t_0,i,x)=\begin{cases}
				\frac{\rho_0(i-t_0,x)}{\pi(i-t_0)}=0,~~~~&\text{for}~i\in\text{supp}(\tau)\cap(t_0,+\infty),\quad x\in\R^d,\\
				\Phi(t_0-i,x)=0,~~~~&\text{for}~i\in\text{supp}(\tau)\cap [0,t_0],\quad x\in\R^d.
			\end{cases}		
		\end{aligned}
	\end{equation*}    
	Assume first that $\text{supp}(\tau)\cap(t_0,+\infty)\neq \emptyset$, we then derive that
	\begin{equation*}
		\rho_0(i-t_0,x)=0~~~~~~~\text{for}~i\in\text{supp}(\tau)\cap(t_0,+\infty),\quad x\in\R^d,
	\end{equation*}
	contradicting \eqref{2.7'}. Assume now that  $\text{supp}(\tau)\subset[0,t_0]$, then it is seen that
	\begin{equation*}
		\Phi(t_0-i,x)=0~~~~~~~\text{for}~i\in\text{supp}(\tau)\subset[0,t_0],\quad x\in\R^d.
	\end{equation*}
Namely, $\Phi(t,x)\equiv 0$ for $t\in(t_0-\text{supp}(\tau))\subset [0,t_0]$ and $x\in\R^d$. Then we apply the formula \eqref{renewal-1} of $\Phi$ and obtain that 
	\begin{equation*}
		\int_0^\infty\omega(i)\int_{\R^d}\K(x-y)\varrho(t,i,y)\md y\md i =0~~~~\text{for}~t\in(t_0-\text{supp}(\tau)),~i\in\text{supp}(\tau), ~  x\in\R^d.
	\end{equation*}
	Hence, 
	\begin{equation}\label{t_1}
		\varrho(t,i,x)=0 ~~~~\text{for} ~t\in(t_0-\text{supp}(\tau)),~i\in\text{supp}(\tau)~x\in\R^d.
	\end{equation}
	Define now 
	\begin{equation*}
		t_1:=\min\{t_0-t~|~ t\in\text{supp}(\tau)\}.
	\end{equation*}
	We notice that $t_1\in[0,t_0)$. The formula \eqref{t_1} implies in particular that
	\begin{equation}
		\varrho(t_1,i,x)=0 ~~~~\text{for}~i\in\text{supp}(\tau),~x\in\R^d.
	\end{equation} 
	If $t_1=0$, then we immediately get a contradiction since $\varpi\in\mathrm{Int}\big(\mathrm{supp}(\tau)\big)$.
	In what follows, let us assume that $t_1\in(0,t_0)$. By repeating the argument as for \eqref{2.4}, we will reach the contradiction as long as $\text{supp}(\tau)\cap(t_1,+\infty)\neq\emptyset$ due to the assumption on the support of $\rho_0$. Otherwise, we have $\text{supp}(\tau)\subset [0,t_1]$ and we then obtain that 
	\begin{equation*}
		\Phi(t_1-i,x)=0~~~~~~~\text{for}~x\in\R^d,~i\in\text{supp}(\tau)\subset[0,t_1].
	\end{equation*}
	Namely, $\Phi(t,x)\equiv 0$ for $x\in\R^d$ and $t\in(t_1-\text{supp}(\tau))\subset [0,t_1]$. Then we apply again the formula \eqref{renewal-1} of $\Phi$ and arrive at 
	\begin{equation}\label{t_2}
		\int_0^\infty\omega(i)\int_{\R^d}\K(x-y)\varrho(t,i,y)\md y\md i =0~~~~\text{for}~t\in(t_1-\text{supp}(\tau)),~i\in\text{supp}(\tau),~ x\in\R^d.
	\end{equation}
	Namely,
	\begin{equation*}
		\varrho(t,i,x)=0 ~~~~\text{for} ~t\in(t_1-\text{supp}(\tau)),~i\in\text{supp}(\tau),~x\in\R^d.
	\end{equation*}
	Then we proceed with $t_2:=\min\{t_1-t~|~t\in\text{supp}(\tau)\}\in[0,t_1)$ and derive that
	\begin{equation*}
		\varrho(t_2,x,i)=0 ~~~~\text{for}~i\in\text{supp}(\tau),~x\in\R^d.
	\end{equation*}
	Again, there is a contradiction with \eqref{H3} provided that $t_2=0$. When $t_2\neq 0$, then we make the discussion as before.
	If $\text{supp}(\tau)\cap (t_2,+\infty)\neq \emptyset$, then we are done. Otherwise, we repeat previous procedure and retrieve \eqref{2.4} with a smaller time $t_3\in[0,t_2)$. After finite steps, we will find a time $t_{min}\in[0,t_2)$ such that either $t_{min}=0$, which is a contradiction; or $\text{supp}(\tau)\cap(t_{min},+\infty)\neq\emptyset$, which will  give a contradiction as well. As a consequence, we conclude that $\varrho(t,0,x)>0$ for $(t,x)\in(0,+\infty)\times\R^d$, as claimed.
	
	Assume now that there is a point $(t_0,i_0,x_0)\in(0,+\infty)\times(0,i_\dagger)\times \R^d$ with $i_0<t_0$ such that $\varrho(t_0,i_0,x_0)=0$.
	We deduce from the formula \eqref{integral eqn-edpI} of $\varrho$ that
	\begin{equation*}
		0\equiv\varrho(t_0,i_0,x_0)=\Phi(t_0-i_0,x_0),
	\end{equation*}
whence
	\begin{equation*}
		\int_0^\infty\omega(i)\int_{\R^d}\K(x_0-y)\varrho(t_0-i_0,i,y) \md y \md i=0.
	\end{equation*}
	Therefore, we are  led to \eqref{2.4} with $t_0$ replaced by $t_0-i_0$ this time. Following the lines as before, we eventually derive a contradiction. Consequently, we conclude that $\varrho(t,i,x)>0$ for  $t>0$, $x\in\R^d$ and $i\in(0, i_\dagger)$ with $t>i$. 
	The proof of Proposition~\ref{prop-positivity} is therefore achieved.
\end{proof}


\subsection{Comparison principle -- Proof of Proposition \ref{prop-cp}}

In this section, we prove the comparison principle for  problem~\eqref{edp} with the aid of Lemma \ref{lem_cp}.

\begin{proof}[Proof of Proposition \ref{prop-cp}]
	Let $\overline \rho$ and $\underline \rho$ be respectively a super- and a  subsolution of \eqref{edp} in $\R_+\times[0,i_\dagger)\times\R^d$ with nonnegative initial data $\overline \rho_0$ and $\underline \rho_0$ defined on $[0,i_\dagger)\times\R^d$ satisfying $\overline \rho_0/\pi,\underline \rho_0/\pi\in L^\infty([0,i_\dagger)\times\R^d)$ and $\overline I_0$ and $\underline I_0$ defined on $[0,i_\dagger)\times\R^d$ which are nonnegative, bounded, continuous and compactly supported in $[0,i_\dagger)\times\R^d$.
Assume that $\overline \rho_0\ge\underline \rho_0\geq0$ and $\overline I_0\ge\underline I_0\geq0$ in $[0,i_\dagger)\times\R^d$. By definition, $\overline \rho$ satisfies
\bqs
\left\{
\begin{split}
\partial_t \overline\rho(t,i,x) + \partial_i\overline\rho(t,i,x) &\geq \overline I_0(i,x) - \gamma(i) \overline\rho(t,i,x) ,\quad t>0, \quad i\in(0,i_\dagger),  \quad x\in\R^d,  \\
\overline\rho(t,0,x) &\geq \S_0 \left(1-\exp\left(-\int_0^\infty \tau(i) \K*\overline\rho(t,i,x)\md i \right)\right), \quad  t> 0,  \quad x\in\R^d,\\
\overline\rho(0,i,x)&=\overline\rho_0(i,x),   \quad     i\in[0,i_\dagger), \quad x\in\R^d.
\end{split}\right.
\eqs
Applying  the method of characteristics, we find that $\overline \rho$  satisfies
\bqs
\frac{\overline\rho(t,i,x)}{\pi(i)}\geq 
\left\{
\begin{split}
\frac{\overline\rho_0(i-t,x)}{\pi(i-t)}+\int_{i-t}^i \frac{\overline I_0(\xi,x)}{\pi(\xi)}\md\xi,& \quad i\ge t,\\
\overline\Phi(t-i,x) +\int_0^i \frac{\overline I_0(\xi,x)}{\pi(\xi)} \md \xi,& \quad i< t,
\end{split}
\right.
\eqs
with 
\bqs
\overline\Phi\ge \F\left(\overline \Phi;\frac{\overline \rho_0}{\pi},\overline I_0\right)  ~~~\text{on}~\R_+\times\R^d.
\eqs
Analogously, the subsolution $\underline \rho$ satisfies
\bqs
\frac{\underline\rho(t,i,x)}{\pi(i)}\leq 
\left\{
\begin{split}
\frac{\underline\rho_0(i-t,x)}{\pi(i-t)}+\int_{i-t}^i \frac{\underline I_0(\xi,x)}{\pi(\xi)}\md\xi,& \quad i\ge t,\\
\underline\Phi(t-i,x) +\int_0^i \frac{\underline I_0(\xi,x)}{\pi(\xi)} \md \xi,& \quad i< t,
\end{split}
\right.
\eqs
with 
\bqs
\underline\Phi\le \F\left(\underline \Phi;\frac{\underline \rho_0}{\pi},\underline I_0\right)  ~~~\text{on}~\R_+\times\R^d.
\eqs
Applying the comparison principle from Lemma \ref{lem_cp} to the Volterra equation \eqref{eqn-F}, we obtain that $\overline\Phi\ge\underline \Phi$ in $\R_+\times\R^d$. This, along with the above integral inequalities satisfied by $\overline \rho$ and $\underline \rho$, immediately implies  that $\overline \rho \ge \underline \rho$ in $\R_+\times[0,i_\dagger)\times\R^d$.
Let $\overline v$ and $\underline v$ be the solutions to \eqref{edp}  in $[0,i_\dagger)\times\R^d$ associated with initial data $\overline v(0,\cdot,\cdot)=\overline \rho_0$ and source term $\overline I_0$ and associated with  $\underline v(0,\cdot,\cdot)=\underline \rho_0$ and source term $\underline I_0$, respectively.

Assume first that  $\overline I_0\neq \underline I_0$ in $[0,i_\dagger)\times\R^d$, and that $\mathrm{Int}(\mathrm{supp}(\tau))\cap \mathrm{Int}(\mathcal{D}_{\overline I_0-\underline I_0})\neq \emptyset$. Using the comparison principle, we also get that \eqref{ordered} is satisfied. 	Set $w:=\frac{\overline v-\underline v}{\pi}$ on $\R_+\times[0,i_\dagger)\times\R^d$, then the function $w$ is nonnegative  and satisfies
\bqs
\left\{
\begin{split}
\partial_t w(t,i,x) &+ \partial_iw(t,i,x) =\frac{\overline I_0(i,x)- \underline I_0(i,x)}{\pi(i)}  ,\quad t>0,  \quad i\in(0,i_\dagger), \\
w(t,0,x)&=\S_0\exp\left(-\int_0^\infty \omega(i) \K*\underline v(t,i,x)\md i \right) \left(1-\exp\left(-\int_0^\infty \omega(i) \K*w(t,i,x)\md i \right)\right), \quad t> 0,\\
w(0,i,x)&= \frac{\overline \rho_0(i,x)-\underline \rho_0(i,x)}{\pi(i)}\geq 0,\quad    i\in[0,i_\dagger), 
\end{split}
\right.
\eqs
for all $x\in\R^d$.  We can then reproduce the argument in Lemma~\ref{lem_strict positivity} to obtain the eventual strict positivity of $w$, that is $w(t,i,x)>0$ for $(t,i,x)\in(0,+\infty)\times[0,i_\dagger)\times\R^d$ with $t>i+i_\star$, and the result follows. 

Assume now that  $\overline \rho_0\neq\underline \rho_0$, and that  \eqref{H3} is satisfied with $\rho_0$ replaced by  $\overline \rho_0-\underline \rho_0$. By the analysis above, we have that
\begin{equation}
	\label{ordered}
	\overline \rho \ge\overline v\ge \underline v\ge\underline \rho,\quad \text{on}~\R_+\times[0,i_\dagger)\times\R^d.
\end{equation} 
It is sufficient to show that the nonnegative function $w=\frac{\overline v-\underline v}{\pi}$ solving
\bqs
\left\{
\begin{split}
	\partial_t w(t,i,x) &+ \partial_iw(t,i,x) =0  ,\quad t>0,  \quad i\in(0,i_\dagger), \\
	w(t,0,x)&=\S_0\exp\left(-\int_0^\infty \omega(i) \K*\underline v(t,i,x)\md i \right) \left(1-\exp\left(-\int_0^\infty \omega(i) \K*w(t,i,x)\md i \right)\right), \quad t> 0,\\
	w(0,i,x)&= \frac{\overline \rho_0(i,x)-\underline \rho_0(i,x)}{\pi(i)}\gneqq 0,\quad    i\in[0,i_\dagger), 
\end{split}
\right.
\eqs
for all $x\in\R^d$, satisfies  $w(t,i,x)>0$ for $(t,i,x)\in(0,+\infty)\times[0,i_\dagger)\times\R^d$ with $t>i$. This problem has the same structure as \eqref{cauchy} in the case that $I_0\equiv0$. Hence, by repeating the argument in the proof of Proposition~\ref{prop-positivity} regarding the positivity of the solutions, we eventually derive that $w(t,i,x)>0$ for $(t,i,x)\in(0,+\infty)\times[0,i_\dagger)\times\R^d$ with $t>i$ since $w(0,\cdot,\cdot)= (\overline \rho_0-\underline \rho_0)/\pi$ satisfies \eqref{H3}. Thus, $\overline v(t,i,x)>\underline v(t,i,x)$ for $(t,i,x)\in(0,+\infty)\times[0,i_\dagger)\times\R^d$ with $t>i$. Together with \eqref{ordered}, we then conclude that   $\overline \rho(t,i,x)>\underline \rho(t,i,x)$ for $(t,i,x)\in(0,+\infty)\times[0,i_\dagger)\times\R^d$ with $t>i$. The proof of the comparison principle is therefore finished.
\end{proof}

\section{Analysis of the homogeneous model with $I_0\equiv0$}\label{secKPP}

In this section,  we focus on the initial boundary value problem \eqref{edp} in the homogenous case where $I_0\equiv0$:
\bqq
\label{fkpp}
\left\{
\begin{split}
	\partial_t \rho(t,i,x) + \partial_i\rho(t,i,x) &= - \gamma(i) \rho(t,i,x) ,\quad t>0, \quad i\in(0,i_\dagger),  \quad x\in\R^d,  \\
	\rho(t,0,x) &=\S_0 \left(1-\exp\left(-\int_0^\infty \tau(i) \K*\rho(t,i,x)\md i \right)\right), \quad  t> 0,  \quad x\in\R^d,\\
	\rho(0,i,x)&=\rho_0(i,x),   \quad     i\in [0,i_\dagger), \quad x\in\R^d.
\end{split}\right.
\eqq
  We expect that the analysis of this problem will shed light on the evolution of the heterogenous problem \eqref{edp} with $I_0\not\equiv0$. Indeed, as will be shown later,  the solutions of the heterogenous and homogeneous problems will possess  very similar  dynamics at large times. Also, as already emphasized in the introduction, the study of  problem \eqref{fkpp}  has its own mathematical interest since it shares lots of common features with the recently studied field-road reaction-diffusion models with so called Fisher-KPP type nonlinearities \cite{BRR16,BRR21}.

Throughout this section, we assume that {\bf(H1)-(H2)} are satisfied, and that the initial condition $\rho_0\not \equiv0$ satisfies {\bf (H4)} as well as \eqref{H3}.  Under such assumptions, we get the existence and uniqueness of a global classical solution $\rho$ for  problem \eqref{fkpp} with  its spatial partial derivatives $\partial_{x_p}\rho$ ($p=1,\cdots,d$) exist a.e. on $(0,+\infty)\times(0,i_\dagger)\times\R^d$, which satisfies $\rho(t,i,x)>0$ for each $(t,i,x)\in(0,+\infty)\times[0,i_\dagger)\times\R^d$ with $t>i$. 

\subsection{Liouville-type result and long time behavior of \eqref{kppI}}

To study the long time behavior of \eqref{fkpp}, let us first look at the corresponding stationary problem:
\bqq
\left\{
\begin{split}
	\partial_i\rho(i,x) &= - \gamma(i) \rho(i,x),\quad i\in(0,i_\dagger)  \quad x\in\R^d,\\
	\rho(0,x)&=\S_0 \left(1-\exp\left(-\int_0^\infty \tau(i) \K*\rho(i,x)\md i \right)\right),  \quad x\in\R^d.
\end{split}
\right.
\label{kppI-stat}
\eqq

Recall that $\mathscr{R}_0>0$ denotes the basic reproduction number associated with the problem, given  by
\bqs
\mathscr{R}_0:=\S_0 \int_0^\infty \omega(i)\md i = \S_0 \int_0^\infty \tau(i) e^{-\int_0^i\gamma(s)\md s} \md i.
\eqs

We have the following Liouville-type result for stationary problem \eqref{kppI-stat}.
\begin{thm}\label{thm-liouville}
 Under the Hypotheses {\bf (H1)}-{\bf (H2)}, problem \eqref{kppI-stat} admits  a trivial stationary solution 0. Moreover,  it admits a  positive stationary solution if and only if  $\mathscr{R}_0>1$. Such a positive stationary solution $\rho^s$, if any, is unique and given explicitly by $\rho^s(i)=\S_0 \rho^*\pi(i)$ for $i\in[0,i_\dagger)$ where $\rho^*\in(0,1)$ is the unique positive constant solution of $v=1-e^{-\mathscr{R}_0 v}$ when $\mathscr{R}_0>1$.
\end{thm}
\begin{proof}[Proof of Theorem \ref{thm-liouville}]
	First of all, we readily see that the disease free state $\rho\equiv 0$ is always a solution of \eqref{kppI-stat}, as expected. Let us now consider the nontrivial case.
	
	  From the first equation of \eqref{kppI-stat}, one infers that 
	\begin{equation}\label{formula-1}
		\rho(i,x)=\rho(0,x)\pi(i),~~~~~(i,x)\in[0,i_\dagger)\times\R^d.
	\end{equation}
	Plugging it into the boundary condition of \eqref{kppI-stat}, one then derives that
	\begin{equation*}
		\rho(0,x)=\S_0\left(1-\exp\left(-\int_0^\infty\omega(i)\md i~\K*\rho(0,x)\right)\right), \quad x\in\R^d.
	\end{equation*}
	Set
	\begin{equation}\label{varphi}
		\varphi(x):=\frac{\rho(0,x)}{\S_0},~~~x\in\R^d,
	\end{equation}
	 then by noticing that $\mathscr{R}_0=\S_0 \int_0^\infty \omega(i)\md i$, it follows that the function $\varphi$ satisfies
\begin{equation}\label{3.1}
	\varphi(x)=1-e^{-\mathscr{R}_0\K*\varphi(x)}, \quad x\in\R^d.
\end{equation}
Due to the monotone increasing property and also the concavity of the mapping $\varphi\in\X \mapsto \Psi(\varphi):=1-e^{-\mathscr{R}_0\K*\varphi}$ where $\X:=\left\{ \varphi\in \mathscr{C}(\R^d) ~|~ \varphi \geq 0\right\}$, it is not difficult to verify that \eqref{3.1} has at most one nontrivial nonnegative solution $\varphi$. Moreover, such $\varphi$, if any, satisfies  $0<\varphi<1$. This in turn implies that stationary problem \eqref{kppI-stat} admits at most one positive solution $0<\rho<\S_0$ in $[0,i_\dagger)\times\R^d$.

To reach our conclusion, we proceed with proving the sufficiency and necessity respectively. Suppose first that \eqref{kppI-stat} admits a unique positive solution $\rho^s$ in $[0,i_\dagger)\times\R^d$. By \eqref{varphi}, it is equivalent to assuming that \eqref{3.1} has a unique positive solution $\varphi$  in $\R^d$. We have to prove that $\mathscr{R}_0>1$. Assume by contradiction that $0<\mathscr{R}_0\le 1$. It follows from the Taylor expansion that
\begin{align}\label{eqn-varphi}
  	\varphi(x)=1-e^{-\mathscr{R}_0\K*\varphi(x)}< \mathscr{R}_0\K*\varphi(x)\le \K*\varphi(x),~~~~x\in\R^d,
\end{align}
whence,  $$\K*\varphi(x)-\varphi(x)=\int_{\R^d}\K(x-y)\left(\varphi(y)-\varphi(x)\right)\md y>0,\quad x\in\R^d.$$
One then infers from $\K>0$ in $\R^d$ that there is  $y_0\in\R^d\backslash\{x\}$ such that $$\varphi(y_0)-\varphi(x)>0.$$ 
 Then we multiply the above inequality  by $\K(y_0-x)$  and integrate over $x\in\R^d$, it follows that
\begin{equation*}
	0<\int_{\R^d}\K(y_0-x)\left(\varphi(y_0)-\varphi(x)\right)\md x=\varphi(y_0)-\K*\varphi(y_0)
\end{equation*}
contradicting \eqref{eqn-varphi}. As a consequence, we conclude that $\mathscr{R}_0>1$. 

Conversely,  let us now assume that $\mathscr{R}_0>1$ and we need to show that \eqref{kppI-stat} admits a unique positive solution. To do so, we first claim that the equation
\begin{equation}
	\label{eqn-v}
	v=1-e^{-\mathscr{R}_0 v}
\end{equation}
  with $\mathscr{R}_0>1$ admits a unique positive constant solution $\rho^*\in (0,1)$. In fact, since $v\mapsto Q(v):=1- e^{-\mathscr{R}_0 v}$ is  monotone increasing and concave in $\R_+$ and $Q(0)=0$, this implies that \eqref{eqn-v} has at most one positive constant solution which, if exists, takes  values in $(0,1)$. We then arrive at the conclusion as claimed, by noticing that $Q'(0)=\mathscr{R}_0>1$. Moreover, we observe that the positive constant function $\rho^*$  satisfies equation \eqref{eqn-v}, which, along with our conclusion that \eqref{eqn-v} admits at most one positive solution, leads to that $\rho^*$ is exactly the unique positive solution to \eqref{eqn-v} under the assumption that $\mathscr{R}_0>1$. Therefore, the stationary problem \eqref{kppI-stat} admits a unique positive solution which is spatially homogeneous,  given explicitly by $\rho^s(i)=\S_0\rho^*\pi(i)$ for $i\in[0,i_\dagger)$. The proof of Theorem \ref{thm-liouville} is therefore complete.
\end{proof}

We are now in position to state the long time behavior of the solutions to \eqref{kppI-stat}.

\begin{thm}
	\label{thm-long time behavior}
 Under the Hypotheses {\bf (H1)}-{\bf (H2)}, let  $\rho$ be the solution of problem \eqref{fkpp} in $\R_+\times[0,i_\dagger)\times\R^d$ associated with  initial condition $\rho_0\not\equiv0$ such that {\bf (H4)} and \eqref{H3} are satisfied. Then we have:
\begin{enumerate}[(i)]
		\item if  $\mathscr{R}_0\leq 1$, then  
		\bqs
		\rho(t,i,x) \rightarrow 0 ~~~ \text{ as } t\rightarrow +\infty, ~~~\text{uniformly in}~(i,x)\in[0,i_\dagger)\times \R^d;
		\eqs
	
		\item if $\mathscr{R}_0>1$, then 
		\bqs
		\rho(t,i,x) \rightarrow \rho^s(i)~~~ \text{ as } t\rightarrow +\infty, ~~~\text{locally uniformly in}~(i,x)\in[0,i_\dagger)\times \R^d,
		\eqs
		 where $\rho^s$ is the unique positive stationary solution to \eqref{fkpp} given in Theorem \ref{thm-liouville}.
	\end{enumerate}
\end{thm}

\begin{proof}[Proof of Theorem \ref{thm-long time behavior}]
 The main ingredient of the proof is based on  a comparison argument as well as the conclusion in Theorem \ref{thm-liouville}. Assume that  {\bf (H1)}-{\bf (H2)} hold.  Let  $\rho$ be the solution of problem \eqref{fkpp} in $\R_+\times[0,i_\dagger)\times\R^d$ associated with  initial condition $\rho_0\not\equiv0$ such that {\bf (H4)} and \eqref{H3} are satisfied.

\noindent
\textit{Proof of statement (i)}. Assume that $\mathscr{R}_0\le 1$. Let $M:=\max\big(\S_0,\Vert \rho_0 / \pi \Vert_{L^\infty([0,i_\dagger)\times\R^d)}\big)$, it is easy to check that $M \pi(i)$ is a supersolution to \eqref{fkpp} in $\R_+\times[0,i_\dagger)\times\R^d$.  Define by $\overline  \rho$  the solution to  problem \eqref{fkpp} with initial condition $\overline \rho_0=M\pi$ on $[0,i_\dagger)\times \R^d$. Applying the comparison principle Proposition~\ref{prop-cp}, we infer that the function $\overline \rho$ is  nonincreasing with respect to  $t$, and $0\le \rho(t,i,x)\le \overline \rho(t,i,x)$ for $(t,i,x)\in\R_+\times[0,i_\dagger)\times \R^d$. Passing to the limit as $t\to+\infty$, it follows from the monotone convergence theorem that $\overline \rho(t,i,x)$ converges to a stationary solution $U$ of \eqref{fkpp} pointwise in $[0,i_\dagger)\times\R^d$, then the convergence holds locally uniformly for $(i,x) \in[0,i_\dagger)\times \R^d$ due to the Dini's theorem by noticing that $U$ and $\overline \rho$ are continuous in $[0,i_\dagger)\times\R^d$ and in $\left\{(t,i,x)\in\R_+\times[0,i_\dagger)\times \R^d~|~ t>i\right\}$ respectively. That is,
	\begin{equation*}
		0\le \liminf_{t\to+\infty} \rho(t,i,x)\le  \limsup_{t\to+\infty} \rho(t,i,x)\le U(i,x)~~~\text{locally uniformly in}~(i,x)\in[0,i_\dagger)\times \R^d.
	\end{equation*}
By virtue of Theorem \ref{thm-liouville}, it is deduced that $U\equiv 0$ in $[0,i_\dagger)\times \R^d$, whence
	\begin{equation*}
	\lim_{t\to+\infty} \rho(t,i,x)= 0~~~\text{uniformly in}~(i,x)\in[0,i_\dagger)\times \R^d.
\end{equation*}

\noindent
\textit{Proof of statement (ii)}. We now assume that $\mathscr{R}_0=\S_0\int_0^{+\infty}\omega(i)\md i>1$. Our aim is to devise a compactly supported stationary subsolution, for which the spirit is the same as in Aronson and Weinberger \cite{AW2} and Diekmann \cite{Diekmann2}.  For $L\in(0, i_\dagger)$, let us define $\beta^L: [0,L)\to \R_+$ as
 \begin{equation}
 		\label{beta-definition}
 	 \begin{aligned}
 		\beta^L(i)=\begin{cases}
 			\gamma(i),~~~&\text{for}~i\in[0,L-2\varep],\\
 			\zeta(i), 
 			&\text{for}~i\in[L-2\varep,L-\varep],\\
 		\max\left(\frac{i}{L-i},2\gamma(i)\right),~~~&\text{for}~i\in[L-\varep,L),
 	\end{cases}
 \end{aligned}
 \end{equation}
where  $\varep\in(0,\min(1,L,i_\dagger-L)/3)$ is sufficiently small, and the  function $i\in[L-2\varep,L-\varep]\mapsto \zeta(i)$ is continuous such that $\zeta\ge\gamma$ on $[L-2\varep,L-\varep]$, and such that $\zeta(L-2\varep)=\gamma(L-2\varep)$ and $\zeta(L-\varep)=\max((L-\varep)/\varep, 2\gamma(L-\varep))$.
We observe from the construction of $\beta^L$ that 
\begin{equation*}
	\beta^L\ge\gamma~~\text{in}~[0,L),~~\beta^L\to\gamma~
	~\text{as}~L\to i_\dagger,~~\int_0^i\beta^L(s)\md s\to +\infty~~\text{as}~i\to L^-.
\end{equation*} 
Since $\int_{\R^d}\K(x)\md x=1$ and $K\in L^1(\R^d)$, there exist $R_0>0$ and $L_0\in(0,i_\dagger)$ large enough such that
\bqs
\widehat{\mathscr{R}_0}^{L,R}:=\S_0\int_0^L\tau(i) e^{-\int_0^i\beta^L(s)\md s}\md i\int_{B_R(0)}\K(x)\md x>1,
\eqs
for each $R\geq R_0$ and $L_0\le L<i_\dagger$. We thus fix $R$ and $L$ such that $\widehat{\mathscr{R}_0}^{L,R}>1$ is satisfied.  Then, for any $\eta>0$ sufficiently small, one can find $\eta'>0$ such that $(1-\eta)\widehat{\mathscr{R}_0}^{L,R} \geq 1+\eta'$.  Furthermore,  we notice that there is $\delta_0>0$ small  such that
\bqq\label{ineqRL}
\S_0 \left(1-\exp\left(- (\widehat{\mathscr{R}_0}^{L,R}/\S_0) u \right)\right)> (1-\eta) \widehat{\mathscr{R}_0}^{L,R} u \geq (1+\eta')u  ~~~~\text{for}~u\in [0,\delta_0).
\eqq 
Next, we denote by $\K^R(x):=\K(x)\mathds{1}_{B_R(0)}(x)$ for $x\in\R^d$ with $B_R(0)$ being the ball of radius $R$ and of center 0 in $\R^d$, and $\K_0^R(z):=\int_{\R^{d-1}} \K^R(z,x_2,\cdots,x_d)\md x_2\cdots \md x_d$ for $z\in\R$ and observe that $\K_0^R$ is even in $\R$, $\mathrm{supp}(\K_0^R)\subset[-R,R]$ and $\int_{B_R(0)} \K(x)\md x=\int_{\R^d}\K^R(x) \md x = \int_\R \K_0^R(z)\md z$. We now define 
\begin{align*}
	\psi_\nu(z):=\begin{cases}
		\cos(\nu z ),~~~~&|z|< \frac{\pi}{2\nu},\\
		0,&\text{elsewhere}.
	\end{cases}
\end{align*}
We claim  that there exist $\nu_0>0$ and $\iota_0>0$ such that for all $\nu\in(0,\nu_0)$ and $\iota\in[0,\iota_0)$ one has 
\bqq
\K_0^R*\psi_\nu(z)\geq \frac{\int_\R \K_0^R(z')\md z'}{1+\eta'}\psi_\nu(z-\iota), \quad z\in\R.
\label{eqkeyR}
\eqq
Indeed, for $|z|< \frac{\pi}{2\nu}$, we first have 
\bqs
\K_0^R*\psi_\nu(z)=\int_{-\frac{\pi}{2\nu}}^{\frac{\pi}{2\nu}}\K_0^R(z-y)\cos(\nu y)\md y \geq \int_\R \K_0^R(z-y)\cos(\nu y)\md y,
\eqs
provided that we select $\nu \leq \frac{\pi}{R}$, by noticing that when $|z-y|\le R$, it follows from $|z|< \frac{\pi}{2\nu}$ that necessarily $|y|\le R+\frac{\pi}{2\nu}\le \frac{3\pi}{2\nu}$ as long as $\nu \leq \frac{\pi}{R}$, which implies that the contribution of $y\in\R\backslash[-\frac{\pi}{2\nu},\frac{\pi}{2\nu}]$ in the above integral is nonpositive. Then, using the fact that $\K_0^R$ is even, we simply notice that
\begin{align*}
\int_\R \K_0^R(z-y)\cos(\nu y)\md y&= \cos(\nu z ) \int_\R \K_0^R(y)\cos(\nu y)\md y-\sin(\nu z ) \underbrace{\int_\R \K_0^R(y)\sin(\nu y)\md y}_{=0}\\
& \geq \cos(\nu z )\left( \int_\R \K_0^R(y)\md y -\frac{\nu^2}{2}\int_\R y^2\K_0^R(y)\md y\right).
\end{align*}
Since $\int_\R y^2\K_0^R(y)\md y>0$, there exists $\nu_0>0$ such that for all $\nu\in(0,\nu_0)$ one has
\bqs
\int_\R\K_0^R(y)\md y -\frac{\nu^2}{2}\int_\R y^2\K_0^R(y)\md y \geq \frac{\int_\R\K_0^R(y)\md y}{1+\eta_0}
\eqs
for $\eta_0\in(0,\eta')$, whence
\bqs
\K_0^R*\psi_\nu(z)\geq \frac{\int_\R\K_0^R(y)\md y}{1+\eta_0}\psi_\nu(z),  
\eqs 
which holds true for all $z$. By continuity, we can then find $\iota_0>0$ such that
\bqs
\K_0^R*\psi_\nu(z) \geq \frac{\int_\R\K_0^R(y)\md y}{1+\eta'}\psi_\nu(z-\iota),
\eqs
for all $\iota\in[0,\iota_0)$ and $z\in\R$. This proves the claim.

 Let us fix $\nu\in(0,\nu_0)$ and $\iota\in(0,\iota_0)$ such that inequality \eqref{eqkeyR} is satisfied, and define 
\bqs 
\underline{\psi}(x):=
\left\{
\begin{array}{cc}
1, & \| x\| \leq D,\\
\cos(\nu (\|x\|-D)),&  D\leq \|x\| \leq D+\frac{\pi}{2\nu} ,\\
0,&\text{elsewhere},
\end{array}\right.
\eqs
with $D> R\geq R_0$ which will be fixed below, and finally set 
\bqq\label{sub_kppI} 
\psi(i,x):=
\left\{
\begin{array}{cc}
e^{-\int_0^i\beta^L(s)\md s}\underline{\psi}(x), & (i,x)\in[0,L)\times\R^d,\\
0,&\text{elsewhere}.
\end{array}\right.
\eqq

We now verify that $\delta \psi$ with  $\delta\in(0,\delta_0)$ is a stationary subsolution to \eqref{fkpp} in $\R_+\times[0,i_\dagger)\times\R^d$. Note that only the boundary condition needs to be checked.  We first have
\begin{align*}
\S_0&\left(1-\exp\left(-\delta \int_0^\infty\tau(i)\K*\psi(i,x)\md i\right)\right)=\S_0\left(1-\exp\left(-\delta \int_0^L\tau(i) e^{-\int_0^i\beta^L(s)\md s}\md i~  \K*\underline{\psi}(x)\right)\right)\\
& \geq \S_0\left(1-\exp\Big(-\delta \int_0^L\tau(i) e^{-\int_0^i\beta^L(s)\md s}\md i  \int_{B_R(0)} \K(y)\underline{\psi}(x-y)\md y \Big)\right).
\end{align*}
Let $\|x\|\leq D-R$. Then for any $\|y\|\leq R$, we have that $\|x-y\|\leq D$ and thus $\underline{\psi}(x-y)=1$. As a consequence, we get that
\begin{align*}
\S_0\left(1-\exp\left(-\delta \int_0^\infty\tau(i)\K*\psi(i,x)\md i\right)\right) &\geq \S_0\left(1-\exp\left(-\delta \int_0^L\tau(i) e^{-\int_0^i\beta^L(s)\md s}\md i\int_{B_R(0)} \K(y)\md y \right)\right)\\
&=\S_0 \left(1-\exp\left(- (\widehat{\mathscr{R}_0}^{L,R}/\S_0) \delta \right)\right)\\ &\overset{\eqref{ineqRL}}{>}(1-\eta)\widehat{\mathscr{R}_0}^{L,R}\delta \geq (1+\eta') \delta \geq \delta = \delta \psi(0,x).
\end{align*}
Let $D-R\leq\|x\|\leq D+\frac{\pi}{2\nu}$. Then for any $\|y\|\leq R$, we use the estimate
\bqs
\|x-y\| \leq \|x\| -\frac{x\cdot y}{\|x\|}+\frac{\|y\|^2}{2\|x\|}\leq \|x\| -\frac{x\cdot y}{\|x\|} +\frac{R^2}{2(D-R)},
\eqs
and we select $D>R$ large enough such that $\frac{R^2}{2(D-R)}\leq \iota$. As a consequence, we have
\begin{align*}
\int_{B_R(0)} \K(y)\underline{\psi}(x-y)\md y&\geq \int_{B_R(0)} \K(y)\max_{\mu \geq -D}\psi_\nu\left( \|x\| -\frac{x\cdot y}{\|x\|} +\iota + \mu \right)\md y\\
&=\max_{\mu \geq -D}\int_{B_R(0)} \K(y)\psi_\nu\left( \|x\| -\frac{x\cdot y}{\|x\|} +\iota + \mu \right)\md y\\
&=\max_{\mu \geq -D} \int_\R \K_0^R(z)\psi_\nu\left( \|x\|  +\iota + \mu -z \right)\md z\\
&\overset{\eqref{eqkeyR}}{\geq} \frac{\int_{B_R(0)}\K(x)\md x}{1+\eta'}  \max_{\mu \geq -D}\psi_\nu\left( \|x\|  + \mu  \right),
\end{align*}
since $\int_{B_R(0)} \K(x)\md x=\int_{\R^d}\K^R(x) \md x = \int_\R \K_0^R(z)\md z$. And thus, in that range, one has
\begin{align*}
\S_0\left(1-\exp\left(-\delta \int_0^\infty\tau(i)\K*\psi(i,x)\md i\right)\right)&\geq\S_0 \left(1-\exp\left(- (\widehat{\mathscr{R}_0}^{L,R}/\S_0) \frac{\delta}{1+\eta'}   \max_{\mu \geq -D}\psi_\nu\left( \|x\|  + \mu  \right) \right)\right) \\
&\geq (1-\eta)\widehat{\mathscr{R}_0}^{L,R}  \frac{\delta}{1+\eta'}   \max_{\mu \geq -D}\psi_\nu\left( \|x\|  + \mu  \right)\\
& \geq \delta  \max_{\mu \geq -D}\psi_\nu\left( \|x\|  + \mu  \right)=\delta \psi(0,x).
\end{align*}
Here, we have used the fact that $\underline{\psi}(x)=\max_{\mu \geq -D}\psi_\nu\left( \|x\|  + \mu  \right)$ for all $x\in\R^d$.

This  implies that $\delta\psi$, with $\delta\in(0,\delta_0)$, is a subsolution to \eqref{fkpp} in $\R_+\times[0,i_\dagger)\times\R^d$.  Moreover, we infer from Proposition~\ref{prop-positivity} that  $\rho(t,i,x)>0$ for $(t,i,x)\in(0,+\infty)\times[0,i_\dagger)\times\R^d$ with $t>i$, due to the assumption \eqref{H3} on $\rho_0\not\equiv0$. Up to decreasing $\delta$ if needed, there further holds $\delta\psi(i,x)\le \rho(T,i,x)$ for $(i,x)\in[0,i_\dagger)\times\R^d$ with some $T>L>0$.

Let now $\overline \rho$ be defined as in the proof of Statement (i) and let $\underline \rho$ denote the solution of \eqref{fkpp} with initial datum $\underline \rho_0=\delta\psi$ in $[0,i_\dagger)\times\R^d$. We infer from the comparison principle Proposition~\ref{prop-cp} that $\overline \rho$ is nonincreasing with respect to $t$, whereas $\underline \rho$ is nondecreasing with respect to $t$. The monotone convergence theorem implies that $\overline \rho$ (resp. $\underline \rho$) converges in $[0,i_\dagger)\times\R^d$ as $t\to+\infty$ to a solution $\overline U$ (resp. $\underline U$) of stationary problem \eqref{kppI-stat} pointwise, and then locally uniformly in $(i,x)\in[0,i_\dagger)\times\R^d$ thanks to the Dini's theorem. Specifically, we have
	\begin{equation*}
	\delta\psi(i,x)\le  \underline U(i,x)\le \liminf_{t\to+\infty} \rho(t,i,x)\le  \limsup_{t\to+\infty} \rho(t,i,x)\le \overline  U(i,x),
\end{equation*}
locally uniformly in $(i,x)\in[0,i_\dagger)\times\R^d$. Together with Theorem~\ref{thm-liouville}, we derive that $\overline U$ and $\underline U$ are nothing but the positive stationary solution $\rho^s$ of \eqref{fkpp}. Consequently,
\begin{equation*}
	\lim_{t\to+\infty} \rho(t,i,x)=\rho^s(i),~~~~\text{locally uniformly in}~ (i,x)\in[0,i_\dagger)\times\R^d.
\end{equation*}
This completes the proof of Theorem \ref{thm-long time behavior}.
\end{proof}

\begin{rmk}Both Theorem~\ref{thm-liouville} and Theorem~\ref{thm-long time behavior} easily generalizes to the case of compactly supported kernels upon assuming now the condition~\eqref{eqrmkpos} on the initial condition. 
\end{rmk}

\subsection{Spreading property of \eqref{kppI}}
In this section, we shall prove under the assumption of $\mathscr{R}_0>1$ as well as {\bf (H1)}-{\bf (H2$\mu$)} that problem \eqref{fkpp}, starting from initial condition $\rho_0\not\equiv0$ such that {\bf (H4)} and \eqref{H3} are satisfied,  exhibits exactly the same spreading property as the Cauchy problem of reaction-diffusion equations with KPP nonlinearities \cite{AW2}. 

\begin{thm}
	\label{thm-spreading property}
	Assume that {\bf (H1)}-{\bf (H2$\mu$)} and $\mathscr{R}_0>1$ hold. Then, there exists some $c_*>0$, which is called the asymptotic spreading speed, such that the solution $\rho$ of  problem \eqref{fkpp}, associated with $\rho_0\not\equiv 0$ such that {\bf (H4)} and \eqref{H3} are satisfied, has the following properties:
	\begin{itemize}
			\item[(i)] For any $c>c_*$ and all $j\in(0,i_\dagger)$
		\bqs
		\underset{t\rightarrow+\infty}{\lim } \underset{\|x\|\geq ct,~ 0 \leq i \leq j}{\sup} \rho(t,i,x) =0.
		\eqs
		
		\item[(ii)] For any $0<c<c_*$ and all $j\in(0,i_\dagger)$
		\bqs
		\underset{t\rightarrow+\infty}{\lim } \underset{\|x\|\leq ct,~ 0 \leq i \leq j}{\sup} \left| \rho(t,i,x)-\rho^s(i)\right| =0.
		\eqs
	\end{itemize}
\end{thm}

\subsubsection{Exponential supersolution}

In this subsection, we aim to establish the upper bound of the asymptotic spreading speed for \eqref{fkpp}. Fix any direction $\mathbf{e}\in\SS^{d-1}$. We look for a supersolution of  problem \eqref{fkpp} of the form
	\begin{equation*}
	\overline \rho(t,i,x)=\min\left( M,  e^{-\alpha(x\cdot \mathbf{e}-ct)-\alpha ci}\right)\pi(i), \quad t\geq0, \quad i\in[0,i_\dagger), \quad x\in\R^d,
\end{equation*}
with $M:=\max\big(\S_0,\Vert \rho_0 / \pi \Vert_{L^\infty([0,i_\dagger)\times\R^d)}\big)>0$  such that, up to shifts, $\rho_0(i,x)< \overline \rho(0,i,x)$ for $i\in[0,i_\dagger)$, $x\in\R^d$.
Here, the exponential term
\begin{equation}
	\label{exponential term}
\vartheta(t,i,x):=e^{-\alpha(x\cdot \mathbf{e}-ct)-\alpha ci}\pi(i), \quad t\geq0, \quad i\in[0,i_\dagger), \quad x\in\R^d,
\end{equation} 
defined along any direction $\mathbf{e}\in\SS^{d-1}$ with speed $c>0$ and with parameter $\alpha>0$ (to be fixed in the following investigation), is expected to solve
the corresponding linearized problem of \eqref{fkpp} around the steady state $\rho\equiv0$:
\begin{equation}
\label{linearized pb}
\left\{
\begin{split}
\partial_t \rho(t,i,x) + \partial_i\rho(t,i,x) &= - \gamma(i) \rho(t,i,x), \quad t>0,\quad i\in(0,i_\dagger), \quad x\in\R^d, \\
\rho(t,0,x)&=\S_0 \int_0^\infty \tau(i) \K*\rho(t,i,x)\md i, \quad t> 0,  \quad x\in\R^d.
\end{split}\right.
\end{equation}
Substituting \eqref{exponential term} into \eqref{linearized pb}, we obtain the following dispersion relation: 
\begin{equation}
	\	\label{dr}
\underbrace{	\S_0\int_{\R^d}\K(x)e^{\alpha x\cdot \mathbf{e}} \md x \int_0^\infty \omega(i) e ^{-\alpha ci}\md i}_{\varphi_c(\alpha)}=1.
\end{equation}
We denote the left-hand side of \eqref{dr} by $\varphi_c(\alpha)$. For each $c\ge 0$, the function $\varphi_c(\alpha)$ is well-defined (at least) on the set 
\begin{equation}\label{Sigma}
	\Sigma:=\{\alpha\ge 0 ~|~ \int_{\R^d} \K(x)e^{\alpha x\cdot \mathbf{e}}\md x<+\infty\}.
\end{equation}
We readily see from {\bf (H2$\mu$)} that indeed $\Sigma=[0,\Lambda)$ with  $\Lambda\in[\mu_0,+\infty]$. Let us also remark that the dispersion relation \eqref{dr} was also derived in \cite{Diekmann1} where some basic properties were already given. Here, we will need a refined analysis of the dispersion relation in order to construct compactly supported subsolutions in the forthcoming section.

\begin{lem}\label{lemma-varphi}
	For each $\alpha\in(0,\Lambda)$, there is a unique $c(\alpha)\in(0,+\infty)$ such that $\varphi_{c(\alpha)}(\alpha)=1$.
\end{lem} 
\begin{proof}[Proof of Lemma \ref{lemma-varphi}]
 We first notice that when $c=0$,
 \begin{align*}
 	\varphi_0(\alpha)=\mathscr{R}_0\int_{\R^d}\K(x)e^{\alpha x\cdot\mathbf{e}}\md x\geq0 ,~~\alpha\in[0,\Lambda).
 \end{align*}
 By the Lebesgue's dorminated convergence theorem, we derive from {\bf (H2$\mu$)} that
 \begin{equation*}
 	\frac{\md}{\md \alpha}\varphi_0(\alpha)= \mathscr{R}_0\int_{\R^d} \left(x\cdot\mathbf{e}\right)\K(x)e^{\alpha x\cdot\mathbf{e}}\md x\geq0 ,~~\alpha\in[0,\Lambda).
 \end{equation*}
 That is,
  $\varphi_0(\alpha)$ is monotone nondecreasing in  $\alpha\in[0,\Lambda)$, whence
  \begin{equation}
  	\label{c-1}
  	\min_{\alpha\in[0,\Lambda)}\varphi_0(\alpha)=\varphi_0(0)=\mathscr{R}_0>1.
  \end{equation}
	Furthermore, due to the continuity of $c\mapsto\varphi_c(\alpha)$ for each $\alpha \in[0,\Lambda)$, it follows that there exists $\delta>0$ small such that
	\begin{equation}
		\label{c>0}
		\inf_{\alpha\in[0,\Lambda)}\varphi_c(\alpha)>1~~~\text{for all}~c\in[0,\delta].
	\end{equation} 
 On the other hand, we observe from the expression of $\varphi_c(\alpha)$ that for each $\alpha\in (0,\Lambda)$, $\varphi_c(\alpha)$ is monotone decreasing and convex in $c\ge 0$, and
 \begin{equation}
 	\label{c-2}
 	\varphi_c(\alpha)\to 0~~~~~\text{as}~~ c\to+\infty.
 \end{equation}
   Together with \eqref{c-1} and \eqref{c-2}, we then derive that for each $\alpha\in (0,\Lambda)$, there is a  unique $c(\alpha)\in(0,+\infty)$ such that $\varphi_{c(\alpha)}(\alpha)=1$. The proof of Lemma \ref{lemma-varphi} is thereby complete.
\end{proof}

Based upon Lemma \ref{lemma-varphi}, we can define
\begin{equation}\label{c*-def}
	c_*:=\underset{\alpha\in(0,\Lambda)}{\inf} c(\alpha).
\end{equation}
We immediately see from \eqref{c>0} that $c_*\in (0,+\infty)$. In the following, we show that there actually exists a unique $\alpha_*\in(0,\Lambda)$ such that above infimum can be attained.
\begin{lem}\label{lemma_c*} Let $c^*$ be given in \eqref{c*-def}.
For each $\alpha\in(0,\Lambda)$, let $(\alpha,c(\alpha))$ be the unique pair given in Lemma \ref{lemma-varphi} such that $\varphi_{c(\alpha)}(\alpha)=1$.	Then there is a unique $\alpha_*\in (0,\Lambda)$ such that 
	\begin{equation}
		\label{c_*}
		c_*=c(\alpha_*)=\min_{\alpha\in (0,\Lambda)}c(\alpha).
	\end{equation}
Moreover, $1=\varphi_{c_*}(\alpha_*)=\min_{\alpha\in(0,\Lambda)}\varphi_{c_*}(\alpha)$ and $\partial_\alpha\varphi_{c_*}(\alpha_*)=0$.
\end{lem}

\begin{proof}[Proof of Lemma \ref{lemma_c*}]
First of all, we proceed with a contradiction argument to prove that the infimum in the definition \eqref{c*-def} of $c^*$ is indeed the minimum, which gives  the existence of $\alpha_*\in(0,\Lambda)$.  Suppose first that $\alpha_*=0$. Since $c_*\in(0,+\infty)$, it follows from \eqref{dr} that $\mathscr{R}_0=1$. This is impossible. Thus, the case that $\alpha_*=0$ is excluded.  Now let us consider the case that $\alpha_*=\Lambda$. We observe that $\int_{\R^d}\K(x)e^{\Lambda x\cdot \mathbf{e}}\md x=\infty$. We distinguish two cases: either $\Lambda\in(0,+\infty)$ or $\Lambda=+\infty$. For the former case, it is easy to see that $1=\varphi_{c_*}(\alpha_*)=\varphi_{c_*}(\Lambda)=+\infty$. This case is ruled out. Finally, let us exclude the case of  $\Lambda=+\infty$.

 
 We let $i_0\geq0$ be such that $\omega(i_0)>0$. From the continuity of $\omega$, there exists $\epsilon_1>0$ such that $\omega(i)>\omega(i_0)/2$ for all $i\in[i_0,i_0+\epsilon_1]$, from which we deduce that
\begin{equation*}
\int_0^\infty \omega(i) e ^{-\alpha c_*i}\md i > \frac{\omega(i_0)}{2} \int_{i_0}^{i_0+\epsilon_1} e ^{-\alpha c_*i}\md i = \frac{\omega(i_0)\, e^{-\alpha c_* i_0} \, \left(1- e^{-\alpha c_*\epsilon_1}\right)}{2\alpha c_*}.
\end{equation*}
On the other hand, using \eqref{eqtildeKmu}, the positivity of $\K_0$ and the fact that $c_*>0$, there exists $z_0\in\R$ such that $\K_0(z_0)>0$ and $z_0>c_*i_0$. As a consequence, there exists $\epsilon_2>0$, small enough, such that $\K_0(z)>\K_0(z_0)/2$ for all $z\in[z_0-\epsilon_2,z_0]$, which implies:
\begin{equation*}
\widetilde{\K}(\alpha)=\int_{\R^d}\K(x)e^{\alpha x\cdot \mathbf{e}} \md x=\int_{\R}\K_0(z)e^{\alpha z} \md z > \frac{\K_0(z_0)\, e^{\alpha z_0}\,\left(1-e^{-\alpha\epsilon_2}\right)}{2\alpha}.
\end{equation*}
Summing up, we have obtained that
\begin{align*}
\varphi_{c_*}(\alpha)= \S_0\int_{\R^d}\K(x)e^{\alpha x\cdot \mathbf{e}} \md x \int_0^\infty \omega(i) e ^{-\alpha c_*i}\md i &> \frac{e^{\alpha(z_0- c_* i_0)}}{\alpha^2} \frac{\K_0(z_0)\omega(i_0)\left(1- e^{-\alpha c_*\epsilon_1}\right)\left(1-e^{-\alpha\epsilon_2}\right)}{4c_*} \\
& ~~~ \underset{\alpha\rightarrow+\infty}{\longrightarrow}+\infty.
\end{align*}
This is a contradiction since $1=\lim_{\alpha\to+\infty}\varphi_{c_*}(\alpha)$ by assumption.  Consequently, we obtain that the infimum in \eqref{c*-def} cannot be reached neither when $\alpha=0$ nor when $\alpha\to\Lambda$. Then, the infimum is necessarily the minimum, which implies the existence of $\alpha_*\in(0,\Lambda)$ such that $c_*=c(\alpha_*)\in(0,+\infty)$.

Finally, let us turn to the proof of the uniqueness of $\alpha_*$. As a matter of fact, since $\varphi_{c_*}(\alpha)\to +\infty$ as $\alpha\to \Lambda$ and $\varphi_{c_*}(0)=\mathscr{R}_0>1$, together with the convexity of the function $\alpha\in[0,\Lambda)\mapsto \varphi_{c_*}(\alpha)$, the uniqueness of $\alpha_*$ immediately follows. We also have $1=\varphi_{c_*}(\alpha_*)=\min_{\alpha\in(0,\Lambda)}\varphi_{c_*}(\alpha)$ and $\partial_\alpha\varphi_{c_*}(\alpha_*)=0$, by noticing that $\alpha\in[0,\Lambda)\mapsto \varphi_{c_*}(\alpha)$ is also analytic. The proof of Lemma \ref{lemma_c*} is complete.
\end{proof}

With the precise information of $(\alpha_*, c_*)$, we can further determine the range of $\alpha$ corresponding to $c>c_*$ such that $\varphi_c(\alpha)=1$. Here is our result.
\begin{lem}
	\label{lemma_alpha<alpha*}
	For each $c\ge c^*$, there is a unique $\alpha\in(0,\alpha_*]$ such that $\varphi_c(\alpha)=1$. Moreover, let $(\alpha, c(\alpha))$ be the pair in $[c_*,+\infty)\times (0,\alpha_*]$ given by Lemma \ref{lemma-varphi} such that $\varphi_{c(\alpha)}(\alpha)=1$, then $\alpha\in(0,\alpha_*]\mapsto c(\alpha)$ is decreasing.
\end{lem}
\begin{proof}[Proof of Lemma \ref{lemma_alpha<alpha*}]
	When $c=c_*$, it is done by Lemma \ref{lemma_c*}.
	Let us now fix any $c_0>c_*$. We first claim that $\varphi_{c_0}(\alpha_*)<1$. Indeed, since $c\mapsto\varphi_c(\alpha_*)$ is decreasing in $c\ge 0$ and since $\varphi_{c_*}(\alpha_*)=1$, along with the fact that $c_0>c_*$, it immediately follows that $\varphi_{c_0}(\alpha_*)<1$, as claimed. Therefore, by combining the convexity of  the function $\alpha\mapsto \varphi_{c_0}(\alpha)$  in $[0,\alpha_*]$, $\varphi_{c_0}(0)=\mathscr{R}_0>1$ and $\varphi_{c_0}(\alpha_*)<1$, it is deduced that $\alpha\mapsto \varphi_{c_0}(\alpha)$ is necessarily decreasing  in $[0,\alpha_*]$. As a consequence, there is a unique $\alpha_0\in (0,\alpha_*)$ such that $\varphi_{c_0}(\alpha_0)=1$. This finishes the first part of the proof.
	
	 Assume now that $0<\alpha_1<\alpha_2\le\alpha_*$, it follows from Lemma \ref{lemma-varphi} that there exist $c_1$ and $c_2$ (both larger than $c_*$) such that  $(c_1, \alpha_1)$ and $(c_2,\alpha_2)$ are respectively the unique pairs such that $\varphi_{c_1}(\alpha_1)=1=\varphi_{c_2}(\alpha_2)$. We first show that $c_1>c_2$, which gives the monotonicity. Moreover, we derive from the above analysis that $\alpha\mapsto \varphi_{c_1}(\alpha)$ is decreasing  in $[0,\alpha_*]$, whence $\alpha_1<\alpha_2$ will imply that $1=\varphi_{c_1}(\alpha_1)>\varphi_{c_1}(\alpha_2)$. Hence, $\varphi_{c_2}(\alpha_2)=1>\varphi_{c_1}(\alpha_2)$. Since $c\mapsto\varphi_c(\alpha_2)$ is decreasing in $c\ge 0$, we immediately have $c_1>c_2$. 
This completes the proof of this lemma.
\end{proof}

So far, we have shown that the exponential supersolution that we are looking for at the beginning of this section indeed exists as long as
 $c\ge c_*$, and it is associated with a unique parameter $\alpha\in(0,\alpha_*]$.
Let us close this section  by showing that $c_*$ defined in \eqref{c*-def} is an upper bound of the asymptotic spreading speed for  problem \eqref{kppI} associated with $\rho_0\not\equiv0$ satisfying certain assumptions. 

\begin{proof}[Proof of statement (i) of Theorem \ref{thm-spreading property}] 
	Assume that {\bf (H1)}-{\bf (H2$\mu$)} and $\mathscr{R}_0>1$, and that $\rho$ is the solution  of  problem \eqref{fkpp} associated with $\rho_0\not\equiv 0$ such that {\bf (H4)} and \eqref{H3} are satisfied.
Let $c_*$ be defined in \eqref{c*-def}, we now show that the solution $\rho$ to the initial boundary value problem \eqref{fkpp} spreads at most with speed $c_*$ along any direction $\mathbf{e}\in\SS^{d-1}$.  Let $\alpha_*\in(0,\Lambda)$ be the unique value given in Lemma \ref{lemma_c*} such that $\varphi_{c_*}(\alpha_*)=1$. Then, up to shifts,
\begin{equation*}
\overline \rho(t,i,x)=\min\left( M,  e^{-\alpha_*(x\cdot \mathbf{e}-c_*t)-\alpha_* c_*i}\right)\pi(i), \quad t\geq0, \quad i\in[0,i_\dagger), \quad x\in\R^d,
\end{equation*}
 is a supersolution to problem \eqref{fkpp} satisfying $\rho_0(i,x)< \overline \rho(0,i,x)$ for $i\in[0,i_\dagger)$ and $x\in\R^d$.
The comparison principle Proposition~\ref{prop-cp} implies that 
$$\rho(t,i,x)\le \overline \rho(t,i,x)~~~\text{for}~(t,i,x)\in\R_+\times[0,i_\dagger)\times\R^d.$$
 Therefore, for any $c>c_*$ and for any $j\in(0,i_\dagger)$, we have
\begin{equation*}
	\underset{t\rightarrow+\infty}{\lim } \underset{|x\cdot \mathbf{e}|\geq ct,~ 0 \leq i \leq j}{\sup} \rho(t,i,x) =0.
\end{equation*}
Since $\mathbf{e}\in\SS^{d-1}$ is arbitrary, we then have
\begin{equation*}
	\underset{t\rightarrow+\infty}{\lim } \underset{\|x\|\geq ct,~ 0 \leq i \leq j}{\sup} \rho(t,i,x) =0.
\end{equation*}
This proves statement (i) of Theorem \ref{thm-spreading property}.
\end{proof}

\subsubsection{Compactly supported subsolution}
For the lower bound of the asymptotic spreading speed, we consider problem \eqref{fkpp} in the moving frame along any direction $\mathbf{e}\in\SS^{d-1}$ with speed $c<c_*$ and  $c\sim c_*$:
\begin{equation}
	\label{linearized pb_sub}
	\left\{\begin{split}
\partial_t \rho(t,i,x)&-c\,\mathbf{e}\cdot\nabla_x\rho(t,i,x) + \partial_i\rho(t,i,x) = - \gamma(i) \rho(t,i,x), \quad t>0,\quad i\in(0,i_\dagger), \quad x\in\R^d, \\
\rho(t,0,x)&=\S_0\left(1-\exp\left(- \int_0^\infty \tau(i) \K*\rho(t,i,x)\md i\right)\right), \quad t> 0,  \quad x\in\R^d.
\end{split}\right.
\end{equation}

\begin{lem}
	\label{lemma_existence of compactly subsol}
	For  $c<c_*$ such that  $c\sim c_*$, problem \eqref{linearized pb_sub} admits a compactly supported stationary subsolution.
\end{lem}

\begin{proof}[Proof of Lemma \ref{lemma_existence of compactly subsol}] We first treat the case when $i_\dagger<+\infty$ is finite. First of all, since by assumption $\mathscr{R}_0=\S_0\int_0^{i_\dagger}\omega(i)\md i>1$, there exists $R_0>\max(1,c_*i_\dagger)$ large enough, such that for any $R\geq R_0$ one has
\bqs
\widehat{\mathscr{R}}_0^R:=\S_0\int_0^{i_\dagger}\omega(i)\md i\int_{\R^d}\K^R(x) \md x>1.
\eqs 
For any direction $\mathbf{e}\in\mathbb{S}^{d-1}$ fixed, define
\bqs
\varphi^R_c(\alpha):=\S_0\int_{\R^d}\K^R(x)e^{\alpha x\cdot\mathbf{e}} \md x \int_0^{i_\dagger} \omega(i) e ^{-\alpha ci}\md i, \quad \text{for}~ c>0,~\alpha>0,
\eqs
Following the analysis of \eqref{dr} conducted in the previous section, we derive that for each $R\geq R_0$ there is a unique $c_*^R\in (0,c_*)$ and correspondingly a unique $\alpha_*^R\in(0,+\infty)$ such that $\varphi^R_{c_*^R}(\alpha_*^R)=1$ together with $\partial_\alpha \varphi^R_{c_*^R}(\alpha_*^R)=0$. 

Consider now $c\in\left(c_*-\frac{\pi^2}{R_0^4}, c_*\right)$. We can fix $R\geq R_0$, depending on $c$, such that $c_*^R=c+\frac{\pi^2}{R^3}$. This is always possible since $R\mapsto f(R):=c+\frac{\pi^2}{R^3}$ is a strictly decreasing function in $[R_0,\infty)$ with $f(R_0)=c+\frac{\pi^2}{R_0^3}>c_*$ and $f(+\infty)=c<c_*$  and that $R\mapsto g(R):=c_*^R$ is a strictly increasing function in $[R_0,\infty)$ with $g(R_0)=c_*^{R_0}<c_*$ and $g(+\infty)=c_*$. As a consequence, with such a choice for $R$, we have $c<c_*^R<c_*$. Since $\widehat{\mathscr{R}}_0^R>1$, we get the existence of $L_0\in(0,i_\dagger)$ and $L_0\sim i_\dagger$ such that
\bqs
\widehat{\mathscr{R}}_0^{L,R}:=\S_0\int_0^L\omega^L(i)\md i\int_{\R^d}\K^R(x) \md x>1,
\eqs 
where we have set $\omega^L(i):=\tau(i)e^{-\int_0^i \beta^L(s)\md s}$ for each $i\in[0,L)$ and $\beta^L: [0,L)\to \R_+$ is given in \eqref{beta-definition}. Repeating the previous step, we now obtain the existence of a unique $c_*^{L,R}\in (0,c_*)$ and correspondingly a unique $\alpha_*^{L,R}\in(0,+\infty)$ such that $\varphi^{L,R}_{c_*^{L,R}}(\alpha_*^{L,R})=1$ and $\partial_\alpha \varphi^{L,R}_{c_*^{L,R}}(\alpha_*^{L,R})=0$ where $ \varphi^{L,R}_c(\alpha)$ is defined as
\bqs
\varphi^{L,R}_c(\alpha):=\S_0\int_{\R^d}\K^R(x)e^{\alpha x\cdot\mathbf{e}} \md x \int_0^{L} \omega^L(i) e ^{-\alpha ci}\md i.
\eqs
Moreover, $c_*^{L,R}\to c_*^R$ and $\alpha_*^{L,R}\to \alpha_*^R$ when $L\to i_\dagger$. From now on, we choose $L\in(L_0,i_\dagger)$ such that $c<c_*^{L,R}<c_*^R<c_*$.

For future reference, we remark that for any $\eta>0$ sufficiently small, one can find $\eta'>0$ such that $(1-\eta)\widehat{\mathscr{R}_0}^{L,R} \geq 1+\eta'$.  Moreover, that there is $\delta_0>0$ small  such that
\bqq\label{ineqRLbis}
\S_0 \left(1-\exp\left(- (\widehat{\mathscr{R}_0}^{L,R}/\S_0) u \right)\right)> (1-\eta) \widehat{\mathscr{R}_0}^{L,R} u \geq (1+\eta')u  ~~~~\text{for}~u\in [0,\delta_0).
\eqq

With  $c<c_*^{L,R}<c_*^R<c_*$ and any direction $\mathbf{e}\in\mathbb{S}^{d-1}$ fixed, we look at the following truncated problem:
 \begin{equation}
 	\label{truncated pb}
 	\left\{
 		\begin{split}
 			-c\mathbf{e}\cdot\nabla_x\rho(i,x) &+ \partial_i\rho(i,x) = - \beta^L(i) \rho(i,x), \quad i\in[0,L),  \quad x\in\R^d,  \\
 			\rho(0,x)&=\S_0 \int_0^L \tau(i) \K^R*\rho(i,x)\md i,  \quad x\in\R^d.
 		\end{split}
 	\right.
 \end{equation}
We look for an exponential solution of \eqref{truncated pb} of the following form
 \begin{equation*}
 	e^{-\alpha x\cdot\mathbf{e}-\alpha ci-\int_0^i\beta^L(s)\md s}.
 \end{equation*}
 This amounts to finding $\alpha\in\mathbb{C}\backslash\R$ such that $\varphi_c^{L,R}(\alpha)=1$. To do so, set 
\begin{equation*}
	\Phi^{L,R}(c,\alpha):=1-\varphi_c^{L,R}(\alpha).
\end{equation*}
We notice that at the point $(c_*^{L,R},\alpha_*^{L,R})$, there holds $\Phi^{L,R}(c_*^{L,R},\alpha_*^{L,R})=0$. Since the function $\varphi^{L,R}_c(\alpha)$ is analytic in $c$ and in $\alpha$, so is $\Phi^{L,R}$. Therefore, we have 
\begin{align*}
	\partial_\alpha \Phi^{L,R}(c_*^{L,R},\alpha_*^{L,R})=0,~~-2\sigma&:=\partial_{\alpha\alpha}\Phi^{L,R}(c_*^{L,R},\alpha_*^{L,R})<0,~~r:=\partial_c \Phi^{L,R}(c_*^{L,R},\alpha_*^{L,R})>0,\\
	~~b&:=\partial_{c\alpha}\Phi^{L,R}(c_*^{L,R},\alpha_*^{L,R}).
\end{align*}
Set
\begin{equation*}
	\xi:=c_*^{L,R}-c>0,~~~~\tau:=\alpha-\alpha_*^{L,R}.
\end{equation*}
We restrict ourselves to a vicinity of $(c_*^{L,R},\alpha_*^{L,R})$ and rewrite $\Phi^{L,R}(c,\alpha)=0$ by expanding $\Phi^{L,R}$  at  $(c_*^{L,R},\alpha_*^{L,R})$ for $(c,\alpha)$ in this vicinity:
\begin{align*}
	0=\Phi^{L,R}(c,\alpha)&=\Phi^{L,R}(c_*^{L,R},\alpha_*^{L,R})+\big[(c-c_*^{L,R})\partial_c+(\alpha-\alpha_*^{L,R})\partial_\alpha\big] \Phi^{L,R}(c_*^{L,R},\alpha_*^{L,R})\\
	&~~~+\frac{1}{2}\big[(c-c_*^{L,R})\partial_c+(\alpha-\alpha_*^{L,R})\partial_\alpha\big]^2\Phi^{L,R}(c_*^{L,R},\alpha_*^{L,R})\\
	&~~~+\frac{1}{3!}\big[(c-c_*^{L,R})\partial_c+(\alpha-\alpha_*^{L,R})\partial_\alpha\big]^3\Phi^{L,R}(c_*^{L,R},\alpha_*^{L,R})+...
\end{align*}
that is,
$$\sigma\tau^2+b\xi\tau+\xi r=\varphi(\tau,\xi),$$
where $\varphi(\tau,\xi)$ is analytic in $\tau$ and in $\xi$ for $(\tau,\xi)$ in a small neighborhood of $(0,0)$,  $\varphi(\tau,\xi)$  is of the order $\xi^2+|\tau|^3$ and vanishes at $(0,0)$. 

We observe that, for $\xi>0$ small enough, the trinomial $\sigma\tau^2+b\xi\tau+\xi r$ has a pair of complex roots $\tau_\pm=\pm \mathbf{i}\sqrt{(r/\sigma)\xi}+\mathcal{O}(\xi)$. By the Rouch\'{e} Theorem, we obtain that for $\xi>0$ small enough, $\Phi^{L,R}(c,\alpha)=0$ admits a pair of complex roots close (of the order $\xi$) to $\tau_+$ and $\tau_-$ respectively, therefore still denoted by $\tau_\pm$ having the form
\begin{equation*}
	\tau_\pm=\pm \mathbf{i}(\sqrt{(r/\sigma)\xi}+\mathcal{O}(\xi))+\mathcal{O}(\xi).
\end{equation*} 
Consequently, the dispersion relation $\Phi^{L,R}(c,\alpha)=1-\varphi_c^{L,R}(\alpha)$ has complex roots
\begin{equation*}
	\alpha=\alpha_*^{L,R}+\tau_\pm=\alpha_*^{L,R}+\mathcal{O}(\xi)\pm \mathbf{i}(\sqrt{(r/\sigma)\xi}+\mathcal{O}(\xi))=:\alpha_1\pm \mathbf{i}\alpha_2, 
\end{equation*}
with $\alpha_1>0$ such that $\alpha_1-\alpha_*^{L,R}$ being of the order $\xi$, and $\alpha_2\neq 0$ of the order $\sqrt\xi$. And we further note that 
\bqs
\alpha_2\sim \sqrt{c_*^{L,R}-c}<\sqrt{c_*^R-c}=\frac{\pi}{R^{3/2}},
\eqs
and thus up to increasing initially $R_0$, we can always ensure that $\alpha_2<\frac{\pi}{2R}$.

A direct consequence is that $\mathrm{Re}\left(e^{-\alpha x\cdot\mathbf{e}-\alpha ci-\int_0^i\beta^L(s)\md s}\right)$ is also a solution of \eqref{truncated pb}, thus
\bqq\label{real part}
\S_0 \int_0^L\omega^L(i) \K_0^R*\left[e^{-\alpha_1(x\cdot\mathbf{e}+ci))}\cos(\alpha_2(x\cdot\mathbf{e}+ci)) \right]\md i = e^{-\alpha_1x\cdot\mathbf{e}}\cos(\alpha_2(x\cdot\mathbf{e})), \quad x\in\R^d.
\eqq

We now define
\begin{align*}
	\psi_{\alpha_1,\alpha_2}(z)=\begin{cases}
		e^{-\alpha_1 z}\cos(\alpha_2 z ),~~~~&|z|< \frac{\pi}{2\alpha_2},\\
		0,&\text{elsewhere}.
	\end{cases}
\end{align*}
We claim that 
\bqs
\S_0 \int_0^L\omega^L(i) \K_0^R*\psi_{\alpha_1,\alpha_2}(z+ci)\md i \geq \psi_{\alpha_1,\alpha_2}(z), \quad z\in\R.
\eqs
Indeed, using the definition of $\psi_{\alpha_1,\alpha_2}$, for any $|z|< \frac{\pi}{2\alpha_2}$ we have
\begin{align*}
\S_0 \int_0^L\omega^L(i) \K_0^R*\psi_{\alpha_1,\alpha_2}(z+ci)\md i&=\S_0 \int_0^L\omega^L(i)  \int_{-\frac{\pi}{2\alpha_2}}^{\frac{\pi}{2\alpha_2}}  \K_0^R(z+ci-y) e^{-\alpha_1 y}\cos(\alpha_2 y) \md y  \md i\\
& \geq \S_0 \int_0^L\omega^L(i)  \int_{\R}  \K_0^R(z+ci-y) e^{-\alpha_1 y}\cos(\alpha_2 y) \md y  \md i\\
&\overset{\eqref{real part}}{=}e^{-\alpha_1 z}\cos(\alpha_2 z )=\psi_{\alpha_1,\alpha_2}(z).
\end{align*}
The first inequality above holds thanks to our choice of $\alpha_2\le \frac{\pi}{2R}$. Indeed, since $\mathrm{supp}(\K_0^R)\subset[-R,R]$, for each $i\in[0,L]$, when $|z+ci-y|\leq R$, it follows from $|z|< \frac{\pi}{2\alpha_2}$ that necessarily $|y| \leq R+ci + \frac{\pi}{2\alpha_2}$ for all $i\in[0,L]$. Due to $R+ci\leq R+c_*i_\dagger<2R$, on then has $-\frac{3\pi}{2\alpha_2}\leq y\leq \frac{3\pi}{2\alpha_2}$ as long as   $\alpha_2\le \frac{\pi}{2R}$, which implies that the above integral for $y\in\R\backslash [-\frac{\pi}{2\alpha_2}, \frac{\pi}{2\alpha_2}]$ is nonpositive, hence the inequality follows.
Thus for any $\eta'>0$, we can find $\iota>0$ small enough such that
\bqq
\label{eqkey11}
\S_0 \int_0^L\omega^L(i) \K_0^R*\psi_{\alpha_1,\alpha_2}(z+ci)\md i \geq \frac{1}{1+\eta'} \psi_{\alpha_1,\alpha_2}(z-\iota), \quad z\in\R.
\eqq
We denote by $\hat z\in (-\frac{\pi}{2\alpha_2},0 )$ the value at which $\psi_{\alpha_1,\alpha_2}$ achieves its maximum on $\R$ 
and set
\bqs
m:=e^{-\alpha_1 \hat z}\cos(\alpha_2 \hat z ).
\eqs

Next, we define a modified version of $\psi_{\alpha_1,\alpha_2}$ as
\bqs 
\underline{\psi}_{\alpha_1,\alpha_2}(z)=
\left\{
\begin{array}{cc}
m, & z\leq  \hat z, \\
e^{-\alpha_1z}\cos(\alpha_2 z),& \quad \hat z \leq z \leq \frac{\pi}{\alpha_2}, \\
0,&\text{elsewhere},
\end{array}\right.
\eqs
which equivalently reads $\underline{\psi}_{\alpha_1,\alpha_2}(z)=\max_{y\geq 0} \psi_{\alpha_1,\alpha_2}(z+y)$. Finally, we set
\bqq\label{subsol-moving frame} 
\Psi(i,x)=
\left\{
\begin{array}{cc}
e^{-\int_0^i\beta^L(s)\md s}\underline{\psi}_{\alpha_1,\alpha_2}(\|x\|+ci-D), & (i,x)\in[0,L)\times\R^d,\\
0,&\text{elsewhere},
\end{array}\right.
\eqq
for some $D>R+cL+|\hat z|\geq R_0$ that will be fixed below. We now verify that $\delta \Psi$ for any $\delta\in(0,\delta_0/m)$ is a subsolution to \eqref{linearized pb_sub}. Once again, only the boundary condition needs to be checked.  We first compute
\begin{align*}
\S_0&\left(1-\exp\left(-\delta \int_0^\infty\tau(i)\K*\Psi(i,x)\md i\right)\right)\\
&\geq \S_0\left(1-\exp\left(-\delta \int_0^L\omega^L(i) \int_{B_R(0)} \K(y)\underline{\psi}_{\alpha_1,\alpha_2}\Big(\|x-y\|+ci-D\Big)\md y \md i \right)\right).
\end{align*}
For $\|x\| \leq D-R-cL+\hat z$, we have that $\|x-y\|+ci-D \leq \hat z$ for all $\|y\|\leq R$. As a consequence, we get 
\begin{align*}
\S_0\left(1-\exp\left(-\delta \int_0^\infty\tau(i)\K*\Psi(i,x)\md i\right)\right) &\geq \S_0\left( 1-\exp\left(-\left(\widehat{\mathscr{R}_0}^{L,R}/\S_0\right) \delta  m\right)\right)\\
&\overset{\eqref{ineqRLbis}}{\geq} (1-\eta)\widehat{\mathscr{R}_0}^{L,R} \delta  m \geq (1+\eta') \delta  m \geq \delta  \Psi(0,x).
\end{align*}
Next, for $D-R-cL+\hat z\leq \|x\| \leq D+\frac{\pi}{\alpha_2}$, we have that, for any $\|y\|\leq R$, 
\bqs
\|x-y\|-D\leq \|x\| -\frac{x\cdot y}{\|x\|} +\frac{R^2}{2(D-R-cL+\hat z)}
\eqs
and so we select $D$ large enough such that $\frac{R^2}{2(D-R-cL+\hat z)}<\iota$ with $\iota$ defined in \eqref{eqkey11}. We then obtain 
\begin{align*}
\int_0^L\omega^L(i) &\int_{B_R(0)} \K(y)\underline{\psi}_{\alpha_1,\alpha_2}\Big(\|x-y\|+ci-D\Big)\md y \md i \\
&\geq \int_0^L\omega^L(i) \int_{B_R(0)} \K(y)\max_{\mu \geq -D }\psi_{\alpha_1,\alpha_2}\Big(\|x\|+ci-\frac{x\cdot y}{\|x\|} +\iota + \mu \Big)\md y \md i\\
&=\max_{\mu \geq -D } \int_0^L\omega^L(i) \K_0^R*\psi_{\alpha_1,\alpha_2}\Big(\|x\|+ci+\iota+\mu \Big)\md i \\
&\overset{\eqref{eqkey11}}{\geq} \frac{1}{\S_0(1+\eta')} \max_{\mu \geq -D }\psi_{\alpha_1,\alpha_2}\left(\|x\|+\mu \right)=\frac{1}{\S_0(1+\eta')} \Psi(0,x).
\end{align*}
Thus, in that range we get
\begin{align*}
\S_0\left(1-\exp\left(-\delta \int_0^\infty\tau(i)\K*\Psi(i,x)\md i\right)\right) \geq \S_0\left( 1-\exp\left(-\frac{\delta }{\S_0(1+\eta')} \Psi(0,x) \right)\right) \geq \delta  \Psi(0,x).
\end{align*}
This completes the proof of Lemma \ref{lemma_existence of compactly subsol} in the case where $i_\dagger<+\infty$.  When $i_\dagger=+\infty$, since $\mathscr{R}_0>1$, there exists $R_0>1$ such that $\widetilde{\mathscr{R}}_0^{R}:=\S_0\int_0^R\omega^R(i)\md i\int_{\R^d}\K^R(x) \md x>1$ for $R\geq R_0$. Moreover, for each $R\ge R_0$, there exist $\widetilde c_*^R\in(0,c_*)$ and $\widetilde \alpha_*^R>0$ such that $\widetilde{\varphi}_{\widetilde c_*^R}^R(\widetilde \alpha_*^R)=1$ and $\partial_\alpha \widetilde{\varphi}_{\widetilde c_*^R}^R(\widetilde \alpha_*^R)=0$, with $\widetilde{\varphi}_c^R(\alpha)$ defined as
\bqs
\widetilde{\varphi}^R_c(\alpha):=\S_0\int_{\R^d}\K^R(x)e^{\alpha x\cdot\mathbf{e}} \md x \int_0^{R} \omega^R(i) e ^{-\alpha ci}\md i \quad\text{for}~ c>0, ~ \alpha>0.
\eqs
Now we consider any $c\in\left(c_*-\frac{\pi^2}{R_0^4}, c_*\right)$. We can fix $R\geq R_0$, depending on $c$, such that $\widetilde{c}_*^R=c+\frac{\pi^2}{R^3}$. From there, we can repeat the analysis with $L=R$ and still construct a compactly supported subsolution of the form \eqref{subsol-moving frame}.
\end{proof}

Consider now the initial boundary value problem corresponding to \eqref{fkpp}  in the moving frame, namely,
\bqq
\label{kppI-moving}
\left\{
\begin{split}
\partial_t \rho(t,i,x)&-c\, \mathbf{e}\cdot\nabla_x\rho(t,i,x) + \partial_i\rho(t,i,x) = - \gamma(i) \rho(t,i,x) ,\quad t>0, \quad i\in(0,i_\dagger),  \quad x\in\R^d,  \\
\rho(t,0,x) &=\S_0 \left(1-\exp\left(-\int_0^\infty \tau(i) \K*\rho(t,i,x)\md i \right)\right), \quad  t> 0,  \quad x\in\R^d,\\
\rho(0,i,x)&=\rho_0(i,x),   \quad     i\in [0,i_\dagger), \quad x\in\R^d,
\end{split}\right.
\eqq
with initial condition  $\rho_0\not\equiv 0$ such that {\bf (H4)} and \eqref{H3} are satisfied. That is, the initial condition $\rho_0$, defined on $[0,i_\dagger)\times\R^d$, is nonnegative, compactly supported, absolutely continuous on $[0,i_\dagger)$ uniformly with respect to its second variable and such that $\rho_0/\pi \in L^\infty([0,i_\dagger)\times\R^d)$. Moreover, condition \eqref{H3} is satisfied for $\rho_0$.


\begin{lem}\label{lemma_longtime beha-moving}
	Assume that {\bf (H1)}-{\bf (H2$\mu$)} and  $\mathscr{R}_0>1$, for  $c<c_*$ satisfying  $c\sim c_*$, let  $\hat\rho$  be the solution of  problem \eqref{kppI-moving} associated with $\rho_0\not\equiv 0$ such that {\bf (H4)} and \eqref{H3} are satisfied. Then,
		\begin{equation}
			\label{eqn_moving}
			\hat \rho(t,i,x) \rightarrow \rho^s(i) \text{ as } t\rightarrow +\infty,~~~\text{locally uniformly in}~(i,x)\in[0,i_\dagger)\times \R^d,
		\end{equation}
	 where  $\rho^s$ is given in Theorem \ref{thm-liouville}. 
\end{lem}

\begin{proof}[Proof of Lemma \ref{lemma_longtime beha-moving}]
We simply note that $\hat \rho(t,i,x)=\rho(t,i,x+ct\mathbf{e})$ where $\rho$ is the solution to the  problem~\eqref{fkpp} and that $\partial_{x_p}\hat \rho$ is well-defined thanks to our running assumptions on $\rho_0$ and the result of Proposition~\ref{prop-edpI-well-posedness}. Now, adapting the proof of Theorem \ref{thm-long time behavior}(ii) with $\psi$ replaced by $\Psi$ given in \eqref{subsol-moving frame} to \eqref{kppI-moving}, we obtain the conclusion. 
\end{proof}

We are now in a position to justify  that $c_*$, defined in Lemma \ref{lemma_c*}, is also a lower bound of the asymptotic spreading speed for the initial boundary value problem \eqref{fkpp}.

\begin{proof}[Proof of statement (ii) in Theorem \ref{thm-spreading property}.]
	Assume that {\bf (H1)}-{\bf (H2$\mu$)} and $\mathscr{R}_0>1$, and that $\rho$ is the solution  of  problem \eqref{fkpp} associated with $\rho_0\not\equiv 0$ such that {\bf (H4)} and \eqref{H3} are satisfied, and let $\hat \rho$ be the solution to  problem \eqref{kppI-moving} with the same initial function $\rho_0$.
	
First of all, we obtain from Lemma \ref{lemma_longtime beha-moving} that, for any $\varep>0$ sufficiently small, there is $c\in(c_*-\varep,c_*)$ such that  the solution $\hat \rho$  of \eqref{kppI-moving} satisfies
$$	\hat \rho(t,i,x) \rightarrow \rho^s(i) ~~~\text{ as } t\rightarrow +\infty,~~~\text{locally uniformly in}~(i,x)\in[0,i_\dagger)\times \R^d.~~~~~$$
This is equivalent to the following
\begin{equation}
	\label{eqn-1}
	\rho(t,i,x+ct\mathbf{e}) \rightarrow \rho^s(i)~~~ \text{ as } t\rightarrow +\infty,~~~\text{locally uniformly in}~(i,x)\in[0,i_\dagger)\times \R^d.
\end{equation}

To complete the proof of Theorem \ref{thm-spreading property}, we need to the following auxiliary lemma.

\begin{lem}
	\label{lemma_c^+-}
Assume that  {\bf (H1)}-{\bf (H2$\mu$)} and $\mathscr{R}_0>1$. Let $c^-<c^+$ be such that any nonnegative nontrivial bounded solution $\rho$ to  problem \eqref{fkpp}, associated with $\rho_0\not\equiv 0$ such that {\bf (H4)} and \eqref{H3} are satisfied, has the property that, along any direction $\mathbf{e}\in\SS^{d-1}$,
\begin{equation}\label{eqn-2}
		\rho(t,i,x+c^\pm t\mathbf{e}) \rightarrow \rho^s(i)~~~ \text{ as } t\rightarrow +\infty,~~~\text{locally uniformly in}~(i,x)\in[0,i_\dagger)\times \R^d.
\end{equation}
Then, there holds
\begin{equation}
	\label{eqn-3}
	\forall j\in(0,i_\dagger),~~~\lim_{t\to+\infty}\sup_{c^-t\le x\cdot \mathbf{e}\le c^+ t,~
	0\le i\le j} \left|\rho(t,i,x)-\rho^s(i)\right|=0.
\end{equation}
\end{lem}

Let us postpone the proof of Lemma \ref{lemma_c^+-} for the moment  and continue the proof of Theorem \ref{thm-spreading property}.
By taking $c^-=0$ and $c^+=c$ in Lemma \ref{lemma_c^+-}, along any direction $\mathbf{e}\in\SS^{d-1}$,  the function $\rho$ satisfies
$$	\forall j\in(0,i_\dagger),~~~\lim_{t\to+\infty}\sup_{0\le x\cdot \mathbf{e}\le c t,~
	0\le i\le j} \left|\rho(t,i,x)-\rho^s(i)\right|=0.$$
Since $\mathbf{e}\in\SS^{d-1}$ is arbitrarily chosen, it follows that
$$	\forall j\in(0,i_\dagger),~~~\lim_{t\to+\infty}\sup_{\|x\|\le c t,~
	0\le i\le j} \left|\rho(t,i,x)-\rho^s(i)\right|=0.$$
This completes the proof of Theorem \ref{thm-spreading property}.
\end{proof}

Let us now complete this section with the proof of Lemma \ref{lemma_c^+-}.
\begin{proof}[Proof of Lemma \ref{lemma_c^+-}]
	Assume that {\bf (H1)}-{\bf (H2$\mu$)} and $\mathscr{R}_0>1$.
Let $\rho$ be the solution of problem \eqref{fkpp} with $\rho_0\not\equiv 0$ such that {\bf (H4)} and \eqref{H3} are satisfied. Let $\overline \rho$ denote  the solution to \eqref{fkpp} with initial condition $\overline\rho_0=\max\left(\S_0,\Vert \rho_0/\pi\Vert_{L^\infty([0,i_\dagger)\times\R^d)}\right)\pi$ and  let $\underline\rho$ be the solution to \eqref{fkpp} with initial datum $\underline \rho_0=\delta\psi$ in $[0,i_\dagger)\times\R^d$ with $\psi$ given by \eqref{sub_kppI} for some small $\delta>0$ such that $\delta\psi$ is a subsolution to problem \eqref{fkpp}, as analyzed in the proof of Statement (ii)  of Theorem \ref{thm-long time behavior}.  Since $\rho(t,i,x)>0$ for $(t,i,x)\in(0,+\infty)\times[0,i_\dagger)\times\R^d$ with $t>i$ due to Proposition~\ref{prop-positivity}, we can decrease $\delta$ (if necessary) such that  $0\le \delta\psi\le \rho(T,\cdot,\cdot)$ in $[0,i_\dagger)\times\R^d$ for some $T>L$, where $L\in(0,i_\dagger)$  is associated with $\psi$. By the comparison principle Proposition~\ref{prop-cp}, we then deduce that
\begin{equation*}
	\underline\rho(t,i,x)\le \rho(t+T,i,x)\le \overline\rho(t+T,i,x)~~~\text{for}~(t,i,x)\in\R_+\times[0,i_\dagger)\times\R^d.
\end{equation*}
Thanks to the assumption \eqref{eqn-2}, we have, 
for any $\epsilon>0$ small and for each $j\in(0,i_\dagger)$, there is $T_1>T>0$ large enough such that 
\begin{equation}
	\label{eqn-4}
	\forall t\ge T_1,~0\le i\le j,~~~\left|\underline\rho(t,i,c^-t \mathbf{e})-\rho^s(i)\right|<\epsilon, \text{ and } \left|\overline\rho(t,i,c^-t \mathbf{e})-\rho^s(i)\right|<\epsilon.
\end{equation}
Moreover, letting $\omega\ge |c^-|T_1$, we also deduce from \eqref{eqn-2} that there is $T'>T_1>0$ such that
\begin{equation}
	\label{eqn-5}
	\forall t\ge T',~|x\cdot \mathbf{e}|\le\omega,~0\le i\le j,~~~\left|\rho(t,i,x+c^+t \mathbf{e})-\rho^s(i)\right|<\epsilon.
\end{equation}
Set $c:=(1-\lambda)c^-+\lambda c^+$ for any $\lambda\in[1/2,1]$, and fix $\tau\ge 2T'$, we now claim that 
\begin{equation}\label{eqn-claim}
	|\rho(\tau,i,c\tau \mathbf{e})-\rho^s(i)|<\epsilon~~~\text{for}~0\le i\le j.
\end{equation}
We divide into two subcases.

\noindent
\textbf{Case I}. $(1-\lambda)\tau\le T_1$. In this case, we observe that $c\tau \mathbf{e}=\big((1-\lambda)c^-+\lambda c^+\big)s \mathbf{e}=(1-\lambda)c^-\tau e+\lambda c^+\tau \mathbf{e}=:x+c^+t \mathbf{e}$ with $|x\cdot \mathbf{e}|=|(1-\lambda)c^-\tau|\le |c^-|T_1\le \omega$ and $t=\lambda \tau\ge T'$. 
Then \eqref{eqn-claim} immediately follows from \eqref{eqn-5}.

\noindent
\textbf{Case II}. $(1-\lambda)\tau >T_1$. Up to increasing  $T_1$, we have the following comparison
\begin{equation*}
	\underline \rho_0=\delta\psi<\rho(\lambda \tau,i, x+c^+\lambda \tau \mathbf{e})\le \overline \rho_0~~~\text{for}~(i,x)\in[0,i_\dagger)\times\R^d.
\end{equation*}
By applying the comparison principle Proposition~\ref{prop-cp}, we have 
\begin{equation*}
	\underline\rho((1-\lambda)\tau,i, x+c^-(1-\lambda)\tau \mathbf{e})\le \rho(\tau,i,x+c\tau \mathbf{e})\le\overline\rho((1-\lambda)\tau,i,x+c^-(1-\lambda)\tau).
\end{equation*}
Since $(1-\lambda)\tau>T_1$, we observe that \eqref{eqn-claim} can be reached by  \eqref{eqn-4}.

Consequently, we conclude from \eqref{eqn-claim} that 
\begin{equation*}
\forall t\ge \tau,~~~~~~\sup_{(c^-+c^+)t/2\le x\cdot \mathbf{e}\le c^+ t,~0\le i\le j}	|\rho(t,i,x)-\rho^s(i)|<\epsilon.
\end{equation*}

On the other hand, set $c:=(1-\lambda)c^++\lambda c^-$ for  any $\lambda\in[1/2,1]$. By repeating the analysis as above, we will get 
\begin{equation*}
	\forall t\ge \tau,~~~~~~\sup_{c^- t\le x\cdot \mathbf{e}\le (c^-+c^+)t/2,~0\le i\le j}	|\rho(t,i,x)-\rho^s(i)|<\epsilon.
\end{equation*}
Consequently, the proof of Lemma \ref{lemma_c^+-} is complete.
\end{proof}

\begin{rmk}
For compactly supported kernels, condition {\bf (H2$\mu$)} is automatically satisfied with $\Lambda=+\infty$. As a consequence, in that case, upon assuming condition~\eqref{eqrmkpos} instead on the initial condition, one naturally obtains that Theorem~\ref{thm-spreading property} remains valid in that case. \end{rmk}

\section{Proofs of the main results}\label{secMain}

Throughout this section, we turn to problem \eqref{edp} with  nontrivial compact perturbation $I_0\not\equiv0$ and general nonnegative compactly supported initial condition $\rho_0$ in $[0,i_\dagger)\times\R^d$, and investigate the long time behavior and further spreading property of the solution to \eqref{edp}.

\subsection{Liouville-type result -- Proof of Theorem \ref{thm_edpI-Liouville}}

Let us prove the Liouville-type result for the stationary problem \eqref{edpI-stationary}.


\begin{proof}[Proof of Theorem \ref{thm_edpI-Liouville}] Assume {\bf(H1)-(H2)} and  that $I_0\not\equiv0$ satisfies {\bf(H3)}. We divide the proof into three steps. We first prove the existence of a nonnegative nontrivial stationary solution, then establish its uniqueness and finally study its asymptotic behavior as $\|x\|\rightarrow+\infty$.
	
	\noindent
\textbf{Existence.} First of all, due to the assumption {\bf (H3)} on $I_0\not\equiv0$, we get the existence of a constant $A>0$ such that
\bqs
I_0(i,x) \leq A \gamma(i) \quad\text{for}~ (i,x)\in \mathrm{supp}(I_0).
\eqs
Now, set $M:=\max\left(\S_0,A\right)$, and define $\overline{\rho}_0(i,x)=M\pi(i)$ for each $(i,x)\in[0,i_\dagger)\times\R^d$. Then, it is easy to check that $\overline{\rho}_0$ is a supersolution to \eqref{edpI-stationary} in $[0,i_\dagger)\times\R^d$. On the other hand,  let $\psi$ be given by \eqref{sub_kppI} for some small $\delta>0$ such that $\delta\psi$ is a subsolution to problem \eqref{edpI-stationary} satisfying $\delta\psi<M\pi$ in $[0,i_\dagger)\times\R^d$.  Let $\underline \rho$ (resp. $\overline \rho$) be the solution of problem \eqref{edp} associated with initial condition $\underline \rho_0=\delta\psi$ (resp. $\overline \rho_0=M\pi$) in $[0,i_\dagger)\times\R^d$. By the comparison principle Proposition \ref{prop-cp}, it follows that $\underline \rho\le\overline \rho$ in $\R_+\times[0,i_\dagger)\times\R^d$. Moreover, $\underline \rho$ is nondecreasing in time for $(t,i,x)\in\R_+\times[0,i_\dagger)\times\R^d$, whereas $\overline \rho$ is nonincreasing in time for $(t,i,x)\in\R_+\times[0,i_\dagger)\times\R^d$. By the monotone convergence theorem and then by the Dini's theorem,  we eventually obtain
	that $\underline \rho$ (resp. $\overline \rho$) converges, locally uniformly for $(i,x)\in[0,i_\dagger)\times\R^d$, as $t\to+\infty$ to a  solution $\underline U$ (resp. $\overline U$) of the stationary problem \eqref{edpI-stationary} such that for $(i,x)\in[0,i_\dagger)\times\R^d$,
	\begin{equation}
		\label{eqn}
		0\le \delta\psi(i,x)\le  \underline U(i,x)\le \overline  U(i,x)\le M \pi(i).
	\end{equation}
	   This gives the existence of nontrivial nonnegative solutions to the stationary problem \eqref{edpI-stationary}  in $[0,i_\dagger)\times\R^d$.
	
	\noindent
\textbf{Uniqueness.}  Let now $U$ be a nontrivial nonnegative solution of \eqref{edpI-stationary} in $[0,i_\dagger)\times\R^d$. From the first equation of \eqref{edpI-stationary}, it follows that 
	\begin{equation}\label{4.1'}
		\frac{U(i,x)}{\pi(i)}=U(0,x)+\int_0^i \frac{I_0(\xi,x)}{\pi(\xi)}\md \xi~~~~\text{for}~(i,x)\in[0,i_\dagger)\times\R^d.
	\end{equation}
	Plugging it into the boundary condition of \eqref{edpI-stationary}, one then derives that
	\begin{equation*}
		U(0,x)=\S_0\left(1-\exp\left(-\left(\int_0^\infty\omega(i)\md i\right) \K*U(0,x)-\int_0^\infty \omega(i)\K*\left(\int_0^i \frac{I_0(\xi,x)}{\pi(\xi)}\md \xi \right)\md i  \right)\right), \quad x\in\R^d.
	\end{equation*}
	Following the idea in the proof of Theorem \ref{thm-liouville}, we set $$\widehat\varphi(x)=\frac{U(0,x)}{\S_0},~~~~~x\in\R^d.$$ 
	By recalling that $\mathscr{R}_0=\S_0\int_0^\infty\omega(i)\md i$, it follows that the function $\widehat \varphi$ satisfies
	\begin{equation}\label{hat varphi}
		\widehat \varphi(x)=1- \mathcal{A}(x) e^{-\mathscr{R}_0\K*\widehat \varphi(x)}:=\mathcal{N}(\widehat \varphi(x),x), \quad x\in\R^d,
	\end{equation}
where  $\mathcal{A}(x):=\exp\left(-\int_0^\infty \omega(i)\K*\left(\int_0^i \frac{I_0(\xi,x)}{\pi(\xi)}\md \xi\right)\md i \right)$ takes values in (0,1).  Using the concavity of the mapping $\widehat \varphi \in \X \mapsto \mathcal{N}(\widehat \varphi,x)$ for each $x\in\R^d$ with $\X=\left\{ \varphi\in\mathscr{C}(\R^d)~|~ \varphi\geq0\right\}$, together with the fact that $0<\mathcal{N}(0,x)<1$ and $\mathcal{N}(+\infty;x)=1$ for $x\in\R^d$, it follows that  \eqref{hat varphi} admits a unique positive solution $0<\widehat\varphi <1$. This then implies that $U(0,x)>0$, whence $U(i,x)>0$ thanks to \eqref{4.1'}.

It is worth to notice from the above equation \eqref{hat varphi} that $\widehat \varphi(x)\ge \varphi(x)$ for $x\in\R^d$, where we recall that  $\varphi$ solves
\begin{equation*}
	\varphi(x)=1-e^{-\mathscr{R}_0\K*\varphi(x)}, \quad x\in\R^d.
\end{equation*}
This implies in particular that, when $\mathscr{R}_0>1$, there holds $U>\rho^s$ in $[0,i_\dagger)\times\R^d$, where $\rho^s$ is the unique positive stationary solution of \eqref{fkpp} given in Theorem \ref{thm-liouville}.  

	\noindent
\textbf{Asymptotic behavior.} Let $U$ be the unique positive solution to \eqref{edpI-stationary} which is uniformly bounded on $[0,i_\dagger)\times\R^d$. For any sequence $(x_n)_{n\in\N}$ in $\R^d$ such that $\|x_n\|\to+\infty$ as $n\to+\infty$, let us consider the function 
\begin{equation*}
	U_n(i,x):=U(i,x+x_n) \quad \text{for}~(i,x)\in[0,i_\dagger)\times\R^d.
\end{equation*}
We observe that $U_n$ satisfies
\begin{equation*}
	\left\{
		\begin{split}
			 \partial_iU_n(i,x) &=I_0(i,x+x_n) - \gamma(i) U_n(i,x) ,\quad i\in(0,i_\dagger), \quad x\in\R^d, \\
			U_n(0,x)&=\S_0 \left(1-\exp\left(-\int_0^\infty \tau(i) \K*U_n(i,x)\md i \right)\right), \quad  x\in\R^d.
		\end{split}
	\right.
\end{equation*}
Passing to the limit as $n\to+\infty$, we have $U_n\to U_\infty$ as $n\to+\infty$ locally uniformly, where $U_\infty$ is the solution of  
\begin{equation*}
	\left\{
		\begin{split}
			 \partial_iU_\infty(i,x) &= - \gamma(i) U_\infty(i,x) ,\quad i\in(0,i_\dagger), \quad x\in\R^d, \\
			U_\infty(0,x)&=\S_0 \left(1-\exp\left(-\int_0^\infty \tau(i) \K*U_\infty(i,x)\md i \right)\right), \quad  x\in\R^d,
		\end{split}
	\right.
\end{equation*}
which is exactly the stationary problem of the homogeneous model \eqref{fkpp}.
 By virtue of Theorem \ref{thm-liouville}, one infers that 	\begin{align*}
 	\lim\limits_{\|x\|\to+\infty}U(i,x)=\begin{cases}
 		0,&\text{if}~\mathscr{R}_0\le 1,\\
 		\rho^s(i), &\text{if}~\mathscr{R}_0>1,
 	\end{cases}
 	~~	\text{locally uniformly in}~i\in[0,i_\dagger).
 \end{align*}
 This finishes the proof of Theorem \ref{thm_edpI-Liouville}. 
\end{proof}

Next, we provide a asymptotic property of the positive stationary solution $U$ to \eqref{edp} by further assuming that the kernel is exponentially localized, i.e., satisfies {\bf(H2$\mu$)}, which illustrates that the function $U$ will exponentially  approach the stationary solution of the homogeneous model \eqref{kppI} when $\|x\|$ is sufficiently large. Our result is the following.
\begin{prop}
	\label{prop_edpI_U}
	Assume {\bf(H1)-(H2$\mu$)} and  that $I_0\not\equiv0$ satisfies {\bf(H3)}.
	Let $U$ be the unique positive stationary solution of \eqref{edp}, given in Theorem \ref{thm_edpI-Liouville}. Then  there is some $\lambda=\lambda(\|x\|)>0$ such that  $U$ satisfies,  for $i\in[0,i_\dagger)$ and for $x\in\R^d$ with $\|x\|$ sufficiently large, 
	\begin{align*}
		U(i,x)=\begin{cases}
			\rho^s(i)+\S_0 e^{-\lambda \|x\|}\pi(i), ~~~&\mathscr{R}_0>1,\\
			\S_0e^{-\lambda\|x\|}\pi(i)~~& \mathscr{R}_0\le 1.
		\end{cases}
	\end{align*}
\end{prop}
\begin{proof}[Proof of Proposition \ref{prop_edpI_U}]
	Let $U$ be the unique positive stationary solution of \eqref{edp}, given in Theorem \ref{thm_edpI-Liouville}. Let $\rho^s$ be the unique positive stationary solution of \eqref{fkpp} when $\mathscr{R}_0>1$. With a slight abuse of notation, here we set $\rho^s\equiv0$ when  $\mathscr{R}_0\le 1$. We now prove the conclusion for $\mathscr{R}_0>1$ and $\mathscr{R}_0\le 1$ simultaneously.

For any direction $\mathbf{e}\in\SS^{d-1}$,
we consider $x\in\R^d$ with $x\cdot\mathbf{e}$ sufficiently large (the case that $x\cdot\mathbf{e}$ sufficiently negative can be dealt with similarly through taking direction $-\mathbf{e}$) and $i\in[0,i_\dagger)$. To reach our conclusion, it suffices to prove that there is $\lambda>0$ (depending on $\|x\|$) such that  for $i\in[0,i_\dagger)$ and for $x\in\R^d$ with $\|x\|$ sufficiently large, the following ansatz makes sense:
\begin{equation*}
	U(i,x)=\rho^s(i)+\S_0 e^{-\lambda x\cdot \mathbf{e}}\pi(i),\quad (i,x)\in[0,i_\dagger)\times \R^d.
\end{equation*}
First of all, we notice that for $i\in[0,i_\dagger)$ and for $x\in\R^d$ with $\|x\|$ sufficiently large, the term $I_0$ is identically zero since it is compactly supported, therefore it is easy to check that the ansatz satisfies the transport equation in \eqref{edp}. Moreover, at the boundary $i=0$, it is seen that the ansatz needs to satisfy
\begin{align*}
 U(0,x)-\rho^s(0)&=\S_0\left(1-\exp\left(-\int_0^\infty\tau(i)\K* U(i,x)\md i\right)\right)-\S_0\left(1-\exp\left(-\int_0^\infty\tau(i)\K* \rho^s(i)\md i\right)\right)\\
	&=\S_0e^{-\mathscr{R}_0\rho^*}\left(1-\exp\left(-\mathscr{R}_0\widetilde{\K}(\lambda) e^{-\lambda x\cdot \mathbf{e}} \right)\right)\\
	&=	\S_0 e^{-\lambda x\cdot \mathbf{e}}
\end{align*}
where $\widetilde{\K}(\lambda)$ is given by \eqref{eqtildeKmu} which is well-defined (at least)  for $\lambda\in\Sigma=[0,\Lambda)$ with the set $\Sigma$ defined by \eqref{Sigma} (remember that $\Lambda\in[\mu_0,+\infty]$), thanks to hypothesis {\bf(H2$\mu$)}. 
This amounts to finding out $\lambda>0$ such that
\begin{equation}\label{A1}
	e^{\mathscr{R}_0\rho^*}  e^{-\lambda x\cdot \mathbf{e}}= 1-\exp\left(-\mathscr{R}_0\widetilde{\K}(\lambda) e^{-\lambda x\cdot \mathbf{e}} \right).
\end{equation}
We claim that there is $\lambda\in(0,\Lambda)$, independent of $\mathbf{e}$ but dependent of $\|x\|$, such that \eqref{A1} is satisfied. Indeed, we notice that the function $\lambda\in[0,\Lambda)\mapsto h_1(\lambda):=e^{\mathscr{R}_0\rho^*}  e^{-\lambda x\cdot \mathbf{e}}$ is analytic, decreasing and convex satisfying $h_1(0)=e^{\mathscr{R}_0\rho^*}>1$, and $h_1(\Lambda)<1$ for each $x\in\R^d$ such that $x\cdot\mathbf{e}$ sufficiently large; whereas the function  $\lambda\in[0,\Lambda)\mapsto h_2(\lambda):= 1-\exp\left(-\mathscr{R}_0\widetilde{\K}(\lambda) e^{-\lambda x\cdot \mathbf{e}} \right)$ is analytic and convex such that $h_2(0)=1-e^{-\mathscr{R}_0}<1$ and $h_2(\lambda)\to 1$ as $\lambda\to\Lambda$ for each $x\in\R^d$ with $x\cdot\mathbf{e}$ sufficiently large. This gives the existence and uniqueness of parameter $\lambda\in(0,\Lambda)$ such that \eqref{A1} is satisfied for each $x\in\R^d$ with $x\cdot\mathbf{e}$ sufficiently large.
 Due to the direction $\mathbf{e}\in\SS^{d-1}$ is arbitrary, we conclude that the parameter $\lambda$ indeed depends on $\|x\|$. We then reach the conclusion, and this completes the proof of Proposition \ref{prop_edpI_U}.
\end{proof}

\subsection{Long time behavior -- Proof of Theorem \ref{thm-edpI_long time behavior}}

In this section, we prove that, under the Hypotheses {\bf (H1)}-{\bf (H2)} and that $I_0\not\equiv0$ satisfies {\bf(H3)},
 the solution of  problem \eqref{edp}, associated with an initial condition $\rho_0$ satisfying {\bf(H4)}, converges locally uniformly towards the unique positive stationary solution $U$ to \eqref{edp} given in from Theorem \ref{thm_edpI-Liouville} for large times.

\begin{proof}[Proof of Theorem \ref{thm-edpI_long time behavior}]
	The main ingredient of the proof is basically a slight modification of  the existence part in the proof of Theorem \ref{thm_edpI-Liouville}. We sketch the outline below for the sake of completeness.
	
	Assume {\bf(H1)-(H2)} and suppose that $I_0\not\equiv0$ satisfies {\bf(H3)}. Let $\rho$ be the solution of \eqref{edp} associated with an initial condition $\rho_0$ satisfying {\bf(H4)}.
	Let $\underline \rho$ and  $\overline \rho$ be the solutions of the initial boundary value problem \eqref{edp}  with initial condition $\underline \rho_0=\delta\psi$ and $\overline \rho_0=
	\max\big(M,\Vert \rho_0/\pi\Vert_{L^\infty([0,i_\dagger)\times\R^d)}\big) \pi $ in $[0,i_\dagger)\times\R^d$, with $\delta>0$  and $M$ given as in the beginning of the proof of Theorem \ref{thm_edpI-Liouville}. Since $\rho(t,i,x)>0$ for $(t,i,x)\in(0,+\infty)\times[0,i_\dagger)\times\R^d$ with $t>i+i_\star$ due to {\bf (H3)} and Proposition~\ref{prop-positivity}, we can decrease $\delta$ (if necessary) such that  $0\le \delta\psi\le \rho(T,\cdot,\cdot)$ in $[0,i_\dagger)\times\R^d$ for some $T>L+i_\star$, where $L\geq L_0>0$ is given in the formula \eqref{sub_kppI} of $\psi$.

	  By the comparison principle Proposition \ref{prop-cp}, it follows that and $\underline \rho(t,i,x)\le \rho(t+T,i,x)\le \overline \rho(t+T,i,x)$ for $(t,i,x)\in\R_+\times[0,i_\dagger)\times\R^d$, and that $\underline \rho$ is nondecreasing in time for $(t,i,x)\in\R_+\times[0,i_\dagger)\times\R^d$, whereas $\overline \rho$ is nonincreasing in time for $(t,i,x)\in\R_+\times[0,i_\dagger)\times\R^d$. By the monotone convergence theorem and then by the Dini's theorem, eventually we obtain
	that $\underline \rho$ (resp. $\overline \rho$) converges locally uniformly for $(i,x)\in[0,i_\dagger)\times\R^d$ as $t\to+\infty$ to a stationary solution $\underline U$ (resp. $\overline U$) to \eqref{edp} such that 
	\begin{equation*}
		\delta\psi(i,x)\le  \underline U(i,x)\le\liminf_{t\to+\infty} \rho(t,i,x)\le\limsup_{t\to+\infty} \rho(t,i,x)\le  \overline  U(i,x)\le	\max\big(M,\Vert \rho_0 /\pi \Vert_{L^\infty([0,i_\dagger)\times\R^d)}\big)\pi(i)
	\end{equation*}
	locally uniformly for $(i,x)\in[0,i_\dagger)\times\R^d$. Thanks to the Liouville type result Theorem \ref{thm_edpI-Liouville}, the conclusion of the large time behavior of the solution $\rho$ to problem \eqref{edp} then immediately follows. 
\end{proof}

\subsection{Spreading properties -- Proof of Theorem \ref{thm-edpI-spreading property}}

Next, under the Hypotheses {\bf (H1)}-{\bf (H2$\mu$)} and that $I_0\not\equiv0$ satisfies {\bf(H3)}, we will prove in the regime $\mathscr{R}_0>1$ the spreading property with speed $c_*>0$, given in \eqref{c_*}, for the solutions of the initial boundary value problem \eqref{edp}, associated with an initial condition $\rho_0$ satisfying {\bf(H4)}.

\begin{proof}[Proof of Theorem \ref{thm-edpI-spreading property}] Assume that {\bf (H1)}-{\bf (H2$\mu$)} and $\mathscr{R}_0>1$, and that $I_0\not\equiv0$ satisfies {\bf(H3)}. Let $\rho$ be the solution  of the problem \eqref{edp} starting from an initial condition $\rho_0$ satisfying {\bf(H4)}. Let $c_*$ be given in \eqref{c_*}. We divide the proof in two parts.

\paragraph{Proof of statement (i).}
For any $c\in(0,c_*)$  and $j\in(0,i_\dagger)$ fixed, we consider an arbitrary sequence $(t_n,x_n)_{n\in\N}$ in $\R_+\times\R^d$ such that $t_n\to+\infty$ as $n\to+\infty$ and $\|x_n\|\le ct_n$ for each $n\in\N$. If $(x_n)_{n\in\N}$ is bounded, then we easily derive from Theorem \ref{thm-edpI_long time behavior} that $\rho(t_n,i,x_n)-U(i,x_n)\to 0$ as $n\to+\infty$ uniformly in $i\in[0,j]$. Suppose now that $(x_n)_{n\in\N}$ diverges to infinity. Since $\rho$ is a supersolution to the KPP model \eqref{fkpp} for which spreading occurs with the asymptotic speed $c_*$, we have 
\begin{equation*}
	\liminf_{n\to+\infty}\sup_{0\le i\le j}\big(\rho(t_n,i,x_n)- \rho^s(i)\big)\ge 0.
\end{equation*}
Together with the asymptotics of $U$ as $\|x\|\to+\infty$ given in Theorem \ref{thm_edpI-Liouville}, it follows that
\begin{equation}\label{4.4}
	\liminf_{n\to+\infty}\sup_{0\le i\le j}\Big(\rho(t_n,i,x_n)-U(i,x_n)\Big)\ge 0.
\end{equation}
To complete the proof of statement (i), it now remains to show that 
\begin{equation}
	\label{4.5}
	\limsup_{n\to+\infty}\sup_{0\le i\le j}\Big(\rho(t_n,i,x_n)-U(i,x_n)\Big)\le 0.
\end{equation}
To do so, we first claim that there is $A>0$ large enough such that for $(t,i,x)\in\R_+\times[0,j]\times\R^d$
\begin{equation}\label{4.6}
	\rho^s(i)+z(i,x):=\rho^s(i)+A(U(0,x)-\rho^s(0))\pi(i)\ge \rho(t,i,x).
\end{equation}
Indeed, for $(i,x)\in\mathcal{Q}:=\mathrm{supp}(I_0)\cup\mathrm{supp}(\rho_0)$, we can choose $A>0$ sufficiently large (independent of $\mathbf{e}$) such that
\begin{equation}
	\label{4.7}
	\rho^s(i)+z(i,x)>\rho(t,i,x),~~~~t\in\R_+,~(i,x)\in\mathcal{Q},
\end{equation}
by noticing that $z>0$ in this region and that $\rho$ is uniformly bounded from above. 
Let us look at the region of $(i,x)\in\mathcal{Q}^c$ (outside $\mathrm{supp}(I_0)\cup\mathrm{supp}(\rho_0)$). We notice that
\begin{equation}
	\label{28}
	\partial_i(\rho^s(i)+z(i,x))=-\gamma(i)(\rho^s(i)+z(i,x)),~~~~(i,x)\in\mathcal{Q}^c.
\end{equation}
 Set
\begin{equation*}
	W(x)=\frac{\rho^s(0)+z(0,x)}{\S_0}=\rho^*+\frac{z(0,x)}{\S_0}>\rho^*,~~~~~x\in\R^d.
\end{equation*}
We recall that $\rho^*$ solves $\rho^*=1-e^{-\mathscr{R}_0\K*\rho^*}$, where the mapping $v\in\X\mapsto 1-e^{-\mathscr{R}_0\K*v}$ with $\X=\{v\in\mathscr{C}(\R^d)~|~v\ge 0\}$  is concave, it then follows from $W(x)>\rho^*$   that $W(x)>1-e^{-\mathscr{R}_0\K*W(x)}$, namely,
\begin{equation*}
	\rho^s(0)+z(0,x)>\S_0\left(1-\exp\left(-\int_0^\infty\tau(i)\K* (\rho^s(i)+z(i,x))\md i\right)\right), ~~~~~x\in\R^d.
\end{equation*}
Combining \eqref{4.7}-\eqref{28} with the above inequality, we derive from a comparison argument  that $\rho^s(i)+z(i,x)\ge \rho(t,i,x)$ for $t\in\R_+$ and $(i,x)\in\mathcal{Q}^c$. This together with \eqref{4.7} proves our claim \eqref{4.6}. This further implies that 
\begin{equation*}
	\limsup_{n\to+\infty}\rho(t_n,i,x_n)\le \limsup_{n\to+\infty}\big(\rho^s(i)+z(i,x_n)\big)=\rho^s(i)<\liminf_{n\to+\infty}U(i,x_n)~~~\forall i\in[0,j],
\end{equation*}
where we have used the fact that $z(i,x)\to 0$ as $\|x\|\to +\infty$ due to Theorem \ref{thm_edpI-Liouville}.
Therefore, \eqref{4.5} is achieved. Since the sequence $(t_n,x_n)_{n\in\N}$ was chosen arbitrarily such that $\|x_n\|\le c t_n$ for all $n\in\N$, we combine \eqref{4.4}-\eqref{4.5} and conclude that 
\begin{equation*}
	\lim_{t\to+\infty}\sup_{\|x\|\le ct,~0\le i\le j}\big|\rho(t,i,x)-U(i,x)\big|= 0.
\end{equation*}
This  implies that $c_*$ is a lower bound of the asymptotic spreading speed.

\paragraph{Proof of statement (ii).}
Let us now prove that $c_*$ is also an upper bound of the asymptotic spreading speed. Since $I_0$ is compactly supported in $[0,i_\dagger)\times\R^d$, and satisfies {\bf (H3)}, we get the existence of a constant $A>0$ such that
\bqs
I_0(i,x) \leq A \gamma(i), \quad (i,x)\in \mathrm{supp}(I_0).
\eqs
For any direction $\mathbf{e}\in\SS^{d-1}$, we construct a function $\overline \rho$ of the form
\begin{equation}\label{supersol-edpI}
	\overline \rho(t,i,x)= C\min \left(\max\left(\S_0,A\right),e^{-\alpha_*(x\cdot \mathbf{e}-c_*t)-\alpha_* c_*i}\right)\pi(i), \quad (t,i,x)\in \R_+\times[0,i_\dagger)\times\R^d,
\end{equation}
with some constant $C>0$ which is fixed large enough such that $\rho_0(i,x)\leq \overline \rho(0,i,x)$ for each $(i,x)\in[0,i_\dagger)\times\R^d$. This is always possible since $\rho_0$ is compactly supported in $[0,i_\dagger)\times\R^d$. With our careful choice of $A$ and the explicit form of the exponential part in \eqref{supersol-edpI}, we readily conclude that  $\overline \rho$  is 
a supersolution to \eqref{edp} in $\R_+\times[0,i_\dagger)\times\R^d$. The comparison principle Proposition \ref{prop-cp} then leads to 
\begin{equation*}
	\rho(t,i,x)\le \overline \rho(t,i,x)~~~~\text{for}~(t,i,x)\in\R_+\times[0,i_\dagger)\times\R^d.
\end{equation*}
As a consequence, for any $c>c_*$ and for any $j\in(0,i_\dagger)$, we have
 \begin{equation*}
 	\underset{t\rightarrow+\infty}{\lim } \underset{x\cdot\mathbf{e}\geq ct,~ 0 \leq i \leq j}{\sup} \rho(t,i,x) =0.
 \end{equation*}
 Since $\mathbf{e}\in\SS^{d-1}$ is arbitrarily chosen, we then derive that
\begin{equation*}
	\underset{t\rightarrow+\infty}{\lim } \underset{\|x\|\geq ct,~ 0 \leq i \leq j}{\sup} \rho(t,i,x) =0,
\end{equation*}
which implies that $c_*$ is also an upper bound of the asymptotic spreading speed. 
This completes the proof of Theorem \ref{thm-edpI-spreading property}.
\end{proof}

\section{Traveling waves}\label{secTF}

In this last section, we assume throughout that {\bf(H1)-(H2$\mu$)} are satisfied, and we will focus on the case that $\mathscr{R}_0>1$. Our aim to show the existence and uniqueness of traveling waves for the KPP model
\bqq
\label{kppI-TW}
\left\{
\begin{split}
\partial_t \rho(t,i,x) &+ \partial_i\rho(t,i,x) = - \gamma(i) \rho(t,i,x) ,\quad t>0, \quad i\in(0,i_\dagger),  \quad x\in\R^d,  \\
\rho(t,0,x) &=\S_0 \left(1-\exp\left(-\int_0^\infty \tau(i) \K*\rho(t,i,x)\md i \right)\right), \quad  t> 0,  \quad x\in\R^d.
\end{split}\right.
\eqq
That is, we look for solutions of the form $\rho(t,i,x)=w(i,x\cdot{\bf e}-ct)$ for any direction ${\bf e}\in \mathbb{S}^{d-1}$, where the profile $w$ satisfies 
	\bqq\label{profile}
		\left\{
			\begin{split}
				-c\partial_z w(i,z) &+ \partial_i  w(i,z) = - \gamma(i) w(i,z),  \quad i\in(0,i_\dagger), \quad z\in\R, \\
				w(0,z)&=\S_0 \left(1-\exp\left(-\int_0^\infty \tau(i) \K_0*w(i,z)\md i \right)\right),\quad  z\in\R,
			\end{split}
		\right.
	\eqq
together with the conditions
	\bqq\label{conditionsTW}
\left\{
		\begin{split}
		w(i,-\infty)&=\rho^s(i)~\text{and}~w(i,+\infty)=0~~\text{for each}~i\in[0,i_\dagger),\\ 0& \, < \, w(i,z) \, < \, \rho^s(i)~~~\text{for}~(i,z)\in[0,i_+)\times\R.
		\end{split}\right.
	\eqq

Before proceeding with the proof, let us give some comments. Since $\mathscr{R}_0>1$, we notice that the only bounded nonnegative solutions of \eqref{profile} with $c=0$ are 0 and $\rho^s(i)$. What we are interested is whether or not there are some values of $c\in\R\backslash\{0\}$ for which \eqref{profile} has a solution satisfying $0 \,< \,w(i,z)\,< \, \rho^s(i)$ for $(i,z)\in[0,i_\dagger)\times\R$.  We point out that if there exists a traveling wave solution $w(i,z)$ with speed $c$, then $w(i,-z)$ will also satisfy \eqref{profile} along the direction $-\mathbf{e}$ with wave speed $-c$. Therefore, we shall restrict ourselves to the case that $c> 0$ in the sequel. 

Let us also remind that $c_*\in(0,+\infty)$ is the asymptotic spreading speed for KPP model \eqref{kppI} proven in Theorem \ref{thm-spreading property}. 
The necessary condition for the existence of traveling fronts of \eqref{kppI-TW} can be easily proved as follows using the spreading property of Theorem \ref{thm-spreading property}.
\begin{lem}
	\label{lem_TW_necessary}
	Traveling fronts of \eqref{kppI-TW} with wave speed $c\in\R_+$, if any, satisfy $c\ge c_*$. 
\end{lem}
\begin{proof}[Proof of Lemma \ref{lem_TW_necessary}]
	Assume that $w(i,x\cdot \mathbf{e}-ct)$ is a traveling front to the KPP model \eqref{kppI-TW} with speed $c\in\R_+$ along any direction $\mathbf{e}\in\SS^{d-1}$. Since by assumption we have $0 < w(i,z)$, we can always put a small compactly supported sub-solution below the traveling front which will satisfy hypothesis \textbf{(H4)} and \eqref{H3}. We then infer from the spreading property (ii) of Theorem \ref{thm-spreading property} that
	\begin{equation*}
		\forall c'\in(0,c_*),~\forall j\in(0,i_\dagger),~~\lim_{t\to+\infty}\sup_{0\le x\cdot\mathbf{e}\le c't,~0\le i\le j}\big|w(i,x\cdot\mathbf{e}-ct)-\rho^s(i)\big|=0.	\end{equation*}
	In particular,
	\begin{equation*}
		\lim_{t\to+\infty}\big|w(i,(c'-c)t)-\rho^s(i)\big|=0,~~\text{locally uniformly in}~i\in[0,i_\dagger).
	\end{equation*}
By virtue of the limit condition \eqref{conditionsTW}, we infer that $c'<c$. Since $c'\in(0,c_*)$ was arbitrarily chosen, we then derive that $c\ge c_*$.
\end{proof}

Next, we see from  the method of characteristics that \eqref{profile} is equivalent to the following integral equation:
\bqs
w(i,z)=\S_0\left(1-\exp\left(-\int_0^\infty\tau(i)\K_0*w(i,z+ci)\md i\right)\right)\pi(i),\quad i\in[0,i_\dagger), \quad z\in\R.
\eqs
It is then natural to introduce the change of unknown
\bqq\label{thm1.4-TW}
\chi(i,z):=\frac{w(i,z)}{\S_0\pi(i)}, \quad i\in[0,i_\dagger), \quad z\in\R,
\eqq
such that the above equation is reduced to 
\bqs
\chi(i,z)=1-\exp\left(-\S_0\int_0^\infty\omega(i)\K_0*\chi(i,z+ci)\md i\right),\quad i\in[0,i_\dagger), \quad z\in\R.
\eqs
We readily note that the right-hand side of the above equation is independent of the variable $i\in[0,i_\dagger)$, and so from now on, we suppress this dependence and simply look for solutions $\chi(z)$ to
\bqq
\label{TW-integral form}
\chi(z)=1-\exp\left(-\S_0\int_0^\infty\omega(i)\K_0*\chi(z+ci)\md i\right):=\T(\chi)(z),\quad z\in\R,
\eqq
together with the conditions
\bqq
\label{conditionsTWchi}
\chi(-\infty)=\rho^*, \quad \chi(+\infty)=0, \quad 0 \,<\, \chi \,<\, \rho^* \text{ in } \R.
\eqq
The operator $\T$ is monotone in the sense that if $\chi_1\le \chi_2$, then $\T(\chi_1)\le \T(\chi_2)$, which can be directly observed through
\bqs
\T(\chi_1)(z)- \T(\chi_2)(z)= e^{-\S_0\int_0^\infty\omega(i)\K_0*\chi_2(z+ci)\md i}\left(1-\exp\left(-\S_0\int_0^\infty\omega(i)\K_0*(\chi_1-\chi_2)(z+ci)\md i\right)\right)\le 0.
\eqs
 This integral equation \eqref{TW-integral form} is essentially similar to \eqref{renewal}. All the ingredients analyzed for \eqref{renewal} can be smoothly adapted here. 
 
It is very interesting and important to remark that the above traveling wave integral equation \eqref{TW-integral form} can be recast to the one originally derived by Diekmann in \cite{Diekmann1}. Indeed, set
 \bqs
 \zeta(z):= \S_0\int_0^\infty\omega(i)\K_0*\chi(z+ci)\md i,
\eqs
 then $\chi(z)=1-\exp(-\zeta(z))$ for each $z$, and multiplying both sides by $\S_0\omega(i)\K_0(z'+ci-z)$ and integrating in $i$ and $z$, one gets
 \bqs
 \zeta(z')=\S_0\int_0^\infty\omega(i)\K_0*\left(1-\exp(-\zeta(z'+ci))\right)\md i, \quad z'\in\R.
 \eqs
 As a consequence, one can directly use \cite[Corollary 6.2]{Diekmann1} to get the existence of traveling front solutions of \eqref{TW-integral form} for each $c>c_*$, and then invoke \cite{AR1976} to obtain the existence for $c=c_*$. Let us remark that the argument of \cite{AR1976} is not constructive in the sense that the existence of traveling fronts at $c=c_*$ are obtained by a limiting procedure from the case $c>c_*$ by taking $c\rightarrow c_*$. Below, we provide a direct constructive proof in the case of $c=c_*$ which allows us to retrieve the precise asymptotic behavior of the critical fronts at $c=c_*$ at $+\infty$. Furthermore, one can then combine the results of \cite{Diekmann4} and \cite{carrchmaj} to get the uniqueness of such monotone traveling waves modulo translation. We summarize all these results in the following lemma.

\begin{lem}\label{lem_existence of TWs}
	For each $c\ge c_*$, problem \eqref{TW-integral form}-\eqref{conditionsTWchi} admits a unique (modulo translation) solution $\chi_c$ which satisfies $\chi_c'<0$ and $0<\chi_c<\rho^*$ in $\R$. Furthermore, we have (up to normalization)
	\begin{align*}
		\frac{\chi_c(z)}{ e^{-\alpha_c z}}\longrightarrow1 ~~~(\text{for}~c>c_*),~~~~\frac{\chi_{c_*}(z)}{ze^{-\alpha_* z}}\longrightarrow1\quad \text{ as } z\rightarrow+\infty.
	\end{align*}
Here, $\alpha_*\in(0,\Lambda)$, given in Lemma \ref{lemma_c*}, is the unique value such that $\varphi_{c_*}(\alpha_*)=1$ and $\alpha_c\in(0,\alpha_*)$ is the unique value such that $\varphi_c(\alpha_c)=1$.
\end{lem}
\begin{proof}[Proof of Lemma \ref{lem_existence of TWs}]
	The proof of existence relies on the super- and subsolution argument. To do so, we focus on the integral equation \eqref{TW-integral form}.
	  Recall that $\alpha_*\in(0,\Lambda)$, given in Lemma \ref{lemma_c*}, is associated with $c_*$ such that $\varphi_{c_*}(\alpha_*)=1$. 
	 
\paragraph{Existence in the case of $c>c_*$.} Let us briefly proceed with the proof of existence of traveling fronts with speed $c>c_*$ which can originally be found in \cite{Diekmann1}. Here, we sketch it for completeness. Fix any $c>c_*$. Let $\alpha_c\in(0,\alpha_*)$  be the unique value given in Lemma \ref{lemma_alpha<alpha*} such that the dispersion relation $\varphi_c(\alpha_c)=1$ holds true. Then, we fix $0<\delta\le \min(\alpha_c,\alpha_*-\alpha_c)$ (therefore $\alpha_c+\delta\le\alpha_*$). It should be noted that $\varphi_c(\alpha_c+\delta)<1$. Indeed,  Lemma \ref{lemma-varphi} implies that  there is a unique $\hat c\in [c_*,c)$ such that $\varphi_{\hat c}(\alpha_c+\delta)=1$. Together with the fact that $c\mapsto \varphi_c(\alpha_c+\delta)$ is decreasing in $[0,+\infty)$, we then deduce that $1=\varphi_{\hat c}(\alpha_c+\delta)>\varphi_c(\alpha_c+\delta)$. On the other hand, there exists a constant $C>0$ such that 
	  \begin{equation}
	  	\label{5.6}
	  	1-e^{-s}\ge s-Cs^2~~~\text{for}~s\ge 0.
	  \end{equation}
	  We then pick 
	  \begin{equation}
	  	\label{M}
	  	M\ge \max\left(1,\frac{C\mathscr{R}_0\rho^*\varphi_c(\alpha_c+\delta)}{1-\varphi_c(\alpha_c+\delta)}\right)>0.
	  \end{equation}
For $z\in\R$, we define
	$$~~~~~~~~~~\overline \chi(z)=\rho^*\min\big(1, e^{-\alpha_c z}\big),~~~~	\underline \chi(z)=\rho^*\max\big(0, e^{-\alpha_c z}-Me^{-(\alpha_c+\delta)z}\big).$$
Following \cite{Diekmann1}, it is straightforward to check that $\overline \chi$ and $\underline \chi$ are respectively a super- and a subsolution of \eqref{TW-integral form} in $\R$.

By using the monotonicity of $\T$, we get that for each $k\in\N$
\begin{equation*}
\underline \chi\le	\T(\underline \chi)\le \T^2(\underline \chi)\le\cdots\le \T^k(\underline \chi)\le  \T^k(\overline \chi)\le \cdots\le \T^2(\overline \chi)\le \T(\overline \chi)\le \overline \chi, \text{ on } \R.
\end{equation*}
By the Arzela-Ascoli theorem, we derive that, up to extraction of a subsequence, $\T^k(\overline \chi)(z)\to \chi(z)$ as $k\to+\infty$ locally uniformly in $z\in\R$, and $\underline \chi\le \chi\le \overline \chi$ for $z\in\R$, which implies that $\chi$ is bounded in $\R$ and $\chi(+\infty)=0$.  Moreover, since $\overline \chi$ is nonincreasing in $z$, as is $\T^k(\overline \chi)$. Hence,  the limit function $\chi$ is nonincreasing in $z$. In particular, we have $0<\chi(-\infty)\le 1$. We claim that $\chi(-\infty)=\rho^*\in(0,1)$, where $\rho^*$ is the unique positive solution to $v=1-e^{-\mathscr{R}_0v}$. Indeed, consider any sequence $(z_n)_{n\in\N}$ diverging to $-\infty$ as $n\to+\infty$ and define $\chi_n(z)=\chi(z+z_n)$ for $z\in\R$ and each $n\in\N$. We observe that, up to some subsequence, $\chi_n(z)\to \chi_\infty$ as $n\to+\infty$. Then, by the Lebesgue's dorminated convergence theorem, we observe that $\chi_\infty$ solves 
\bqs
\chi_\infty=1-\exp(-\mathscr{R}_0\chi_\infty),
\eqs
which has a unique positive solution $\rho^*\in(0,1)$ since $\mathscr{R}_0>1$. Since the limit does not depend on the  particular sequence $(z_n)_{n\in\N}$, we arrive at $\chi(-\infty)=\rho^*$. Consequently, we have proved the existence of a nontrivial solution $\chi$ for problem \eqref{TW-integral form}, which is translation invariant, and satisfies  $0\le \chi\le \rho^*$ and $\chi'\le 0$ in $\R$ as well as $\chi(-\infty)=\rho^*$ and $\chi(+\infty)=0$. Since $\underline\chi\le \chi\le \overline\chi$ in $\R$, we have (up to translation)
\bqs
\chi(z)\sim  e^{-\alpha_c z}~~~ \text{ as } z\rightarrow+\infty,
\eqs
where $\alpha_c\in(0,\alpha_*)$ is the unique value such that $\varphi_c(\alpha_c)=1$.

\paragraph{Existence in the critical case of $c=c_*$.} Let $C>0$ satisfy \eqref{5.6}. We fix a large constant $A>0$ such that  there exists $z_0>0$ such that $Az_0 e^{-\alpha_* z_0}\ge 1$. One can then define $\bar z=\max\{z>0 ~|~ A ze^{-\alpha_*  z}=1\}$. Choose $0<\delta<\min(\alpha_*,\Lambda-\alpha_*)/4$.  Since $\alpha\in[0,\Lambda)\mapsto \varphi_{c_*}(\alpha)$ is convex and since  $\varphi_{c_*}(\alpha_*)=1=\min_{\alpha\in(0,\Lambda)}\varphi_{c_*}(\alpha)$, we infer that $\alpha\in [\alpha_*,\Lambda)\mapsto \varphi_{c_*}(\alpha)$ is increasing, whence $\varphi_{c_*}(\alpha_*+2\delta)>\varphi_{c_*}(\alpha_*+\delta)>1$. 
 Then, we can pick  $B>1$ sufficiently large such that 
\begin{equation}\label{B}
z^2e^{-2\alpha_* z}\le e^{-(\alpha_*+2\delta)z},~~~~C\mathscr{R}_0\rho^*A^2 \varphi_{c_*}(\alpha_*+2\delta)e^{-\delta z}\le \varphi_{c_*}(\alpha_*+\delta)-1~~~~~\text{for all}~z> (B-1)/A.
\end{equation}
For $z\in\R$, let us now define
\begin{align*}
	\overline \chi(z)=\rho^*
	\begin{cases}
		1, & z\le \bar z,\\
		Aze^{-\alpha_* z} & z> \bar z,
	\end{cases}
~~~~~~	\underline \chi(z)=\rho^*\max\big(0, A ze^{-\alpha_* z}-Be^{-\alpha_* z}+e^{-(\alpha_*+\delta)z}\big).
\end{align*}
Let us now verify that $\overline  \chi$ is a supersolution to  \eqref{TW-integral form} in $\R$. Again, it is suffices to consider the region where $\overline \chi(z)=\rho^*Aze^{-\alpha_* z}$. Indeed, by a straightforward computation, we derive that
	\begin{align*}
	\T(\overline \chi)(z)&=1-\exp\left(-\S_0\int_0^\infty\omega(i)\K_0*\overline \chi(z+c_*i)\md i\right)\\
	&\leq \S_0\int_0^\infty\omega(i)\K_0*\overline \chi(z+c_*i)\md i\\
	&= \rho^* A\S_0 \int_0^\infty\omega(i)\K_0*\big((z+c_* i)e^{-\alpha_*(z+c_* i)}\big)\md i\\
	&=\rho^* A\S_0\int_0^\infty\omega(i)e^{-\alpha_* c_*i} \K_0*\big(ze^{-\alpha_* z}+c_*i e^{-\alpha_* z}\big)\md i \\
	&=\rho^* A ze^{-\alpha_* z}=\overline \chi(z),
\end{align*}
by recalling that $\varphi_{c_*}(\alpha_*)=1$, and $\partial_{\alpha} \varphi_{c_*}(\alpha_*)=0$, namely,
$$\left(\int_{\R}\K_0(z)ze^{\alpha_* z}\md z\right)\left(\int_0^\infty\omega(i)e^{-\alpha_* c_*i }\md i\right)-\left(\int_{\R}\K_0(z)e^{\alpha_* z}\md z\right)\left(\int_0^\infty\omega(i) c_*i e^{-\alpha_* c_*i }\md i\right)=0,$$
as well as  
\bqs
\K_0*(ze^{-\alpha_* z})=e^{-\alpha_* z}\left(z\int_{\R}\K_0(x)e^{\alpha_* x}\md x-\int_{\R}\K_0(x)x e^{\alpha_* x}\md x\right).
\eqs 
Therefore, we conclude that $\overline \chi$ is a supersolution of \eqref{TW-integral form} in $\R$.

Let us now prove  that $\underline \chi$ is a subsolution to \eqref{TW-integral form} in $\R$. It is sufficient to take into account the case that $\underline \chi\neq 0$, which implies that $z>(B-1)/A$. Thanks to \eqref{B}, we derive that
\begin{equation*}
\underline \chi^2(z) \leq  \overline \chi^2(z)= (\rho^*)^2  A^2 z^2 e^{-2\alpha_* z}\!\le\! (\rho^*)^2 A^2 e^{-(\alpha_*+2\delta)z}~~\text{for}~z\!>\!(B-1)/A.
\end{equation*}
By applying \eqref{5.6}, we derive that
 \begin{equation}
 	\label{sub-1}
 \begin{aligned}
 	\T(\underline \chi)(z)&=1-\exp\left(-\S_0\int_0^\infty\omega(i)\K_0*\underline \chi(z+c_*i)\md i\right)\\
 	&\ge \S_0\int_0^\infty\omega(i)\K_0*\underline \chi(z+c_*i)\md i-C \left(\S_0\int_0^\infty\omega(i)\K_0*\underline \chi(z+c_*i)\md i\right)^2.
 \end{aligned}
\end{equation}
On the other hand, we also have
\begin{equation}
	\label{sub-2}
\begin{aligned}
	\int_0^\infty&\omega(i)\K_0*\underline \chi(z+c_*i)\md i=\int_0^\infty\int_\R \omega(i)\K_0(z+c_*i-y)\underline \chi(y)\md y\md i\\
	&\le \left( \int_0^\infty\int_\R \omega(i)\K_0(z+c_*i-y)\md y\md i\right)^\frac{1}{2} \left(\int_0^\infty \omega(i) \K_0*\underline \chi^2(z+c_*i)\md i\right)^\frac{1}{2}\\
	&=\left(\frac{\mathscr{R}_0}{\S_0}\right)^{\frac{1}{2}}  \left(\int_0^\infty \omega(i) \K_0*\underline \chi^2(z+c_*i)\md i\right)^\frac{1}{2}
\end{aligned}
\end{equation}
Combining \eqref{sub-1} and \eqref{sub-2} as well as \eqref{B}, we then arrive at
\begin{align*}
	\T(\underline \chi)(z)&\ge  \S_0\int_0^\infty\omega(i)\K_0*\underline \chi(z+c_*i)\md i -C\S_0\mathscr{R}_0 \int_0^\infty \omega(i) \K_0*\underline \chi^2(z+c_*i)\md i\\
	&\geq \rho^* \left(A ze^{-\alpha_* z}-Be^{-\alpha_* z}+\varphi_{c_*}(\alpha_*+\delta)e^{-(\alpha_*+\delta)z}-C\mathscr{R}_0\rho^* A^2\varphi_{c_*}(\alpha_*+2\delta) e^{-(\alpha_*+2\delta)z}\right)\\
	&\ge \rho^* \left(A ze^{-\alpha_* z}-Be^{-\alpha_* z}+e^{-(\alpha_*+\delta)z}\right)=\underline \chi(z),
\end{align*} 
which implies that $\underline \chi$ is a subsolution of \eqref{TW-integral form} in $\R$.

By repeating the argument from the case $c>c_*$, we eventually get the  existence of a nonincreasing  solution $\chi_{c_*}$ of \eqref{TW-integral form} satisfying $0\le\chi_{c_*}\le \rho^*$ in $\R$, $\chi_{c_*}(-\infty)=\rho^*$ and $\chi_{c_*}(+\infty)=0$. Again, we observe from the construction of $\chi_{c_*}$ that (up to translation)
\bqs
\chi_{c_*}(z)\sim z e^{-\alpha_*z}~~~ \text{ as } z\rightarrow+\infty.
\eqs
This completes the existence part of this lemma.

\paragraph{Strict monotonicity.} Now, let us prove that $0<\chi_c<\rho^*$ and $\chi_c'< 0$ in $\R$  for each $c\ge c_*$. From the existence proof, we have that $0\le \chi_c\leq \rho^*<1$ in $\R$. Assume now that  there is $z_0\in\R$ such that $\chi_c(z_0)=0$. It then follows from  \eqref{TW-integral form} that
	\begin{equation*}
		\int_0^\infty\omega(i)\K_0*\chi(z_0+ci)\md i=0.
	\end{equation*}
Since $\K_0>0$ in $\R$ and $\chi_c\ge 0$ in $\R$, one infers that $\chi_c\equiv 0$ in $\R$, which contradicts $\chi_c(-\infty)=\rho^*$. Therefore, we arrive at $0<\chi_c\leq \rho^*< 1$ in $\R$.  Next, assume that there is $z_0\in\R$ such that $\chi_c(z_0)=\rho^*$. Using the fact that $\rho^*=1-\exp(-\mathscr{R}_0\rho^*)$, it follows from \eqref{TW-integral form} that
\bqs
\mathscr{R}_0\chi_c(z_0)=\mathscr{R}_0\rho^*=\S_0\int_0^\infty\omega(i)\K_0*\chi_c(z_0+ci)\md i,
\eqs
and using the definition of $\mathscr{R}_0$ we equivalently get
\bqs
\S_0\int_0^\infty\omega(i)\K_0*\left(\chi_c(z_0)-\chi_c(z_0+ci)\right)\md i=0.
\eqs
But since $\chi_c(z)\leq \chi_c(z_0)=\rho^*$ for all $z\in\R$, we deduce that $\chi_c(z_0)=\chi_c(z_0+ci-z)$ for all $i\in\mathrm{supp}(\tau)$ and $z\in\R$, which is a contradiction as $\chi_c(+\infty)=0$. As a consequence, we have $0<\chi_c<\rho^*$ in $\R$. Next, to prove $\chi_c'< 0$ in $\R$, we observe from \eqref{TW-integral form} that
\begin{equation*}
	\chi'_c(z)=\S_0(1-\chi_c(z))\int_0^\infty\omega(i)\K_0*\chi_c'(z+ci)\md i,\quad z\in\R.
\end{equation*}
Assume by contradiction that there is $z_1\in\R$ such that $\chi_c'(z_1)=0$. Since $0<\chi_c<\rho^*$ in $\R$, the above equation implies that $\int_0^\infty\omega(i)\K_0*\chi_c'(z_1+ci)\md i=0$. Hence, one further deduces from $\K_0>0$ in $\R$ and  $\chi_c'\le 0$ in $\R$ that  $\chi_c'\equiv0$ in $\R$, which contradicts the limit condition. Consequently, $\chi_c'<0$ in $\R$.

\paragraph{Uniqueness.} Let us finally discuss the uniqueness. We first note that the uniqueness in the super-critical case, that is for $c>c_*$, can be obtained by applying \cite{Diekmann4}. We thus only focus on the critical case $c=c_*$ and explain how one can derive such a result by using the method in \cite{carrchmaj}. The idea is to first prove that any solution $\chi$ of \eqref{TW-integral form}-\eqref{conditionsTWchi} with $c=c_*$ satisfies, up to translation, 
\bqq\label{asymptotic}
\frac{\chi(z)}{ z e^{-\alpha_*z}}\rightarrow 1, \text{ as } z\rightarrow+\infty.
\eqq
We first obtain from \cite[Lemma 4.5]{Diekmann4} that any nontrivial solution $\chi$ of \eqref{TW-integral form} satisfying  $\chi(+\infty)=0$ is such that $\chi(z)=\mathcal{O}(e^{-\delta z})$ as $z\rightarrow+\infty$ for some $\delta>0$. As a consequence, the two sided Laplace transform of $\chi$:
\bqs
\widetilde{\chi}(\alpha):=\int_\R \chi(z)e^{- \alpha z} \md z,
\eqs
is well-defined for each $-\delta<\mathrm{Re}(\alpha)<0$. Now, from \eqref{TW-integral form}, we get that
\bqs
\left(1-\varphi_{c_*}(-\alpha)\right)\widetilde{\chi}(\alpha)=\!\int_\R e^{-\alpha z} \left[ 1-\exp\left(-\S_0\int_0^\infty\omega(i)\K_0*\chi(z+ci)\md i\right) -\S_0\int_0^\infty\omega(i)\K_0*\chi(z+ci)\md i \right]\!\md z 
\eqs
where the right-hand side is well-defined  for each $\alpha\in\mathbb{C}$ such that $-2\delta<\mathrm{Re}(\alpha)<0$ since $1-e^{-s}-s=\mathcal{O}(s^2)$ as $s\rightarrow0$. Using \cite{carrchmaj} and the positivity of $\chi$, we get that $\widetilde{\chi}(\alpha)$ is actually defined for $-\alpha_*<\mathrm{Re}(\alpha)<0$. Now, since $1-\varphi_{c_*}(-\alpha)=0$ has a simple double root at $\alpha=\alpha_*$, we get that $1-\varphi_{c_*}(-\alpha)=(\alpha+\alpha_*)^2\mathcal{H}(\alpha)$ where $\mathcal{H}(\alpha)$ is a holomorphic function in the strip $-\alpha_*\leq \mathrm{Re}(\alpha)<0$ with $\mathcal{H}(\alpha_*)\neq0$. As a consequence, a direct application of Ikehara's Theorem, as recalled in \cite[Proposition 2.3]{carrchmaj}, gives the desired asymptotic result \eqref{asymptotic}. We now conclude the argument by following the strategy presented in \cite{carrchmaj}. For $\epsilon>0$, let us define
\bqs
\varsigma_\epsilon(z)=\frac{\chi_1(z)-\chi_2(z)}{(\epsilon|z|+1)e^{-\alpha_*z}},~~~~~\varsigma_0(z)=\frac{\chi_1(z)-\chi_2(z)}{e^{-\alpha_*z}}, \quad z\in\R,
\eqs
where $\chi_1$ and $\chi_2$ are two given solutions of \eqref{TW-integral form} bounded between 0 and $\rho^*$ satisfying
\bqs
\frac{\chi_k(z)}{ z e^{-\alpha_*z}}\rightarrow 1, \text{ as } z\rightarrow+\infty, \text{ for } k=1,2.
\eqs
We note that $\varsigma_\epsilon(\pm\infty)=0$ and assume that $\varsigma_\epsilon\not\equiv0$ on $\R$. Without loss of generality, we let $z_\epsilon\in\R$ be such that $\varsigma_\epsilon(z_\epsilon)=\max_{z\in\R}\left|\varsigma_\epsilon(z)\right|>0$. We divide into three cases. Assume first that $(z_\epsilon)_{\epsilon>0}$ remains bounded as $\epsilon\rightarrow0$, then, up to a subsequence, we have that $z_\epsilon\rightarrow z_0$ as $\epsilon\rightarrow0$ for some finite $z_0\in\R$, 
and $|\varsigma_\epsilon(z_\epsilon)|\rightarrow |\varsigma_0(z_0)|$ as $\epsilon\rightarrow0$ and thus $|\varsigma_0(z)|\leq |\varsigma_0(z_0)|$ for all $z\in\R$. Now, one derives from \eqref{TW-integral form} that
\bqq
\label{eqInterm}
\left|\chi_1(z)-\chi_2(z)\right|\leq \S_0 \int_0^\infty \omega(i) \K_0*\left| \chi_1(z+c_*i)-\chi_2(z+c_*i)\right|\md i,
\eqq
which implies that
\bqs
|\varsigma_0(z)|\leq \S_0 \int_0^\infty \omega(i) e^{-\alpha_*c_* i } \int_\R \K_0(z')\left|\varsigma_0(z+c_*i-z') \right| e^{\alpha_*z'} \md z'\md i \leq \left|\varsigma_0(z_0) \right|.
\eqs
Thus, the above inequality at $z=z_0$ must be an equality, but this is only possible if  $\varsigma_0(z_0)=\varsigma_0(z_0+c_*i-z)$ for all $z\in\R$ and $i\in\mathrm{supp}(\tau)$. Recalling that $\varsigma_0(-\infty)=0$, we obtain that  $\varsigma_0\equiv0$ and thus $\chi_1\equiv\chi_2$. Assume next that $z_\epsilon\rightarrow-\infty$  as $\epsilon\rightarrow0$, then $|\varsigma_\epsilon(z_\epsilon)|\rightarrow 0$ as $\epsilon\rightarrow0$. Since $\varsigma_\epsilon\rightarrow \varsigma_0$ as $\epsilon\rightarrow0$, we get $|\varsigma_0(z)|\leq0$ for $z\in\R$ which gives a contradiction since $|\varsigma_\epsilon(z)|\leq|\varsigma_0(z)|\le 0$ for each $z\in\R$ however $\varsigma_\epsilon(z_\epsilon)>0$. It is left to consider the case that $z_\epsilon\rightarrow+\infty$  as $\epsilon\rightarrow0$. We derive from \eqref{eqInterm} that 
\bqs
|\varsigma_\epsilon(z)|(\epsilon|z|+1)\leq \S_0 \int_0^\infty \omega(i) e^{-\alpha_*c_* i } \int_\R \K_0(z')\left|\varsigma_\epsilon(z+c_*i-z') \right| (\epsilon|z+c_*i-z'|+1) e^{\alpha_*z'} \md z'\md i.
\eqs
Evaluating the above inequality at $z=z_\epsilon>0$ for $\epsilon$ small enough, and using the fact that $\varphi_{c_*}(\alpha_*)=1$ together with $\partial_\alpha\varphi_{c_*}(\alpha_*)=0$, we eventually obtain that
\begin{align*}
&\S_0 \int_0^\infty \omega(i) e^{-\alpha_*c_* i } \int_\R \K_0(z')\left[\varsigma_\epsilon(z_\epsilon)-\left|\varsigma_\epsilon(z_\epsilon+c_*i-z') \right| \right] e^{\alpha_*z'} \md z'\md i\\
&+\S_0 \epsilon \int_0^\infty \omega(i) e^{-\alpha_*c_* i } \int_\R \K_0(z')\left[\varsigma_\epsilon(z_\epsilon)(z_\epsilon-z'+c_* i)-\left|\varsigma_\epsilon(z_\epsilon+c_*i-z') \right| |z_\epsilon-z+c_* i|\right] e^{\alpha_*z'} \md z'\md i \leq 0.
\end{align*}
Rearranging the above integrals, we equivalently derive that
\begin{align*}
&-\underbrace{\S_0 \epsilon \varsigma_\epsilon(z_\epsilon) \int_0^\infty \omega(i) e^{-\alpha_*c_* i } \int_\R \K_0(z')\left(|z_\epsilon-z'+c_*i|-(z_\epsilon-z'+c_*i)\right) e^{\alpha_*z'} \md z'\md i}_{:=\mathscr{I}^1_\epsilon}\\
&+\underbrace{\S_0 \int_0^\infty \omega(i) e^{-\alpha_*c_* i } \int_\R \K_0(z')(1+\epsilon|z_\epsilon-z'+c_* i|)\left[\varsigma_\epsilon(z_\epsilon)-\left|\varsigma_\epsilon(z_\epsilon+c_*i-z') \right|\right] e^{\alpha_*z'} \md z'\md i}_{:=\mathscr{I}^2_\epsilon} \leq 0,
\end{align*}
since $\varsigma_\epsilon(z_\epsilon)>0$. 
We note that $\mathscr{I}^1_\epsilon> 0$ and $\mathscr{I}^2_\epsilon\geq 0$ by definition of $z_\epsilon$ and since $\K_0>0$ in $\R$. The above inequality simply reads $-\mathscr{I}^1_\epsilon+\mathscr{I}^2_\epsilon\leq0$. Our aim is to prove that $\mathscr{I}^2_\epsilon= 0$ for $\epsilon$ small enough. Assume by contradiction that $\mathscr{I}^2_\epsilon>0$. We claim that for $\epsilon$ small enough, we can always ensure that $\mathscr{I}^1_\epsilon \leq \mathscr{I}^2_\epsilon/2$.

First, using the fact that $z_\epsilon>0$ for $\epsilon$ small enough, we have
\begin{align*}
\mathscr{I}^1_\epsilon&= 2 \epsilon \varsigma_\epsilon(z_\epsilon) \S_0 \int_0^\infty \omega(i) e^{-\alpha_*c_* i } \int_{z_\epsilon+c_*i}^\infty \K_0(z')\left(z'-z_\epsilon-c_*i\right) e^{\alpha_*z'} \md z'\md i \\
&\leq 2 \epsilon \varsigma_\epsilon(z_\epsilon) \S_0 \int_0^\infty \omega(i) e^{-\alpha_*c_* i } \int_{z_\epsilon+c_*i}^\infty (z'-z_\epsilon) \K_0(z')e^{\alpha_*z'} \md z'\md i\\
&\leq 2 \epsilon \varsigma_\epsilon(z_\epsilon) \S_0 \left(\int_0^\infty \omega(i) e^{-\alpha_*c_* i }\md i\right)\left( \int_{z_\epsilon}^\infty (z'-z_\epsilon) \K_0(z')e^{\alpha_*z'} \md z'\right).
\end{align*}
Then, for any $\mu\in(\alpha_*,\Lambda)$ and $\epsilon$ small enough, we get that
\bqs
\int_{z_\epsilon}^\infty (z'-z_\epsilon) \K_0(z')e^{\alpha_*z'} \md z=e^{(\alpha_*-\mu)z_\epsilon}\int_0^\infty u \K_0(u+z_\epsilon)e^{\mu(u+z_\epsilon)}e^{(\alpha_*-\mu)u}\md u \leq e^{(\alpha_*-\mu)z_\epsilon}\int_0^\infty u e^{(\alpha_*-\mu)u}\md u
\eqs
which then implies 
\bqs
\mathscr{I}^1_\epsilon \leq 2 \epsilon \varsigma_\epsilon(z_\epsilon) e^{(\alpha_*-\mu)z_\epsilon} \S_0 \left(\int_0^\infty \omega(i) e^{-\alpha_*c_* i }\md i\right)\left( \int_0^\infty u e^{(\alpha_*-\mu)u}\md u\right).
\eqs 
We remark that when $\Lambda=+\infty$, since $\mathscr{I}^2_\epsilon>0$ for each $\epsilon$ fixed small enough and since $\mathscr{I}^1_\epsilon\rightarrow0$ as $\mu\rightarrow+\infty$ independently of $\epsilon$, we can always take $\mu>0$ large enough such that $\mathscr{I}^1_\epsilon \leq \mathscr{I}^2_\epsilon/2$, as claimed. We focus on $\Lambda<+\infty$ in the sequel. 

By letting $\mu\rightarrow\Lambda$, we get the following estimate
\bqs
\mathscr{I}^1_\epsilon \leq 2 \epsilon \varsigma_\epsilon(z_\epsilon) e^{(\alpha_*-\Lambda)z_\epsilon} \S_0 \left(\int_0^\infty \omega(i) e^{-\alpha_*c_* i }\md i\right)\left( \int_0^\infty u e^{(\alpha_*-\Lambda)u}\md u\right).
\eqs 
Now, let us estimate the second integral. Since $\varsigma_\epsilon(-\infty)=0$, one gets the existence of $M>0$ such that  $\left|\varsigma_\epsilon(z) \right|\leq \varsigma_\epsilon(z_\epsilon)/2$ for all $z<-M$. Next, we introduce $\delta_\epsilon>0$, to be fixed later, satisfying $\delta_\epsilon z_\epsilon \rightarrow +\infty$ as $\epsilon\rightarrow0$. As a consequence, for $\epsilon>0$ small enough, one can ensure that $\delta_\epsilon z_\epsilon>M$. We now compute
\begin{align*}
\mathscr{I}^2_\epsilon& \geq \S_0 \int_0^\infty \omega(i) e^{-\alpha_*c_* i } \int_{(1+\delta_\epsilon)z_\epsilon +c_*i}^\infty \K_0(z')\left[\varsigma_\epsilon(z_\epsilon)-\left|\varsigma_\epsilon(z_\epsilon+c_*i-z') \right|\right] e^{\alpha_*z'} \md z'\md i\\
&\geq \frac{\S_0}{2} \varsigma_\epsilon(z_\epsilon) \int_0^\infty \omega(i) e^{-\alpha_*c_* i } \int_{(1+\delta_\epsilon)z_\epsilon +c_*i}^\infty \K_0(z')e^{\alpha_*z'} \md z'\md i \\
&\geq \frac{\S_0}{2} \varsigma_\epsilon(z_\epsilon) \int_0^\infty \omega(i) e^{-\alpha_*c_* i } \int_{(1+\delta_\epsilon)z_\epsilon +c_*i}^\infty e^{(\alpha_*-\Lambda)z'} \md z'\md i \\
&\geq \frac{\S_0}{2(\Lambda-\alpha_*)} \varsigma_\epsilon(z_\epsilon) e^{(\alpha_*-\Lambda)(1+\delta_\epsilon)z_\epsilon} \int_0^\infty \omega(i) e^{-\Lambda c_* i } \md i.
\end{align*}
We can now set $\delta_\epsilon:=\frac{-\ln(\epsilon)}{2(\Lambda-\alpha_*)z_\epsilon}>0$, such that we have indeed $\delta_\epsilon z_\epsilon = \frac{-\ln(\epsilon)}{2(\Lambda-\alpha_*)}\rightarrow+\infty$ as $\epsilon\rightarrow0$. With such a choice, we get that
\bqs
\mathscr{I}^2_\epsilon \geq \frac{\S_0}{2(\Lambda-\alpha_*)} \epsilon^{1/2}\varsigma_\epsilon(z_\epsilon) e^{(\alpha_*-\Lambda)z_\epsilon} \int_0^\infty \omega(i) e^{-\Lambda c_* i } \md i.
\eqs
As a consequence, for $\epsilon$ small enough, we can ensure that $\mathscr{I}^1_\epsilon \leq \mathscr{I}^2_\epsilon/2$. This proves the initial claim.

Recalling that $-\mathscr{I}^1_\epsilon+\mathscr{I}^2_\epsilon\leq0$ and $\mathscr{I}^1_\epsilon \leq \mathscr{I}^2_\epsilon/2$ for $\epsilon$ small enough, we deduce that $0<\mathscr{I}^2_\epsilon/2 \leq 0$ which is a contradiction. Thus $\mathscr{I}^2_\epsilon=0$ for $\epsilon$ small enough, which then implies that $\varsigma_\epsilon(z_\epsilon)=\left|\varsigma_\epsilon(z_\epsilon+c_*i-z) \right|$ for each $i\in\mathrm{supp}(\tau)$ and $z\in\R$, which leads to a contradiction since $\varsigma_\epsilon(-\infty)=0$ and we assumed that $\varsigma_\epsilon(z_\epsilon)>0$. This concludes the uniqueness of the critical waves with $c=c_*$.
\end{proof}

Due to \eqref{thm1.4-TW}, Theorem \ref{thm-TW-KPP model} is then an immediate consequence of Lemmas \ref{lem_TW_necessary}-\ref{lem_existence of TWs}. 

\begin{rmk}
It is easy to check that the above proofs readily extend to the case of compactly supported kernels. Actually, only the last step of the strict monotonicity of $\chi_c$ needs to be slightly adapted since the other parts of the proof (existence and uniqueness) do not require the positivity of the kernel.
\end{rmk}

\section{Perspectives and open problems}
We have shed a new light on an old, and widely used model that describes the propagation of epidemics, by exploiting an analogy with reaction-diffusion models where the overall propagation is dictated by a line of fast diffusion. This analogy will be pursued in a forthcoming work \cite{FRZ25}, where we study the sharp large time asymptotics of the solutions of (\ref{kppI}), something that might have been nontrivial to do under the initial formulation. 

 The study of models accounting for the spread of epidemics has flourished over the last 70 years, and has given rise to an enormous amount of literature to which it is impossible to do justice in a paper. Thus, we have limited ourselves to quoting the works that we deem both conceptually important and close to the questions that we are studying. 
Notwithstanding the aforementionned numerous contributions, our work also sheds a light on some basic questions that remain unanswered.  They  are discussed below,  we would be happy to see some of them solved in the future. 
\subsubsection*{What happens in the presence of heterogeneities?}
A natural starting point is to question the key assumption  that the initial susceptible population $S_0$ is constant. While it is reasonable, one may also ask what happens when it is dropped, and when, for instance, the initial datum $S(0,x)$ is a smooth function, bounded from below and bounded away from 0. This apparently innocent modification is in fact a serious perturbation of the structure of the problem, as one cannot expect spreading at an asymptotically constant speed. One cannot either hope for the existence of traveling waves. 

As far as the spreading speed is concerned, the work one might take inspiration from is that of Berestycki-Nadin \cite{BNad}: there, upper and lower bounds for the propagation speed in Fisher-KPP type models with diffusion given by a second order elliptic operator are derived, and these bounds turn out to yield the correct spreading speeds in all the known cases. Whether the methods of this paper carry over to (\ref{kppI}) does not look that obvious, for two reasons: the first one is the presence of the nonlocal diffusion, which means that one should forget any idea of relying on compactness of the solutions; the second one is the presence of the additional variable $i$. 

The notion by which to replace the notion of traveling waves is that of transition waves; to describe them informally we rely on the work of Berestycki-Hamel \cite{BH}: a transition wave of a reaction-diffusion equation of the type $u_t+Au=f(x,u)$, $A$ a diffusive operator, is a time global solution (that is, defined for all real times) connecting the state 0 at $+\infty$ to some stable state $u_*$ at $-\infty$, uniformly in time, and such that the distance between two level sets is bounded independently of time. Such waves have been proved to exist in a number of models, including nonlocal equations when $f$ is periodic in $x$ (see Coville, Davila, Martinez \cite{CDM}; see also results by Ducrot-Giletti \cite{DG14} or Deng-Ducrot \cite{DD} for models with periodic heterogeneities closer to that under study in this paper). We are not aware of such results for nonlocal equations when the dependence in $x$ of $f$ has no particular structure, something that reverberates on Model (\ref{kppI}).  The study of transition fronts is not only a subject having its own high interest, one can use them to estimate the (upper and lower) spreading speed, in several space dimensions, as is beautifully realised by Rossi \cite{Ros} in the context of general reaction-diffusion equations. 

When the domain is heterogeneous, another interesting issue is the following: what happens if the unknowns $S(t,x)$ and $I(t,x,i)$ are defined on, say, a perforated line and the kernel $\mathcal{K}$ is compactly supported? Whether or not spreading can be blocked, and, if yes, how the support of $\mathcal{K}$ is involved, certainly deserves further investigation. Whenever it was possible, we have endeavoured to work in several space dimensions, as there is no reason to assume that the spatial domain is a line; the natural dimension domain is indeed $\mathbb{R}^2$. Things become really involved if the domain is made of disconnected patches or if it has holes. The question of spreading is very much open in this context,  as blocking phenomena or intermediate configurations, such as the existence of steady solutions tending to 0 at infinity, may happen. Of course, transition waves and sharp asymptotics are outstanding issues, a general understanding of them seems to be quite remote.

All these questions can be reformulated word by word if the diffusion kernel $\mathcal{K}$ has a dependence on the position, that is, the term $\mathcal{K}*I$ is replaced by an integral operator of the form $\int_{\mathbb{R}}\mathcal{K}(x,y,i)I(t,y,i)~\mathrm{d}y$, where, for every $x\in\mathbb{R}$, the function $(y,i)\mapsto\mathcal{K}(x,y,i)$ is smooth and compactly supported.  

\subsubsection*{Presence of diffusion in the susceptible population}
In the same way as assuming that the susceptible population is homogeneous, it is perfectly reasonable to assume that its movements can be neglected. However, it is also reasonable to take them into account, and System (\ref{EDP}) becomes
\bqq
\left\{
\begin{split}
\partial_t \S(t,x)+AS(t,x) &= - \left( \int_0^\infty \tau(i) \K*I(t,i,x)\md i \right) \S(t,x), \quad t>0, \quad x\in\R^d,\\
\partial_t I(t,i,x) + \partial_iI(t,i,x) &= - \gamma(i) I(t,i,x) , \quad t>0,  \quad i>0,  \quad x\in\R^d, \\
I(t,0,x)&= \left( \int_0^\infty \tau(i) \K*I(t,i,x)\md i \right) \S(t,x), \quad t>0,  \quad x\in\R^d,
\end{split}
\right.
\label{EDP1}
\eqq 
where $A$ is any diffusion operator: it can be of the form $-d\Delta S$, but it is also legitimate to take it nonlocal $AS=S-K*S$, $K$ smooth, nonnegative, compactly supported with unit mass. The operator $A$ can also be the fractional Laplacian of order $\alpha\in(0,1)$
$$
(-\partial_{xx})^\alpha S(t,x)=c_\alpha\lim_{\varepsilon\to0}\int_{\vert h\vert\geq\varepsilon}\frac{S(t,x)-S(t,x+h)}{\vert h\vert^{1+2\alpha}}~\mathrm{d}h,$$ where $c_\alpha$ is a normalisation constant.

Such a model does not allow a simple computation of $S(t,x)$ in terms of $I$, ans destroys the structure (\ref{kppI}) on which we have based this work. As a consequence, questions that were trivial when (\ref{EDP}) reduces to (\ref{edpI})) become, all of a sudden, open. One of them is the 
bounds for the Cauchy Problem; it becomes an important open problem which, to our knowledge, is also open in the SIR model (that is, the removal and transmission coefficients are constant). Indeed, in such a setting, we fall in the context of reaction-diffusion systems with nonequal diffusion coefficients, which renders the maximum principle useless when it comes to estimating the size of the solutions. Existence of traveling waves is a question that needs to be understood almost entirely, not to talk about the spreading speed. Here,  upper and lower estimates, involving in a significant way the parameters of the model, on how the solution propagates would be an interesting achievement. The question  of the existence of a unique asymptotic spreading speed is, in this context, a mostly open question.
 
\subsubsection*{Changes of dimensionality}
We end this review of open problems by recalling that, as is widely  acknowledged, the spreading of epidemics  is best modelled as a spatial domain that undergoes changes of dimensionality, as towns and cities play a major role in the development of this type of phenomena. One model to which (\ref{EDP}) could be coupled is the  Besse-Faye \cite{BF} systems where the domain is a graph, whose vertices model roads or transportation structures, and whose edges model the cities. On the edges, a contamination process of the type (\ref{EDP}) (with no diffusion kernel) is occurring, while only transport phenomena occur on the vertices. Exchange conditions between edges and vertices are prescribed.
On such a structure, it is already interesting to study the steady states, in terms of the intensity of the exchanges between the edges and the vertices of the graph. It is also an interesting question to examine whether the introduction of the variable $i$ influences the qualitative or quantitative results. This is certainly a highly nontrivial question. The   dynamical properties come next, and represent a wide open arena of mathematical challenges. The case of a homogeneous tree, as studied in \cite{BF21}, is already an  issue of high interest.

\section*{Acknowledgements} 
The authors acknowledge support from the ANR via the project ReaCh under grant agreement ANR-23-CE40-0023-01. G.F. and M.Z. acknowledge support from the ANR via the project Indyana under grant agreement ANR- 21- CE40-0008. M.Z. acknowledges support by the Occitanie region, the European Regional Development Fund (ERDF), and the French government, through the France 2030 project managed by the National Research Agency (ANR)  ANR-22-EXES-0015. The work was also partially supported by the French National Institute of Mathematical Sciences and their Interactions (Insmi) via the platform MODCOV19 and by Labex CIMI under grant agreement ANR-11-LABX-0040.


\appendix

%
%
%

\end{document}